\documentclass[11pt]{amsart}

\usepackage{amsmath,amssymb,amsthm,amsfonts,amsrefs}
\usepackage[margin=1in,marginparwidth=0.8in, marginparsep=0.1in]{geometry}
\usepackage{graphicx}
\usepackage[all]{xy}

\usepackage{enumerate}
\usepackage{tikz}
\usepackage{tikz-cd}
    \usetikzlibrary{arrows}

\usepackage{xcolor,float}
\usepackage{mathrsfs}
\usepackage{caption}

\usepackage{verbatim}

\usepackage{mathtools}
\usepackage[makeroom]{cancel}
\usepackage{stmaryrd}
\usepackage[bottom]{footmisc}
\usepackage{enumitem}
\usepackage{todonotes}
\usepackage[hidelinks]{hyperref}
\usepackage[title]{appendix}
\usepackage{float} 

\usepackage{fancyhdr}
\usepackage{mathrsfs}
\newtheorem{definition}{Definition}[section]

\newtheorem{definition_theorem}[definition]{Definition-Theorem}

\theoremstyle{definition}

\newtheorem{remark}[definition]{Remark}

\theoremstyle{plain}
\newtheorem{theorem}[definition]{Theorem}
\newtheorem{lemma}[definition]{Lemma}
\newtheorem{proposition}[definition]{Proposition}
\newtheorem{corollary}[definition]{Corollary}

\renewcommand{\Re}{\operatorname{Re}}

\newcommand\abs[1]{\lvert #1 \rvert}

\makeatletter
\newcommand*\bigcdot{ {\mathpalette\bigcdot@{.5}} }
\newcommand*\bigcdot@[2]{\mathbin{\vcenter{\hbox{\scalebox{#2}{$\m@th#1\bullet$}}}}}
\makeatother

\newcommand\NN{ \mathbb{N} }

\newcommand\RR{ \mathbb{R} }

\newcommand\cD{ \mathcal{D} }

\newcommand\cF{ \mathcal{F} }

\newcommand\cS{ \mathcal{S} }
\newcommand\cT{ \mathcal{T} }

\newcommand\cV{ \mathcal{V} }

\newcommand\sC{ \mathscr{C} }
\newcommand\sD{ \mathscr{D} }

\newcommand\pt{\operatorname{\{\ast\} } }
\newcommand\pre{ {\operatorname{pre}} }
\newcommand\Rp{ { \mathbb{R}_{>0} } }

\newcommand\str{\operatorname{star}}

\newcommand\Fun{ {\operatorname{Fun}} }
\newcommand\id{\operatorname{id}}
\newcommand\Id{\operatorname{Id}}
\newcommand\Ind{ {\operatorname{Ind} } }
\newcommand\Hom{\operatorname{Hom}}

\newcommand\PrLcs{  {  \operatorname{Pr}^{\operatorname{L} }_{\omega,st} } }
\newcommand\PrLst{  {  \operatorname{Pr}^{\operatorname{L} }_{st} } }

\newcommand\PrRst{  {  \operatorname{Pr}^{\operatorname{R} }_{st} } }

\newcommand\clmi[1]{  \underset{ {#1} }{\operatorname{colim}} }

\newcommand\lmi[1]{  \underset{ {#1} }{\lim} }

\newcommand\dMod{\operatorname{-Mod}}
\newcommand\fib{\operatorname{fib}}

\newcommand\st{   {\operatorname{st}} }

\newcommand\msh{\operatorname{\mu sh}}
\newcommand\ms{\operatorname{SS}}
\newcommand\msif{\operatorname{SS}^{\infty}}

\newcommand\Op{\operatorname{Op}}

\newcommand\Sh{\operatorname{Sh}}

\newcommand\sHom{\mathscr{H}om}

\newcommand\supp{ {\operatorname{supp}} }
\newcommand\wsh{\operatorname{\mathfrak{w} sh}}

\newcommand\WF{ {\mathcal{W} } }
\newcommand\wrap{\mathfrak{W}}

\newcommand\Coh{\operatorname{Coh}}
\newcommand\IndCoh{\operatorname{IndCoh}}
\newcommand\Perf{\operatorname{Perf}}
\newcommand\QCoh{\operatorname{QCoh}}

\newcommand\PrLV[1][\cV]{  {  \operatorname{Pr}^{\operatorname{L} }_{\cV,st} } }

\newcommand\SD[1]{ {\operatorname{D_{#1}  }  } }

\newcommand\VD[1]{ {\operatorname{D}_{#1}  } }
\newcommand\ND[1]{ {\operatorname{D}_{#1}^\prime  } }
\newcommand\WD[1]{ {\operatorname{D}_{#1}^w  } }
\newcommand\dT{ {\dot{T} } }

\newcommand{\SC}{\mathcal{C}}

\newcommand{\bR}{\mathbb{R}}

\begin{document}

\title[Duality, K\"unneth formulae, and integral transforms in microlocal geometry]{\textbf{Duality and kernels in microlocal geometry}\\\vspace{3mm} {\textbf{\footnotesize{-- An approach by contact isotopies --}}}} 

\date{}
\author{Christopher Kuo}
\address{Department of Mathematics, University of Southern California}
\email{chrislpkuo@berkeley.edu} 
\author{Wenyuan Li}
\address{Department of Mathematics, Northwestern University.}
\email{wenyuanli2023@u.northwestern.edu}
\maketitle

\begin{abstract}
    We study the dualizability of sheaves on manifolds with isotropic singular supports $\Sh_\Lambda(M)$ and microsheaves with isotropic supports $\msh_\Lambda(\Lambda)$ and obtain a classification result of colimit-preserving functors by convolutions of sheaf kernels. Moreover, for sheaves with isotropic singular supports and compact supports $\Sh_\Lambda^b(M)_0$, the standard categorical duality and Verdier duality are related by the wrap-once functor, which is the inverse Serre functor in proper objects, and we thus show that the Verdier duality extends naturally to all compact objects $\Sh_\Lambda^c(M)_0$ when the wrap-once functor is an equivalence, for instance, when $\Lambda$ is a full Legendrian stop or a swappable Legendrian stop.
\end{abstract}

\tableofcontents

\section{Introduction}

\subsection{Context and background}
This paper is the second in the series of study, along with \cite{Kuo-Li-spherical, Kuo-Li-Calabi-Yau}, on the non-commutative geometric framework in the setting of microlocal sheaf theory. We are interested in the category of sheaves arising from the symplectic geometry structure on the Lagrangian skeleton of the pair $(T^*M, \Lambda)$, where $\Lambda \subseteq S^*M$ is a subanalytic Legendrian subset in the ideal contact boundary $S^*M$ of the exact symplectic manifold $T^*M$.
The focus of this paper are duality and bimodules, in the forms of integral kernels, for sheaves and microsheaves.

Let $\Bbbk$ be a field of characteristic $0$ and $X$ be a proper scheme and $\omega_X$ be the dualizing sheaf. 
Then the classical Serre duality asserts that the following functor is an equivalence
$$\operatorname{D}_X: \Coh(X) \to \Coh(X)^{op}, \; F \mapsto \sHom(F, \omega_X).$$
One modern interpretation of this equality is that the $(\infty,1)$-category $\IndCoh(X)$ is self-dual \cite{Gaitsgory-IndCoh} when viewed as an object in the symmetric monoidal category $\PrLst$ of presentable $(\infty,1)$-categories whose symmetric monoidal structure is defined by Lurie in \cite{Lurie-HA}. Moreover, the Serre duality also asserts that for $G \in \Coh(X)$ and $F \in \Perf(X)$, there is an equivalence
$$\Hom(G,F \otimes \omega_X) = \Hom(F,G)^\vee$$
where the latter is the linear dual. A modern interpretation of this equality is that under the self duality of $\IndCoh(X)$ by Serre duality \cite{Preygel,Gaitsgory-IndCoh,Gaitsgory-Rozenblyum} and $\QCoh(X)$ by the naive duality \cite{Ben-Zvi-Francis-Nadler,Gaitsgory-Rozenblyum}, the functor 
$$\Psi_X^\vee: \QCoh(X) \to \IndCoh(X), \, F \mapsto F \otimes \omega_X$$
is the dual of the inclusion functor $\Psi_X: \IndCoh(X) \to \QCoh(X)$ \cite{Gaitsgory-IndCoh}. In other words, the Serre duality and the naive duality are related by the Serre functor $\Psi_X^\vee = (-) \otimes \omega_X$.

Many other examples of dualities have been studied from this viewpoint. In the setting of $\cD$-modules, the category of $\cD$-modules on a reasonable quasi-compact stack $U$ is self dual. Moreover, when $X$ is a miraculous stack \cite{DrinfeldGaitsgory}, the naive duality and Verdier duality are related by the pseudo-identity functor which is the inverse Serre functor
$$\mathrm{PsId}_{X,!}: \cD\text{-Mod}(X) \to \cD\text{-Mod}(X).$$
In the setting of (topological) sheaf theory, one example that has been studied is sheaves on universal cotruncative quasi-compact open substacks $U \subset Bun_G(\Sigma)$ with nilpotent singular supports $\Sh_{\mathcal{N}ilp}(U)$ \cite{Arinkin-Gaitsgory-Kazhdan-Raskin-Rozenblyum-Varshavsky}. The naive duality and Verdier duality are also related by the pseudo-identity functor or the inverse Serre functor
$$\mathrm{PsId}_{U,!}: \Sh_{\mathcal{N}ilp}(U) \to \Sh_{\mathcal{N}ilp}(U),$$
which extends to a miraculous functor from the co-version of sheaves on $Bun_G(\Sigma)$ to sheaves on $Bun_G(\Sigma)$ with nilpotent singular supports.

In fact, this viewpoint exhibits a clean connection between duality and the Fourier--Mukai transformation, which states that all colimit-preserving functors between quasi-coherent sheaves are given geometrically by convolutions \cite{Toen-Morita-theory}.
This is first studied in the algebro-geometric setting by Mukai \cite{Mukai} (thus sometimes referred as Fourier--Mukai), 
and later by Orlov \cite{Orlov}, To\"en \cite{Toen-Morita-theory} and others \cite{Ben-Zvi-Francis-Nadler,Preygel}. In the setting of sheaf theory, the Fourier--Mukai transformation states that all colimit-preserving functors are given geometrically by convolutions, i.e., the assignment
\begin{align*}
\Sh(X \times Y) &\xrightarrow{\sim} \Fun^L(\Sh(X), \Sh(Y)) \\
K &\mapsto \left(F \mapsto K \circ F \coloneqq {\pi_2}_! ( K \otimes \pi_1^* F) \right)
\end{align*}
is an equivalence.

We show that similar phenomenon holds in the microlocal sheaf setting, following the approach of Ben-Zvi--Nadler--Francis \cite{Ben-Zvi-Francis-Nadler}, Preygel \cite{Preygel} and Gaitsgory--Rozenblyum \cite{Gaitsgory-IndCoh,Gaitsgory-Rozenblyum} in the derived algebraic geometric setting.
The relevant $(\infty, 1)$-categories will have the form $\Sh_{\Lambda}(M)$ of sheaves on $M$ microsupported in a singular isotropic subset $\Lambda \subseteq S^* M$ at infinity. 
Here, we say a subanalytic set $X \subseteq S^*M$ is isotropic if it can be stratified by isotropic submanifolds.
While the situation on manifolds seems to be much easier than the one on non-quasi-compact stacks \cite{Arinkin-Gaitsgory-Kazhdan-Raskin-Rozenblyum-Varshavsky}, we emphasize that as we are dealing with arbitrary real subanalytic isotropics, interesting phenomenon will happen.

Such sheaf categories are closely related to Fukaya categories \cite{Nadler-Zaslow,Nad}. By the main result of \cite{Ganatra-Pardon-Shende3}, the $(\infty,1)$-category of sheaves are topological models of the wrapped Fukaya category, after taking Ind-completion:
$$\Sh_\Lambda(M) = \Ind\,\WF(T^*M, -\Lambda).$$
Therefore, under homological mirror symmetry \cite{KonHMS,AurouxAnti}, the microlocal sheaves should be thought of as the mirror to coherent sheaves. We emphasize however that this paper is purely sheaf-theoretic. We will make remarks on the relation of our results with Floer theory at the end of the introduction.

\subsection{Results and corollaries}
We follow the higher categorical convention in this paper. That is, unless specified, a category will mean an $(\infty,1)$-category.
We will also work in the real analytic setting so all manifolds are assumed to be real analytic and subobjects such as stratifications or isotropics are assumed to be subanalytic. 

We will consider the category $\Sh_{\widehat\Lambda}(M)$ of sheaves microsupported on conic subanalytic isotropic subsets $\widehat\Lambda \subseteq T^*M$. Write $\Lambda \subseteq S^*M$ for the quotient of the complement of $\widehat\Lambda$ away from zero section $\widehat\Lambda \setminus M$ by the $\bR_{>0}$-action. When $\widehat\Lambda$ contains the zero section $M$, this is equivalent to $\Sh_\Lambda(M)$ of sheaves microsupported on $\Lambda \subseteq S^*M$ at infinity. When $\widehat\Lambda$ has compact intersection with the zero section $M$ that contains all the bounded strata of $M \setminus \pi(\Lambda)$, this is equivalent to $\Sh_\Lambda(M)_0$ of compactly supported sheaves microsupported on $\Lambda \subseteq S^*M$ at infinity. 

Our first result is Fourier--Mukai property of sheaves with isotropic singular supports. Here, we use the notation $\Fun^L(-, -)$ for the category of colimit preserving functors.

\begin{theorem}\label{thm: fmt}
Let $M$ and $N$ be real analytic manifolds and $\widehat\Lambda \subseteq T^* M$, $\widehat\Sigma \subseteq T^* N$
be closed conic subanalytic singular isotropics.
Then, 
duality induces an equivalence
$$ \Sh_{-{\widehat\Lambda} \times {\widehat\Sigma}}(M \times N) = \Fun^L(\Sh_{\widehat\Lambda}(M),\Sh_{\widehat\Sigma}(N))$$
which is given by $K \mapsto (H \mapsto K \circ H)$ for $H \in \Sh_{\smash{\widehat\Sigma}}(N)$.
\end{theorem} 

We will see in fact that the above theorem follows from the K\"unneth formula for sheaves with isotropic microsupports and the duality between $\Sh_{\widehat\Lambda}(M)$ and $\Sh_{-\widehat\Lambda}(M)$. We call this the standard duality, which is closer in relation with the naive duality of quasi-coherent sheaves \cite[Chapter II.3, Section 4.3.1]{Gaitsgory-Rozenblyum} (and the miraculous duality of automorphic sheaves \cite[Section 0.1.3]{Arinkin-Gaitsgory-Kazhdan-Raskin-Rozenblyum-Varshavsky}). We emphasize that the standard duality is not the Verdier duality.

\begin{theorem}[{The K\"unneth formula}]\label{pd:g}
Let $M$ and $N$ be real analytic manifolds and $\widehat\Lambda \subseteq T^* M$, $\widehat\Sigma \subseteq T^* N$
be closed conic subanalytic singular isotropics. 
Then there is an equivalence 
\begin{align*}
\Sh_{\widehat\Lambda}(M) \otimes \Sh_{\widehat\Sigma}(N) &= \Sh_{\widehat\Lambda \times \widehat\Sigma}(M \times N) \\
(F,G) &\mapsto F \boxtimes G.
\end{align*} 
\end{theorem}

\begin{definition_theorem} \label{def-thm: canonical-duality}
Denote by $\Delta: M \hookrightarrow M \times M$ the diagonal, $p: M \rightarrow \{*\}$ the projection,
and $\iota_{-\widehat\Lambda \times \widehat\Lambda}^*: \Sh(M \times M) \rightarrow \Sh_{-\widehat\Lambda \times \widehat\Lambda}(M \times M)$ the left adjoint of the inclusion $\Sh_{-\widehat\Lambda \times \widehat\Lambda}(M \times M) \subset \Sh(M \times M)$.
Then the triple $(\Sh_{-\widehat\Lambda}(M),\epsilon,\eta)$ where
\begin{equation}
\begin{split}
\epsilon &= p_! \Delta^* : \Sh_{-\widehat\Lambda \times \widehat\Lambda}(M \times M) \rightarrow \cV \\
\eta &= \iota_{\widehat\Lambda \times -\widehat\Lambda}^* \Delta_* p^*
: \cV \rightarrow \Sh_{\widehat\Lambda \times -\widehat\Lambda}(M \times M)
\end{split}
\end{equation}
exhibits $\Sh_{-\widehat\Lambda}(M)$ as a dual of $\Sh_{\widehat\Lambda}(M)$.
As a consequence, there is an identification $\Sh_{-\widehat\Lambda}(M) = \Sh_{\widehat\Lambda}(M)^\vee$ 
and we call the induced duality $\SD{\widehat\Lambda}: \Sh_{-\widehat\Lambda}^c(M)^{op} \xrightarrow{\sim} \Sh_{\widehat\Lambda}^c(M)$
as the \textit{standard duality}. 
\end{definition_theorem}

The proof of Theorem \ref{pd:g} will be the focus of Section \ref{sec: Kunneth-sheaf} and the proof of Definition-Theorem \ref{def-thm: canonical-duality} will be the focus of Section \ref{sec: dual}. We also show in Section \ref{sec: geometric-dual} in Theorem \ref{thm: standard-dual-by-wrappings} that this standard dual admits a geometrical construction using wrapped sheaves, defined in \cite{Kuo-wrapped-sheaves}, when $\hat{\Lambda}$ contains the zero section. 
We point out that using the doubling construction,
we are able to deduce a K\"unneth formula and Fourier-Mukai property for microsheaves supported on singular isotropic subsets.
See Section \ref{sec: Kunneth-microsheaf} for details.
  
\begin{theorem} \label{thm: Kunneth-Fourier-Mukai-microsheaves}
Let $\Lambda \subseteq S^* M$ and $\Sigma \subseteq S^* N$ be compact isotropics. Then there are equivalences:
\begin{gather*}
\msh_\Lambda(\Lambda) \otimes \msh_\Sigma(\Sigma) = \msh_{\Lambda \times \Sigma \times \bR}(\Lambda \times \Sigma),\\
\Fun^L\left(\msh_\Lambda(\Lambda), \msh_\Sigma(\Sigma) \right) = \msh_{-\Lambda \times \Sigma \times \bR}(-\Lambda \times \Sigma).
\end{gather*}
\end{theorem}
\begin{remark}
For $\Lambda \subseteq S^* M$ and $\Sigma \subseteq S^* N$, $\Lambda \times \Sigma$ is not a Legendrian in $S^*(M \times N)$. Therefore, in the K\"unneth formula, we need to take microsheaves supported on the thickened subset $\Lambda \times \Sigma \times \bR \subseteq S^*M \times S^*N \times \bR \subseteq S^*(M \times N)$ and consider sections on $\Lambda \times \Sigma := \Lambda \times \Sigma \times 0$. See also \cite{Shende-h-principle,Nadler-Shende}.
\end{remark}

We also hope that the notion of duality from the categorical viewpoint will help us understand the classical Verdier duality 
\begin{align*}
\VD{M}: \Sh_{-\widehat\Lambda}^b(M)^{op} \xrightarrow{\sim} \Sh_{\widehat\Lambda}(M)^b, \;\;\;
F \mapsto \sHom(F, \omega_M)
\end{align*}
where $\omega_M$ is the dualizing sheaf of $M$ and $\Sh_{\widehat\Lambda}^b(M)$ the subcategory of $\Sh_{\smash{\widehat\Lambda}}(M)$ consisting of sheaves with perfect stalks, which is contained in the subcategory consisting of compact objects $\Sh_{\widehat\Lambda}^c(M)$.

When $\widehat\Lambda$ has compact intersection with the zero section, on $\Sh_{-\widehat\Lambda}^b(M)^{op}$ the Verdier duality $\VD{M}$
is given by $S_{\widehat\Lambda}^- \circ \SD{\widehat\Lambda} (-) \otimes \omega_M$ where $S_{\widehat\Lambda}^-$ is the negative wrap-once functor. 
When $S_{\widehat\Lambda}^-(-) \otimes \omega_M$ is invertible, it restricts to the Serre functor on $\Sh_{\widehat\Lambda}^b(M)$ by our first paper of the series \cite{Kuo-Li-spherical}. In this case, the Verdier duality $\VD{M}$ can be extended to an equivalence 
$\Sh_{-\widehat\Lambda}^c(M)^{op} \xrightarrow{\sim} \Sh_{\widehat\Lambda}^c(M)$,
which, by taking Ind-completion, provides another duality triple. We show that the converse is also true. 
We call this the Verdier duality, which is analogous to the Serre duality on ind-coherent sheaves \cite[Chapter II.3, Section 4.4.2]{Gaitsgory-Rozenblyum} (and Verdier duality on automorphic sheaves \cite[Section 0.2.1]{Arinkin-Gaitsgory-Kazhdan-Raskin-Rozenblyum-Varshavsky}). 

\begin{theorem}\label{converse-statement-Verdier}
Let $M$ be a connected manifold, $\widehat\Lambda \subseteq T^*M$ a subanalytic singular isotropic such that $\widehat\Lambda \cap M$ is compact, and denote by $\epsilon^V$ the colimit-preserving functor 
$$p_* \Delta^!: \Sh_{-\widehat\Lambda \times \widehat\Lambda}(M \times M) \rightarrow \cV.$$
There exists an object $\eta^V$, which we identify as a colimit-preserving functor 
$$\eta^V: \cV \rightarrow \Sh_{\widehat\Lambda \times -\widehat\Lambda}(M \times M),$$
such that the triple $(\Sh_{-\widehat\Lambda}(M), \epsilon^V,\eta^V)$ provides a duality data for $\Sh_{\widehat\Lambda}(M)$ in the sense of Definition \ref{smdual} in $\PrLst$ if and only if the functor $S_{\widehat\Lambda}^-$ or equivalently the left adjoint $S_{\widehat\Lambda}^+$ is invertible and the induced duality on $\Sh_{-\widehat\Lambda}^c(M)^{op} \xrightarrow{\sim} \Sh_{\widehat\Lambda}^c(M)$ restricts to the Verdier duality $\VD{M}$ on $\Sh_{\widehat\Lambda}^b(M)$. 
\end{theorem}

Therefore, just like the algebraic setting, the Verdier duality on microlocal sheaves (which is analogous to the Serre duality on coherent sheaves), when well defined on the category of all compact objects, is related to the standard duality (which is analogous to the naive duality on quasi-coherent sheaves), by the inverse Serre functor on proper objects.

As shown in our first paper of the series \cite[Section 7]{Kuo-Li-spherical}, the wrap-once functor is not always an equivalence. Therefore, we can conclude that the Verdier duality cannot always be extended to a categorical duality. However, we also gave sufficient conditions for the wrap-once to be an equivalence \cite{Kuo-Li-spherical}, in which case the Verdier duality can be extended to a categorical duality:

\begin{corollary}\label{cor:verdier-swappable}
Let $M$ be a closed manifold, $\widehat\Lambda$ be a subanalytic conic isotropic subset and $\Lambda \subseteq S^*M$ is a full Legendrian stop or swappable Legendrian stop. Then the triple $(\Sh_{-\widehat\Lambda}(M), \epsilon^V,\eta^V)$ provides a duality data for $\Sh_{\widehat\Lambda}(M)$ in the sense of Definition \ref{smdual} in $\PrLst$.
\end{corollary}

The above result provides a categorical approach to recover the Serre functor on certain categories of topological sheaves, for example, sheaves on the flag variety that are constructible with respect to Schubert stratification \cite{TiltingExercise}.

The above result is also connected to derived algebraic geometry in the sense of homological mirror symmetry. For toric mirror symmetry (coherent-constructible correspondence), one can show the following relation between the dualities on both sides, which essentially follows of \cite[Remark 12.11 \& 12.12]{KuwaCCC}.

\begin{corollary}
Consider toric stacks $X_\Sigma$ and the mirror Lagrangian skeleton $\Lambda_\Sigma \subseteq T^*T^n$. Under Kuwagaki's mirror functor \cite{KuwaCCC}
$$K_\Sigma: \IndCoh(X_\Sigma) \xrightarrow{\sim} \Sh_{\Lambda_\Sigma}(T^n),$$
the Serre duality on $\IndCoh(X_\Sigma)$ intertwines with Verdier duality on $\Sh_{\Lambda_\Sigma}(T^n)$ in Corollary \ref{cor:verdier-swappable}.
\end{corollary}
\begin{remark}
Note that Kuwagaki's mirror functor $K_\Sigma$ and Fang--Liu--Treumann-Zaslow's mirror functor $\kappa_\Sigma$ are related by the Serre functor $K_\Sigma(- \otimes \omega_{X_\Sigma}) = \kappa_\Sigma(-)$ \cite[Remark 9.3]{KuwaCCC}, which is why under Fang--Liu--Treumann-Zaslow's mirror functor, the naive duality on perfect complexes intertwines with the Verdier dual on constructible sheaves \cite[Proposition 7.3]{FLTZCCC}.
\end{remark}

Finally, we briefly explain the implications of our Fourier--Mukai results for Fukaya categories. The K\"unneth formula is known for wrapped Fukaya categories. Indeed, Ganatra--Pardon--Shende \cite{Ganatra-Pardon-Shende2} and Gao \cite{Gao1} showed that
$$\Perf\WF(X, \Lambda) \otimes \Perf\WF(Y, \Sigma) = \Perf\WF(X \times Y, \mathfrak{c}_X \times \Sigma \cup_{\Lambda \times \Sigma \times \bR} \Lambda \times \mathfrak{c}_Y).$$
Then by abstract categorical arguments similar to Section \ref{sec: dual} and the observation that $\WF(X, \Lambda)^{op} = \WF(X^-, \Lambda)$ (where $X^-$ is the manifold $X$ with the negative symplectic form), it follows that
\begin{align*}
\Fun^{ex}(\WF(X, \Lambda), \Ind\,\WF(Y, \Sigma)) &= \Fun^L(\Ind\,\WF(X, \Lambda), \Ind\,\WF(Y, \Sigma)) \\
&= \Ind\,\WF(X^- \times Y, \mathfrak{c}_X \times \Sigma \cup_{\Lambda \times \Sigma \times \bR} \Lambda \times \mathfrak{c}_Y).
\end{align*}
Here, $\Fun^{ex}(-, -)$ means the category of exact functors. Our result provides a sheaf theoretic proof of the result when $X$ and $Y$ are cotangent bundles or Weinstein hypersurfaces in cotangent bundles.



\subsection*{Acknowledgement} We would like to thank Mohammed Abouzaid, Pramod Achar, Shaoyun Bai, Roger Casals, Laurent C\^ot\'e, Sheel Ganatra, Yuichi Ike, Emmy Murphy, Nick Rozenblyum, Germ\'an Stefanich, Vivek Shende, Pyongwon Suh, Alex Takeda, Dima Tamarkin, Harold Williams, and Eric Zaslow  
for helpful discussions. In particular, the new discussion regarding the relation between the notion of duality and wrapping owes its existence from the discussion with Harold Williams, with special instances observed in mirror symmetry.
CK was partially supported by NSF CAREER DMS-1654545, VILLUM FONDEN grant 37814, and also NSF grant DMS-1928930 when in residence at the SLMath during Spring 2024.


\section{Microlocal sheaf theory}

\subsection{Microsupport of sheaves}

Let $\cV$ be a compactly generated rigid symmetric monoidal category. Let $M$ be a smooth manifold and $\Sh(M)$ be the category of sheaves with coefficients in $\cV$. Following Kashiwara--Schapira \cite{KS} and Robalo--Schapira \cite{Robalo-Schapira}, for a sheaf $F \in \Sh(M)$, one can define a conic closed subset in the cotangent bundle $\ms(F) \subseteq T^*M$ called the singular support of $F$ and the corresponding closed subset in the cosphere bundle $\msif(F) \subseteq S^*M$ called the singular support at infinity of $F$.

For $\widehat X \subseteq T^*M$, we define $\Sh_{\widehat X}(M)$ to be the full subcategory of sheaves $F$ such that $\ms(F) \subseteq \widehat X$. For $X \subseteq S^*M$, we define $\Sh_X(M)$ to be the full subcategory of sheaves $F$ such that $\msif(F) \subseteq X$. 
The inclusion functor
$$\iota_{\widehat X*}: \Sh_{\widehat X}(M) \hookrightarrow \Sh(M)$$
is limit and colimit preserving by \cite[Proposition 3.4]{Guillermou-Viterbo}, and thus admits both left and right adjoint, which we denote by 
$\iota_{\widehat X}^*$ and $\iota_{\widehat X}^!$. In particular, the left adjoint functor is also colimit perserving.

Then we recall the definition of microsheaves following \cite{Gui,NadWrapped,Nadler-Shende,Kuo-Li-spherical}. First, define the presheaf
\begin{align*}
\msh^\pre: (\Op_{T^* M}^{\Rp})^{op} &\longrightarrow \st \\
\Omega &\longmapsto \Sh(M)/ \Sh_{\Omega^c}(M)
\end{align*}
where we restrict our attention to conic open sets $\Op_{T^* M}^{\Rp}$,
and the target $\st$ is the (large) category of stable categories with morphisms being exact functors.
We denote by $\msh$ its sheafification and refer it as the sheaf of microsheaves.
Note that $\msh |_{0_M} = \Sh$ as sheaves of categories on $M$.

We note that since $\msh$ is conic, $\msh |_{\dT^* M}$ descends naturally to a sheaf on $S^* M$, 
and we abuse the notation, denoting it by $\msh$ as well.

\begin{definition}
Fix a subanalytic isotropic subset $X \subseteq S^*M$. 
Let $\msh_X$ denote the subsheaf of $\msh$ which consists of objects microsupported in $X$ or $X \times \RR_{>0}$.
\end{definition}

We note that this sheaf coincides with the sheafification of the following 
subpresheaf $\msh^\pre_\Lambda$ of $\msh^\pre$ (where $\ms_\Omega(F) := \ms(F) \cap \Omega$):

\begin{align*}
\msh^\pre_\Lambda: (\Op_{T^* M}^{\Rp})^{op} &\longrightarrow \st \\
\Omega &\longmapsto \{F \in \msh^\pre(\Omega) \mid \ms_\Omega(F) \subseteq \Lambda \}
\end{align*}
Note that $\msh_X$ is a sheaf on $S^*M$ or $\dT^*M$ supported on $X$ or $X \times \bR_{>0}$,
and we will use the same notation $\msh_X$ to denote the corresponding sheaf on $X$ or ${X} \times \bR_{>0}$.

Since $\msh_X$ forms a sheaf, for open subsets $\Omega \subseteq \Omega^\prime$, there are natural restriction maps $\msh_X(\Omega^\prime) \rightarrow \msh_X(\Omega)$. 
In particular, we will refer to the restriction map associated to $\dT^* M \subseteq T^* M$ as
the microlocalization functor along $X$
$$m_X: \Sh_{\widehat X}(M) \rightarrow \msh_X(X).$$
Later, in Section \ref{ims}, we will see that the natural restriction functors admit both the left and the right adjoints when $X$ are isotropic subsets.

\subsection{Constructible sheaves}

Under some mild regularity assumptions, having an isotropic microsupport implies that the sheaf is constructible.

Recall that a \textit{stratification} $\cS$ of $X$ is a decomposition of $X$ into to a disjoint union of locally closed subset 
$\{ X_s \}_{ s \in \cS}$. 
In this paper, we work with stratifications which are locally finite, consist of subanalytic submanifolds, and satisfies the \textit{frontier
condition} that $\overline{X_s} \setminus X_s$ is a disjoint union of strata in $\cS$.
In this case, there is an ordering which is defined by $s \leq  t$ if and only if $X_t \subseteq \overline{X_s}$.
We also use $\str(s)$ to denote $\coprod_{t \leq s} X_t$, which is the smallest open set built out of the strata that
contain $s$, and we note that $s \leq  t$ if and only if $\str(s) \subseteq \str(t)$.

\begin{definition}
For a given stratification $\cS$,
a sheaf $F$ is said to be $\cS$-constructible if $F|_{X_s}$ is a local system for all
$s \in \cS$.
We denote the subcategory of $\Sh(X)$ consisting of such sheaves by $\Sh_{\cS}(X)$.
A sheaf $F$ is said to be constructible if $F$ is $\cS$-constructible for some stratification $\cS$.
\end{definition}
\begin{remark}
We do not impose any finiteness condition on the stalks of $F$. What we call constructible sheaves here corresponds to what Kashiwara--Schapira call weakly constructible sheaves \cite[Chapter 8]{KS}.
\end{remark}

We use $\cS \dMod$ to denote $\Fun(\cS^{op},\cV)$
and note that there is a canonical functor
\begin{align*}
\cS \dMod &\rightarrow \Sh_{\cS}(X) \\
1_s &\mapsto 1_{X_s}
\end{align*}
where $1_s \in \cS \dMod$ is the index functor representing $s$.
The following lemma provides a criterion when this functor is an isomorphism:

\begin{lemma}[{\cite[Lemma 4.2]{Ganatra-Pardon-Shende3}}]\label{lem:sheaves_by_representations}
Let $\Pi$ be a poset with a map to $Op_M$, and let $\cV[\Pi]$ denote its stabilization.
The following are equivalent
\begin{itemize}
\item $\Gamma(U;1_\cV) = 1_\cV$ for $U \in \Pi$ and $\Gamma(U;1_\cV) \xrightarrow{\sim} \Gamma(U \setminus V;1_\cV)$
whenever $U \not \subseteq V$.
\item The composition $\cV[\Pi] \rightarrow \cV[Op_M] \rightarrow \Sh(M)$ is fully faithful 
where the second map is given by the $!$-pushforward.
\end{itemize} 
\end{lemma}

Since simplices are contractible, the above lemma implies the following proposition from the same paper.
\begin{proposition}[{\cite[Lemma 4.7]{Ganatra-Pardon-Shende3}}]
Let $\cS$ be triangulation of $M$. Then $\Sh_\cS(M) = \cS \dMod$.
\end{proposition}

Recall a stratification is called a \textit{triangulation} if $X = \abs{K}$ is a realization of some simplicial complex $K$
and $\cS \coloneqq \{ \abs{\sigma} \mid \sigma \in K \}$ is given by the simplexes of $K$.
Since simplexes are contractible, the conditions in the above lemma are satisfied by triangulations.
Let $N^*_\infty (X_s)$ be the conormal bundle of the locally closed submanifold $X_s$
We use the notation $N^* \cS \coloneqq \cup_{s \in \cS} N^* (X_s)$ and call it the conormal of the stratification.
In general, $\Sh_{N^* \cS}(M)$ and $\Sh_\cS(M)$ can be different \cite[Example 2.52]{Kuo-wrapped-sheaves}.
Nevertheless, they coincide when the stratification is Whitney:

\begin{definition}
We say a stratification $\cS = \{X_s\}$ is Whitney if for any $X_s \subseteq \overline{X_t}$,
any sequence $x_n \in X_t$ and $y_n \in X_s$ both converging to $x$,
if the sequence of lines $\overleftrightarrow{x_n y_n}$ converges to $l$ and the sequence $T_{x_n} X_t$ converges to $\tau$,
then $\tau \supseteq l$.
\end{definition}

\begin{proposition}[{\cite[Proposition 8.4.1]{KS}, \cite[Proposition 4.8]{Ganatra-Pardon-Shende3}}]
For a Whitney stratification $\cS$ of a $C^1$ manifold $M$, we have $\Sh_{\cS}(M) = 
\Sh_{N^* \cS}(M)$ (i.e.~having microsupport contained in $N^*\cS$ is equivalent to being
$\cS$-constructible).
\end{proposition}

Combining with the comment on triangulations, we obtain a simple description of sheaves microsupported in $N^*_\infty \cS$ 
for some $C^1$ Whitney triangulation $\cS$.

\begin{proposition}[{\cite[Proposition 4.19]{Ganatra-Pardon-Shende3}}]\label{mc=cc}
Let $\cS$ be a $C^1$ Whitney triangulation. 
Then there is an equivalence 
\begin{align*}
\Sh_{N^* \cS}(M) &= \cS \dMod \\
1_{X_s} &\leftrightarrow 1_s
\end{align*}
where $1_s$ is the indicator which is defined by 
$$1_s(t) =
\begin{cases}
1, &t \leq s. \\
0, &\text{otherwise}.
\end{cases}
$$
In particular, the category $\Sh_{N^* \cS}(M)$ is compactly generated and its
compact objects $\Sh_{N^* \cS}^c(M)$ are given by sheaves with compact support and perfect stalks $\Sh_{N^* \cS}^b(M)_0$.
\end{proposition}

\subsection{Isotropic microsupport}\label{ims}
We say a subset $\widehat\Lambda \subseteq T^* M$ is isotropic if it can be stratified by isotropic submanifolds.
A standard class of isotropic subsets are given by the conormal $N^* \cS$ of a stratification $\cS$ which we study 
in the last section.
Assume $M$ is real analytic and we recall that a general isotropic subset which satisfies a decent regularity condition
are bounded by isotropics of this form.

\begin{definition}
A subset $Z$ of $M$ is said to be subanalytic at $x$ if there exists an open set $U \ni x$, 
compact manifolds $Y_j^i$ $(i = 1, 2, 1 \leq j \leq N)$ and analytic morphisms
$f_j^i: Y_j^i \rightarrow M$ such that 
$$Z \cap U = U \cap \bigcup_{j=1}^N (f_j^1(Y_j^1) \setminus f_j^2(Y_j^2)).$$
We say $Z$ is subanalytic if $Z$ is subanalytic at $x$ for all $x \in M$.
\end{definition}

\begin{lemma}[{\cite[Corollary 8.3.22]{KS}}]
Let $\widehat\Lambda$ be a closed subanalytic conic isotropic subset of $T^* M$.
Then there exists a $C^\omega$ Whitney stratification $\cS$ such that $\widehat\Lambda \subseteq N^* \cS$.
\end{lemma}

Combining with the above lemma, we obtain a microlocal criterion for a sheaf $F$ with
subanalytic microsupport being constructible:

\begin{proposition}[{\cite[Theorem 8.4.2]{KS}}]
Let $F \in \Sh(M)$ and assume $\ms(F)$ is subanalytic. 
Then $F$ is constructible if and only if $\ms(F)$ is a singular isotropic. 
\end{proposition}

Another feature of subanalytic geometry is that relatively compact subanalytic sets form an o-minimal structure.
Thus, one can apply the result of \cite{Czapla} to refine a $C^p$ Whitney stratification to a Whitney triangulation, for $1 \leq p < \infty$.

\begin{lemma}\label{cbs}
Let $\widehat\Lambda$ be a closed subanalytic conic isotropic subset in $T^* M$.
Then there exists a $C^1$-Whitney triangulation $\cS$ such that $\widehat\Lambda \subseteq N^* \cS$.
\end{lemma}

Combining the above two results, we conclude:
\begin{theorem}\label{extri}
Let $F \in \Sh(M)$ and assume $\ms(F)$ is a subanalytic singular isotropic.
Then $F$ is $\cS$-constructible for some $C^1$-Whitney triangulation $\cS$. 
\end{theorem}

Collectively, sheaves with the same subanalytic isotropic microsupport form a category with nice finiteness properties.
Let $\widehat\Lambda$ be a subanalytic conic isotropic in $T^* M$.
By picking a Whitney triangulation $\cS$ such that $\widehat\Lambda \subseteq N^* \cS$.
The fact that the inclusion $\Sh_{\widehat\Lambda}(M) \subseteq \Sh_{N^* \cS}(M) = \cS \dMod $
preserves both limits and colimits implies the following finiteness conditions:

\begin{proposition}[{\cite[Corollary 4.21]{Ganatra-Pardon-Shende3}}]\label{prop:stopremoval}
Let $\widehat\Lambda \subseteq T^*M$ be a subanalytic conic isotropic subset.
Then $\Sh_{\widehat\Lambda}(M)$ is compactly generated.
If $\widehat\Lambda \subseteq \widehat\Lambda^\prime$ is an inclusion of subanalytic singular isotropics,
then the left adjoint of $\Sh_{\widehat\Lambda}(M) \hookrightarrow \Sh_{\widehat\Lambda^\prime}(M)$, i.e.,
$\Sh_{\widehat\Lambda^\prime}(M) \twoheadrightarrow \Sh_{\widehat\Lambda}(M)$ preserves compact objects.
\end{proposition}

One can describe the fiber of $\Sh_{\widehat\Lambda^\prime}(M) \twoheadrightarrow \Sh_{\widehat\Lambda}(M)$ as follows. Let $(x,\xi) \in \widehat\Lambda$ be a smooth Lagrangian point. 
Up to a shift, there is a microstalk functor $\mu_{(x,\xi)}:\Sh_{\widehat\Lambda}(M) \rightarrow \cV$,
which admits descriptions by sub-level sets of functions whose differential is transverse to $\widehat\Lambda$ 
\cite[Proposition 7.5.3]{KS} \cite[Theorem 4.11]{Ganatra-Pardon-Shende3}. 
For a sheaf $F$ with $\ms(F) \subseteq \widehat\Lambda^\prime$, we have $\ms(F) \subseteq \widehat\Lambda$ if and only if $\mu_{(x,\xi)}(F) = 0$ for any smooth Lagrangian point $(x, \xi)$ in the complement. 
Since $\Sh_{\widehat\Lambda}(M) \hookrightarrow \Sh_{\widehat\Lambda^\prime}(M)$ admits left and right adjoints, one can conclude that $\mu_{(x,\xi)}$ also admits left and right adjoints.

By applying the left adjoint to the generator $1_\cV \in \cV$, we see that it is
tautologically corepresented by a compact object $\mu_{(x,\xi)}^l (1_\cV) \in \Sh_{\widehat\Lambda}^c(M)$.
Furthermore, when $\widehat\Lambda \subseteq \widehat\Lambda^\prime$ and $(x,\xi) \in \widehat\Lambda^\prime$,
the corepresentative in $\Sh_{\widehat\Lambda^\prime}^c(M)$ 
is sent under $\Sh_{\widehat\Lambda^\prime}(M) \twoheadrightarrow \Sh_{\widehat\Lambda}(M)$ to a similar corepresentative 
in $\Sh_{\widehat\Lambda}^c(M)$ and,
they are tautologically sent to the zero object when  $(x,\xi)$ is a smooth point in the complement.
The converse is also true:

\begin{proposition}[{\cite[Theorem 4.14]{Ganatra-Pardon-Shende3}}]\label{wdstopr}
Let $\widehat\Lambda \subseteq \widehat\Lambda^\prime \subseteq T^*M$ be closed subanalytic conic isotropics and let
${\sD^\mu}({\widehat\Lambda^\prime,\widehat\Lambda})$ denote the fiber of the canonical left adjoint functor
$\Sh_{\widehat\Lambda^\prime}(M) \twoheadrightarrow \Sh_{\widehat\Lambda}(M)$.
Then ${\sD^\mu}({\widehat\Lambda^\prime,\widehat\Lambda})$ 
is compactly generated by the corepresentatives of the microstalk functors $\mu_{(x,\xi)}$ for smooth Lagrangian points 
$(x,\xi) \in \widehat\Lambda^\prime \setminus \widehat\Lambda$.
\end{proposition}
\begin{remark}\label{rem:stalk-corepresentative}
As explained in \cite[Section 4.4]{Ganatra-Pardon-Shende3}, when $\xi = 0$, the microstalk functor $\mu_{(x, \xi)}$ is simply the stalk functor, and the above result also applies.
\end{remark}

When $\Lambda$ is a subanalytic isotropic subset, the above results plus the microlocal cut-off lemma \cite[Proposition 6.1.3]{KS} then implies that $\msh^\pre_\Lambda$ in fact takes value in the category of compactly generated stable categories $\PrLcs$,
whose morphisms are given by functors which admit both the left and the right adjoints;
therefore, its sheafification in $\st$ coincides with the sheafification in $\PrLcs$ \cite{Kuo-Li-spherical}.
In other words, for $\Omega \subseteq \Omega^\prime$, the restriction maps $\msh_\Lambda(\Omega^\prime) \rightarrow \msh_\Lambda(\Omega)$ admit both left and right adjoints. 
In particular, for the microlocalization functor along $\Lambda$
$$m_\Lambda: \Sh_{\widehat\Lambda}(M) \rightarrow \msh_\Lambda(\Lambda),$$
we denote its left and right adjoint by $m_\Lambda^l$ and $m_\Lambda^r$. These left adjoints preserve compact objects. Furthermore, in this case, $\msh_\Lambda$ defines a constructible sheaf on $\Lambda$ by using the microlocal cut-off lemma again \cite[Proposition 6.1.3]{KS}.

Consider the microlocalization functor
$$m_{\widehat\Omega}: \Sh_{\widehat\Lambda^\prime}(M) \to \msh_{\widehat\Lambda^\prime}(\widehat\Omega),$$
where $\widehat\Omega = \widehat\Lambda^\prime \setminus \widehat\Lambda$. When $\widehat\Omega$ is a small neighbourhood at a smooth Lagrangian point, by constructibility of $\msh_\Lambda$, this is just the microstalk functor. For a sheaf $F$ with $\ms(F) \subseteq \widehat\Lambda^\prime$, we know $\ms(F) \subseteq \widehat\Lambda$ if and only if $m_{\widehat\Omega}(F) = 0$. This leads to the following proposition:

\begin{proposition}\label{prop:stop-removal-microlocalization}
Let $\widehat\Lambda \subseteq \widehat\Lambda^\prime \subseteq T^*M$ be subanalytic conic isotropic subsets and let $\widehat\Omega = \widehat\Lambda^\prime \setminus \widehat\Lambda$. Then
the fiber ${\sD^\mu}({\widehat\Lambda^\prime,\widehat\Lambda})$ of the canonical left adjoint functor
$\Sh_{\widehat\Lambda^\prime}(M) \twoheadrightarrow \Sh_{\widehat\Lambda}(M)$
is the essential image of the left adjoint of microlocalziation functor $m_{\widehat\Omega}^l: \msh_{\widehat\Lambda^\prime}(\widehat\Omega) \to \Sh_{\widehat\Lambda^\prime}(M)$.
\end{proposition}

\begin{remark}\label{rem:polarization}
When $\Lambda \subseteq S^* M$ is a high codimensional isotropic subset, microsheaves supported on $\Lambda$ would be zero. However, we thicken the isotropic subset by a Lagrangian section $U_\Lambda$ in the normal bundle and take microsheaves supported on $U_\Lambda$. This is equivalent to choosing a (stable) polarization. We will use this viewpoint in Section \ref{sec: Kunneth-microsheaf}. See also \cite{Shende-h-principle,Nadler-Shende}.
\end{remark}

\section{The K\"unneth formulae}

We prove the K\"unneth formula for sheaves and microsheaves, Theorem \ref{pd:g} and (1) of Theorem \ref{thm: Kunneth-Fourier-Mukai-microsheaves} in this section.
We begin with recalling general facts about products of compactly generated categories and basic properties of constructible sheaves which will be needed in the proofs.
We refer to \cite{Kuo-wrapped-sheaves, Kuo-Li-spherical} for more detailed reviews of microlocal sheaf theory.
 

\subsection{Tensor products of categories}

Denote by $\PrLst$ the (large) category of presentable stable categories whose morphisms are given by colimit-preserving functors. 
Compactly generated categories in $\PrLst$ form a subcategory $\PrLcs$, and it is equivalent to $\st_\omega$, the category of idempotent complete small stable categories whose morphisms are given by exact functors by taking compact objects,
\begin{align*}
\PrLcs &\xrightarrow{\sim} \st_\omega \\
\sC &\mapsto C\coloneqq \sC^c.
\end{align*} 
The inverse map of this identification is given by taking Ind-completion $C \mapsto \Ind (C)$.
There is a symmetric monoidal structure $\otimes$ on $\PrLst$ \cite{Lurie-HA, Hoyois-Scherotzke-Sibilla}, and the following lemma implies that it restricts to a symmetric monoidal structure on $\PrLcs$, which further induces a symmetric monoidal structure $\otimes$ on $\st$ by sending $(C, D)$ to $\left( \Ind(C) \otimes \Ind(D) \right)^c$.

\begin{lemma}[{\cite[Proposition 7.4.2]{Gaitsgory-Rozenblyum}}]\label{pgc,PrLst}
Assume that $\sC$ and $\sD$ are compactly generated stable categories over $\cV$.
\begin{enumerate}
\item The tensor product $\sC \otimes \sD$ is compactly generated by objects of the form
$c_0 \otimes d_0$ with $c_0 \in C$ and $d_0 \in D$.
\item For $c_0$, $d_0$ as above, and $c \in \sC$, $d \in \sD$, we have a canonical isomorphism
$$ \Hom_\sC(c_0,c) \otimes \Hom_\sD(d_0,d) = \Hom_{\sC \otimes \sD}(c_0 \otimes d_0, c \otimes d).$$
\end{enumerate}
\end{lemma}

We will need this lemma concerning the full-faithfulness of the tensors of functors.
\begin{lemma}\label{ff;pr}
If the functors  $f_i: \sC_i \rightarrow \sD_i$ for $i = 1, 2$ in $\PrLst$ are fully faithful, 
then their tensor product  $f_1 \otimes f_2: \sC_1 \otimes \sC_2 \rightarrow \sD_1 \otimes \sD_2$
is fully faithful if one of the following condition is satisfied:
\begin{enumerate}
\item The functor $f_i$ admits a left adjoint.
\item The right adjoint of $f_i$ is colimit-preserving.
\end{enumerate}
\end{lemma}

\begin{proof}
We prove (1) and leave (2) to the reader.
We first note that since $f_1 \otimes f_2 = (\id_{\sD_1} \otimes f_2) \circ (f_1 \otimes \id_{\sC_2})$. 
It is sufficient to prove the case when $f_2 = \id_{\sC_2}$.
Denote by $f_1^l: \sD_1 \rightarrow \sC_1$ the left adjoint of $f_1$.
We note that since for any $Y \in \sC_1$,
$$ \Hom(f_1^l f_1 X, Y) = \Hom(f_1 X, f_1 Y) = \Hom(X, Y),$$
the left adjoint $f_1^l$ is surjective.
Now we notice that being surjective and being a left adjoint are both preserved under $(-) \otimes \id_{\sC_2}$.
Thus the right adjoint $f_1 \otimes \id_{\sC_2}$ is fully-faithful since it has a surjective left adjoint
by a similar argument as above.
\end{proof}

The following Lemma holds more generally but we will apply it in the special case when $\SC = \PrLst$ and $\PrRst$. (When $\SC$ is stable, this is known as the octahedral identity.)

\begin{lemma} \label{lem: second-isom}
Let $\SC$ be a category with finite limits and consider the commutative diagram in $\SC$:
$$
\begin{tikzpicture}
\node at (0,1.6) {$A$};
\node at (2,1.6) {$C$};
\node at (0,0) {$B$};
\node at (2,0) {$D$};

\draw [->, thick] (0.4,1.6) -- (1.7,1.6) node [midway, above] {$\alpha^\prime$};
\draw [->, thick] (0.4,0) -- (1.7,0) node [midway, above] {$\alpha$};

\draw [->, thick] (0,1.3) -- (0,0.3) node [midway, right] {$\beta^\prime$}; 
\draw [->, thick] (2,1.3) -- (2,0.3) node [midway, right] {$\beta$};
\end{tikzpicture}
$$ 
Then $\mathrm{Fib}(\mathrm{Fib}(\alpha) \to \mathrm{Fib}(\alpha^\prime)) = \mathrm{Fib}(\mathrm{Fib}(\beta) \to \mathrm{Fib}((\beta^\prime)) = \mathrm{Fib}(A \to (B \times_D C))$.
\end{lemma}

\begin{proof}
Limits commutate with limits and taking consecutive limits is the same as taking the total colimits. More precisely, the first equality, also known as the third isomorphism theorem, can be obtained by consider the following diagram as in \cite[Lemma 2.11]{Kuo-wrapped-sheaves}:
$$
\begin{tikzpicture}

\node at (0,3.2) {$A$};
\node at (2,3.2) {$C$};
\node at (4,3.2) {$0$};
\node at (0,1.6) {$B$};
\node at (2,1.6) {$D$};
\node at (4,1.6) {$0$};
\node at (0,0) {$0$};
\node at (2,0) {$0$};
\node at (4,0) {$0$};

\draw [<-, thick] (1.5,3.2) -- (0.3,3.2) node {$ $};
\draw [<-, thick] (2.5,3.2) -- (3.6,3.2) node {$ $};
\draw [<-, thick] (1.5,1.6) -- (0.3,1.6) node {$ $};
\draw [<-, thick] (2.5,1.6) -- (3.6,1.6) node {$ $};
\draw [<-, thick] (1.5,0) -- (0.3,0) node {$ $};
\draw [<-, thick] (2.5,0) -- (3.6,0) node {$ $};

\draw [<-, thick] (0,1.3) -- (0,0.3) node {$ $};
\draw [<-, thick] (2,1.9) -- (2,2.9) node {$ $};
\draw [<-, thick] (4,1.3) -- (4,0.3) node {$ $}; 
\draw [<-, thick] (0,1.9) -- (0,2.9) node {$ $};
\draw [<-, thick] (2,1.3) -- (2,0.3) node {$ $};
\draw [<-, thick] (4,1.9) -- (4,2.9) node {$ $}; 

\end{tikzpicture}
$$ 
That is, taking the limits first on the rows and then the resulting three-term column diagram produces $\mathrm{Fib}(\mathrm{Fib}(\alpha) \to \mathrm{Fib}(\alpha^\prime))$, and first on the columns and then the resulting three-term column diagram produces $ \mathrm{Fib}(\mathrm{Fib}(\beta) \to \mathrm{Fib}((\beta^\prime))$.
Then, to see that the total limit is given by $\mathrm{Fib}(A \to (B \times_D C))$, we notice that redundant vertices, edges, and faces can be discarded or added without changing the limit:

$$
\begin{tikzpicture}

\node at (0,3.2) {$A$};
\node at (2,3.2) {$C$};
\node at (4,3.2) {$0$};
\node at (0,1.6) {$B$};
\node at (2,1.6) {$D$};
\node at (4,1.6) {$0$};
\node at (0,0) {$0$};
\node at (2,0) {$0$};
\node at (4,0) {$0$};

\draw [->, thick] (0.3,3.2) -- (1.5,3.2) node {$ $};
\draw [->, thick] (3.6,3.2) -- (2.5,3.2) node {$ $};
\draw [->, thick] (0.3,1.6) -- (1.5,1.6) node {$ $};
\draw [->, thick] (3.6,1.6) -- (2.5,1.6) node {$ $};
\draw [->, thick] (0.3,0) -- (1.5,0) node {$ $};
\draw [->, thick] (3.6,0) -- (2.5,0) node {$ $};

\draw [<-, thick] (0,1.3) -- (0,0.3) node {$ $};
\draw [<-, thick] (2,1.9) -- (2,2.9) node {$ $};
\draw [<-, thick] (4,1.3) -- (4,0.3) node {$ $}; 
\draw [<-, thick] (0,1.9) -- (0,2.9) node {$ $};
\draw [<-, thick] (2,1.3) -- (2,0.3) node {$ $};
\draw [<-, thick] (4,1.9) -- (4,2.9) node {$ $}; 

\node at (5,1.6) {$\sim$};







\node at (6,3.2) {$A$};
\node at (8,3.2) {$A$};
\node at (10,3.2) {$A$};
\node at (6,1.6) {$B$};
\node at (8,1.6) {$D$};
\node at (10,1.6) {$C$};
\node at (6,0) {$0$};
\node at (8,0) {$0$};
\node at (10,0) {$0$};

\draw [->, thick] (6.3,1.6) -- (7.5,1.6) node {$ $};
\draw [->, thick] (9.6,1.6) -- (8.5,1.6) node {$ $};
\draw [->, thick] (6.3,0) -- (7.5,0) node {$ $};
\draw [->, thick] (9.6,0) -- (8.5,0) node {$ $};
\draw [double equal sign distance, thick] (7.5,3.2) -- (6.3,3.2) node {$ $};
\draw [double equal sign distance, thick] (8.5,3.2) -- (9.6,3.2) node {$ $};

\draw [<-, thick] (6,1.3) -- (6,0.3) node {$ $};
\draw [<-, thick] (8,1.9) -- (8,2.9) node {$ $};
\draw [<-, thick] (10,1.3) -- (10,0.3) node {$ $}; 
\draw [<-, thick] (6,1.9) -- (6,2.9) node {$ $};
\draw [<-, thick] (8,1.3) -- (8,0.3) node {$ $};
\draw [<-, thick] (10,1.9) -- (10,2.9) node {$ $}; 

\node at (9,2.4) {$\circlearrowleft$};
\node at (7,2.4) {$\circlearrowleft$};

\end{tikzpicture}
$$ 
Here $``\sim"$ indicates the fact the both diagrams have the same limit.
Now, taking the limit of the rows in the first diagram produces the limit $\mathrm{Fib}(A \to (B \times_D C))$.
\end{proof}

\begin{corollary} \label{product-of-quotient-categories}
Let $\sC_0 \hookrightarrow \sC$ and $\sD_0 \hookrightarrow \sD$ be inclusions in both $\PrLst$ and $\PrRst$ and denote the fibers by $\overline{\sC}$ and $\overline{\sD}$.
Then, the category $\overline{\sC} \otimes \overline{\sD}$ is the fiber, both in $\PrLst$ and $\PrRst$, of the map
$$\langle \sC_0 \otimes \sD, \sC \otimes \sD_0 \rangle \hookrightarrow \sC \otimes \sD.$$
\end{corollary}

\begin{proof}
Apply the above Lemma \ref{lem: second-isom} to $A = \sC \otimes \sD$, $B = \sC \otimes \sD_0$, $C = \sC_0 \otimes \sD$, and $D = \sC_0 \otimes \sD_0$.
\end{proof}

\subsection{K\"unneth formula for sheaves} \label{sec: Kunneth-sheaf}
Now we consider pairs of the form $(M,\widehat\Lambda)$ and $(N, \widehat\Sigma)$ where $M$, $N$ are manifolds and $\widehat\Lambda \subseteq T^* M$, $\widehat\Sigma \subseteq T^*N$ are singular conic isotropic.
We can form the product pair $(M \times N, \widehat\Lambda \times \widehat\Sigma)$.
The main proposition of this subsection is the following compatibility statement between this geometric product and the product structure we recall earlier:

Let $\cS$ and $\cT$ be triangulations of $M$ and $N$.
We note that, although the product stratification $\cS \times \cT$ are no longer a triangulation, it still satisfies the conditions in Lemma \ref{lem:sheaves_by_representations}.
Thus the following slight generalization of \cite[Lemma 4.7]{Ganatra-Pardon-Shende3} holds:
\begin{proposition}\label{srep}
Let $\cS$ and $\cT$ be triangulations of $M$ and $N$. Then $\Sh_{\cS \times \cT}(M \times N) = (\cS \times \cT) \dMod$.
Here we denote by $\cS \times \cT$ the product stratification.
\end{proposition}

Thus when $\widehat\Lambda = N^* \cS$ and $\widehat\Sigma = N^* \cT$ are given by Whitney triangulations, one can check directly that $(\cS \dMod) \otimes (\cT \dMod) = (\cS \times \cT) \dMod$, and we can conclude the following special case.
\begin{proposition}\label{pd:ws}
Let $\cS$ and $\cT$ Whitney triangulations of $M$ and $N$. There is an equivalence
$$\boxtimes: \Sh_{N^* \cS}(M) \otimes \Sh_{N^* \cT}(N) \xrightarrow{\sim} \Sh_{N^* (\cS \times \cT)}(M \times N)$$
sending $1_{\str(s)} \otimes 1_{\str(t)}$ to $1_{\str(s) \times \str(t)}$.
\end{proposition}

\begin{proof}[Proof of Theorem \ref{pd:g}]
To deduce the general case from the triangulation case,
pick a Whitney triangulation $\cS$ of $M$ and $\cT$ of $N$ such that $\widehat\Lambda \subseteq N^*\cS$ and $\widehat\Sigma \subseteq N^* \cT$ 
and consider the following diagram:
$$
\begin{tikzpicture}
\node at (0,2) {$\Sh_{\widehat\Lambda}(M) \otimes \Sh_{\widehat\Sigma}(N)$};
\node at (7,2) {$\Sh_{\widehat\Lambda \times \widehat\Sigma}(M \times N)$};
\node at (0,0) {$\Sh_{N^* \cS}(M) \otimes \Sh_{N^* \cT}(N)$};
\node at (7,0) {$\Sh_{N^* \cS \times N^* \cT}(M \times N)$};

\draw [->, thick] (1.8,2) -- (5.4,2) node [midway, above] {$\boxtimes$};
\draw [double equal sign distance, thick] (2.2,0) -- (5.1,0) node [midway, below] {$\boxtimes$};

\draw [right hook->, thick] (0,1.7) -- (0,0.3) node [midway, left] {$ $}; 
\draw [right hook->, thick] (7,1.7) -- (7,0.3) node [midway, right] {$ $};
\end{tikzpicture}
$$ 
The fully-faithfulness of the vertical functor on the left is implied by Lemma \ref{ff;pr}.
Since the diagram commutes, the horizontal map on the upper row is also fully-faithful.
Pass to the left adjoints and restrict to compact objects,
the equivalence for the general case will be implied by Proposition \ref{wdstopr} 
and the proposition cited below,
whose counterpart in the Fukaya setting is discussed in a more general situation 
in \cite[Section 6]{Ganatra-Pardon-Shende2}.
\end{proof}

\begin{proposition}\label{prop:thom-sebastiani}
Let $(x,\xi) \in N^* \cS$ and $(y,\eta) \in N^* \cT$.
We denote by $D_{(x,\xi)}$ and $D_{(y,\eta)}$ corepresentatives of the microstalk
functors at $(x,\xi)$ and $(y,\eta)$.
Then $D_{(x,\xi)} \boxtimes D_{(y,\eta)}$ corepresents the microstalk at $(x,y,\xi,\eta)$.
\end{proposition}

\begin{proof}
By Proposition \ref{wdstopr}, it's sufficient to show that for $F \in \Sh_\cS(M)$ and $G \in \Sh_\cT(N)$, there is an equivalence
$$ \mu_{(x,\xi)}(F) \boxtimes \mu_{(y,\eta)}(G) = \mu_{(x,y,\xi,\eta)}(F \boxtimes G)$$
since corepresentative are unique.
This is the Thom-Sebastiani theorem whose proof in the relevant setting can be found in for example
\cite[Sebastiani-Thom Isomorphism]{D.Massey} or \cite[Theorem 1.2.2]{Schurmann}.
\end{proof}

\begin{remark}
As explained in Remark \ref{rem:stalk-corepresentative}, when $\xi = 0$, the microstalk functors $\mu_{(x,\xi)}$ are simply stalk functors $i_x^*$, and the above result in particular applies as $i_x^*(F) \boxtimes i_y^*(G) = i_{(x,y)}^*(F \boxtimes G)$.
\end{remark}

\begin{remark}
We remark that the theorem is stated as compatibility between vanishing cycles
with exterior products $\boxtimes$ in the setting of complex manifold.
The proof, however, holds in our case since vanishing cycles $\phi_f(F)$ are traded with
$\Gamma_{\{ \Re f \geq 0\} }(F) |_{f^{-1}(0)}$ at the beginning of the proof in for example \cite{D.Massey}.  
Furthermore the various computations performed there, for example,
$$ f^* (\VD{Y}(F) ) \cong \VD{X}( f^! F),$$
for a real analytic map $f: X \rightarrow Y$ and Verdier dualities $\VD{X}$ and $\VD{Y}$, require only $\RR$-constructibility. 
\end{remark}

In fact, note that the localization functors in Proposition \ref{prop:stopremoval} are compatible with the K\"unneth formula. More precisely, we have the following statement:

\begin{proposition}\label{prop:stopremoval-kunneth}
Let $\widehat\Lambda \subseteq \widehat\Lambda^\prime \subseteq T^*M$ and $\widehat\Sigma \subseteq \widehat\Lambda^\prime \subseteq T^*N$ be subanalytic conic isotropic subsets. Then there is a commutative diagram
    \[\xymatrix@C=25mm@R=12mm{
    \Sh_{\widehat\Lambda^\prime}(M) \otimes \Sh_{\widehat\Sigma^\prime}(N) \ar[r]^{\sim} \ar[d]_{\iota_{\widehat\Lambda}^* \otimes \,\iota_{\widehat\Sigma}^*} & \Sh_{\widehat{\Lambda}^\prime \times \widehat{\Sigma}^\prime}(M \times N) \ar[d]^{\iota_{\widehat\Lambda \times \widehat\Sigma}^*} \\
    \Sh_{\widehat\Lambda}(M) \otimes \Sh_{\widehat\Sigma}(N) \ar[r]^{\sim} & \Sh_{\widehat\Lambda \times \widehat\Sigma}(M \times N).
    }\]
\end{proposition}
\begin{proof}
    Under the K\"unneth formula, we have $\iota_{\widehat\Lambda \times \widehat\Sigma*}(F \boxtimes G) = \iota_{\widehat\Lambda*}(F) \boxtimes \,\iota_{\widehat\Sigma*}(G) = F \boxtimes G$. Taking the left adjoints then gives the commutative diagram.
\end{proof}

\subsection{K\"unneth functor and the doubling}\label{sec:Kunneth-doubling}

We would like to deduce the K\"unneth formula for microsheaves by reducing it to the case of sheaves.
We will obtain the statement from their sheaf theoretic equivalents by using the doubling trick. 

For a subanalytic isotropic subset $\Lambda \subseteq S^*M$, in this section, we define the conic isotropic subset
\begin{align*}
\widehat\Lambda = M \cup (\Lambda \times \bR_+).
\end{align*}
Recall that a contact flow $\varphi: S^*M \times I \to S^*M$ is called positive if $\alpha(\partial_t\varphi_t) \geq 0$. We set $\Lambda_{\epsilon}, \Lambda_{-\epsilon} \subseteq S^* M$ to be any positive and negative contact push-off of $\Lambda$ that displaces the isotropic subset. In this section, we will adopt the notation that $\Lambda_{\pm\epsilon} = \Lambda_{-\epsilon} \cup \Lambda_{\epsilon}$ and respectively $\widehat\Lambda_{\pm\epsilon} = \widehat\Lambda_{-\epsilon} \cup \widehat\Lambda_{\epsilon}$. 

First, 
using the doubling functor in \cite{Kuo-Li-spherical}, we identify $\msh_\Lambda(\Lambda)$ and $\msh_\Sigma(\Sigma)$ as sheaves microsupported on the doubling:

\begin{theorem}[{\cite[Theorem 4.1 \& Proposition 6.3]{Kuo-Li-spherical}}] \label{thm: doubling} 
Let $\Lambda \subseteq S^*M$ be a compact subanalytic isotropic subset. There is a fully faithful functor
$$m_\Lambda^l: \msh_\Lambda(\Lambda) \hookrightarrow \Sh_{\Lambda_{-\epsilon} \cup \Lambda_\epsilon}(M)$$
which induces a recollement that gives the localization sequence in $\PrLst$, in the sense of \cite[Definition 3.2]{Hoyois-Scherotzke-Sibilla}, 
$$\msh_\Lambda(\Lambda) \hookrightarrow \Sh_{\Lambda_{-\epsilon} \cup \Lambda_\epsilon}(M) \twoheadrightarrow \Sh_{\Lambda_\epsilon}(M),$$
and the essential image of $m_\Lambda^l$ is the category $\Sh_{\widehat\Lambda_{\cup,\epsilon}}(M)$, where
$$\widehat\Lambda_{\cup,\epsilon} = \big((\Lambda_{-\epsilon} \cup \Lambda_\epsilon) \times \bR_+\big) \cup \bigcup\nolimits_{-\epsilon \leq s \leq \epsilon}\pi(\Lambda_s).$$
\end{theorem}

Using the above Proposition, we can deduce a fully faithful embedding of the product of microsheaves into the product of sheaves.

\begin{proposition}\label{prop:kunneth-micro-image}
Let $\Lambda \subseteq S^*M$ and $\Sigma \subseteq S^*N$ be a compact subanalytic isotropic subset. Then there is a fully faithful embedding
$$\msh_{\Lambda}(\Lambda) \otimes \msh_{\Sigma}(\Sigma) \hookrightarrow \Sh_{\Lambda_{\pm\epsilon}}(M) \otimes \Sh_{\Sigma_{\pm\epsilon}}(N) \xrightarrow{\sim} \Sh_{\widehat\Lambda_{\pm\epsilon} \times \widehat\Sigma_{\pm\epsilon}}(M \times N),$$ 
where the essential image is $\Sh_{\widehat\Lambda_{\cup,\epsilon} \times \widehat\Sigma_{\cup,\epsilon}}(M \times N)$.
\end{proposition}


Second, using the relative doubling functor in \cite{Kuo-Li-spherical}, we identify $\msh_{\Lambda \times \Sigma \times \bR}(\Lambda \times \Sigma)$ as sheaves microsupported on the product of the doubling. Here, we fix the embedding of $\Lambda \times \Sigma \times \bR \subseteq S^*M \times S^*N \times \bR \subseteq S^*(M \times N)$ by considering the conic subset $(\Lambda \times \bR_{>0}) \times (\Sigma \times \bR_{>0}) \subseteq \dT^*(M \times N)$ quotient by the diagonal $\bR_{>0}$-action (this is equivalent to fixing a polarization; see Remark \ref{rem:polarization}). The main theorem we prove will be the following:

\begin{theorem}\label{thm: doubling-small-piece}
Let $\Lambda \subseteq S^*M$ and $\Sigma \subseteq S^*N$ be compact subanalytic isotropic subsets. There is a fully faithful functor
$$m_{\Lambda \times \Sigma}^l: \msh_{\Lambda \times \Sigma \times \bR}(\Lambda \times \Sigma) \hookrightarrow \Sh_{\widehat\Lambda_{\pm\epsilon} \times \widehat\Sigma_{\pm\epsilon}}(M \times N)$$
which induces a localization sequence
$$\msh_{\Lambda \times \Sigma \times \bR}(\Lambda \times \Sigma) \hookrightarrow \Sh_{\widehat\Lambda_{\pm\epsilon} \times \widehat\Sigma_{\pm\epsilon}}(M \times N) \twoheadrightarrow \Sh_{\widehat\Lambda_{\pm\epsilon} \times \widehat\Sigma_{\pm\epsilon} \setminus (\Lambda \times \Sigma \times \bR \times \bR_+)}(M \times N).$$
\end{theorem}

First, we need to identify $\msh_{\Lambda \times \Sigma \times \bR}(\Lambda \times \Sigma)$ with the section of $\msh_{\widehat\Lambda_{\cup,\epsilon} \times \widehat\Sigma_{\cup, \epsilon}}$ on the open subset $(\Lambda_{-\epsilon} \times \bR_{>0}) \times (\Sigma_{-\epsilon} \times \bR_{>0})$ quotient by the diagonal $\bR_{>0}$-action, which is contactomorphic to $(\Lambda \times \bR_{>0}) \times (\Sigma \times \bR_{>0})$ quotient by the diagonal $\bR_{>0}$-action.

\begin{lemma}\label{lem:contact-trans}
Let $\Lambda \times \Sigma \times \bR \subseteq S^{*}(M \times N)$ be identified with the subanalytic isotropic subset at infinity of $(\Lambda_{-\epsilon} \times \bR_{>0}) \times (\Sigma_{-\epsilon} \times \bR_{>0}) \subseteq \widehat\Lambda_{\pm\epsilon} \times \widehat\Sigma_{\pm\epsilon}$. Then there is an equivalence
$$\msh_{\Lambda \times \Sigma \times \bR}(\Lambda \times \Sigma) = \msh_{\widehat\Lambda_{\pm\epsilon} \times \widehat\Sigma_{\pm\epsilon}}(\Lambda_{-\epsilon} \times \Sigma_{-\epsilon}) = \msh_{\widehat\Lambda_{\pm\epsilon} \times \widehat\Sigma_{\pm\epsilon}}(\Lambda_{-\epsilon} \times \Sigma_{-\epsilon} \times \bR).$$
\end{lemma}
\begin{proof}
First, since the microstalk functors along $\Lambda_{-\epsilon} \times \Sigma_{-\epsilon} \times \bR$ are locally constant along $\bR$, the microsheaves along $\Lambda_{-\epsilon} \times \Sigma_{-\epsilon} \times \bR$ are locally constant along $\bR$. Thus, we know that
$$\msh_{\widehat\Lambda_{\pm\epsilon} \times \widehat\Sigma_{\pm\epsilon}}(\Lambda_{-\epsilon} \times \Sigma_{-\epsilon} \times \bR) = \msh_{\widehat\Lambda_{\pm\epsilon} \times \widehat\Sigma_{\pm\epsilon}}(\Lambda_{-\epsilon} \times \Sigma_{-\epsilon}).$$
Then, since microsheaves form a sheaf of categories and is invariant under contact isotopies \cite[Theorem 7.2.1]{KS}, for $\Lambda_{-\epsilon} \times \Sigma_{-\epsilon} \times \bR \subseteq \widehat\Lambda_{\pm\epsilon} \times \widehat\Sigma_{\pm\epsilon}$, we know that
$$\msh_{\Lambda \times \Sigma \times \bR}(\Lambda \times \Sigma) = \msh_{\Lambda_{-\epsilon} \times \Sigma_{-\epsilon}}(\Lambda_{-\epsilon} \times \Sigma_{-\epsilon}) = \msh_{\widehat\Lambda_{\pm\epsilon} \times \widehat\Sigma_{\pm\epsilon}}(\Lambda_{-\epsilon} \times \Sigma_{-\epsilon}).$$
This therefore completes the proof.
\end{proof}

Then, we state the relative doubling theorem in this setting. For an open subset $\Omega \subseteq \widehat\Lambda_{\pm\epsilon} \times \widehat\Sigma_{\pm\epsilon}$, we recall the convention in \cite[Section 4.6]{Kuo-Li-spherical} that a non-negative contact flow $T_t: S^*(M \times N) \to S^*(M \times N)$ is supported on it if it is supported on an open subset $\tilde{\Omega} \subseteq S^*(M \times N)$ such that
$$\tilde{\Omega} \cap \widehat\Lambda_{\cup,\epsilon} \times \widehat\Sigma_{\cup, \epsilon} = \Omega.$$

First, we consider the image of microsheaves on the open subset $\Lambda \times \Sigma$ in the product of the doubling:
$$\msh_{\widehat\Lambda_{\pm\epsilon} \times \widehat\Sigma_{\pm\epsilon}}(\Lambda \times \Sigma) \rightarrow \Sh_{\widehat\Lambda_{\pm\epsilon} \times \widehat\Sigma_{\pm\epsilon}}(M \times N).$$
Relative doubling functor on $\Lambda \times \Sigma \times \bR$ gives an explicit characterization of this functor.

\begin{proposition}[{\cite[Theorem 4.47]{Kuo-Li-spherical}}]\label{thm:relative-doubling}
Let $\hat T_t$ be a non-negative contact flow supported on the open set $\Lambda \times \Sigma \times \bR$. Then for $\delta > 0$ sufficiently small, there is a fully faithful functor
$$w_{\Lambda \times \Sigma}: \msh_{\widehat\Lambda_{\pm\epsilon} \times \widehat\Sigma_{\pm\epsilon}}(\Lambda_{-\epsilon} \times \Sigma_{-\epsilon}) \hookrightarrow \Sh_{\hat T_{-\delta}(\Lambda \times \Sigma \times \overline{\bR}) \cup \hat T_\delta(\Lambda \times \Sigma \times \overline{\bR})}(M \times N),$$
where $\overline{\bR}:=[-\infty, +\infty]$ is the closure of $\mathbb{R}$. Moreover, we have $m_{\Lambda \times \Sigma}^l = \iota_{\widehat\Lambda_{\pm\epsilon} \times \widehat\Sigma_{\pm\epsilon}}^* \circ w_{\Lambda \times \Sigma}[-1]$.
\end{proposition}

Consider the relative doubling functor composed with the localization functor
$$\msh_{\widehat\Lambda_{\pm\epsilon} \times \widehat\Sigma_{\pm\epsilon}}(\Lambda_{-\epsilon} \times \Sigma_{-\epsilon}) \hookrightarrow \Sh_{\hat T_{-\delta}(\Lambda \times \Sigma \times \overline{\bR}) \cup \hat T_\delta(\Lambda \times \Sigma \times \overline{\bR})}(M \times N) \twoheadrightarrow \Sh_{\widehat\Lambda_{\pm\epsilon} \times \widehat\Sigma_{\pm\epsilon}}(M \times N).$$
We already know that the first functor is fully faithful, so it suffices to show that the second functor is fully faithful.

Recall that when we consider the category $\Sh_{\Lambda}(M) = \Sh_{\widehat\Lambda}(M)$, it is shown by the first author \cite{Kuo-wrapped-sheaves} that the left (resp.~right) adjoint of $\iota_{\Lambda*}: \Sh_{\Lambda}(M) \hookrightarrow \Sh(M)$ is given by colimit (resp.~limit) of positive (resp.~negative) contact push-offs that are supported away from $\Lambda$ (they are also called wrappings). More precisely, for any contact isotopy $\Phi: S^* M \times I \rightarrow S^* M$ with time-1 flow $\varphi: S^*M \to S^*M$, where $I$ is an open interval containing $0$, there exists a unique Guillermou--Kashiwara--Schapira sheaf kernel $K(\Phi) \in \Sh(M \times M \times I)$ that restricts to $K(\varphi) \in \Sh(M \times M)$ \cite[Theorem 3.7]{Guillermou-Kashiwara-Schapira}, which induces an equivalence functor
$$K(\varphi) \circ (-): \Sh(M) \to \Sh(M), \;\;\; F \mapsto \varphi(F) := K(\varphi) \circ F.$$
We will also use the notation $F^\varphi$ for $\varphi(F) = K(\varphi) \circ F$. When $\Phi$ is positive, there exists a continuation map $F \to F^\varphi$. Then the left and the right adjoint of $\iota_{\Lambda*}$ are given as follows:
\begin{equation}\label{for: large-wrappings}
\iota_\Lambda^*(F) = \wrap_\Lambda^+(F) = \clmi{\varphi \in W(T^*M \setminus \Lambda)}F^\varphi, \; \iota_\Lambda^!(F) = \wrap_\Lambda^-(F) = \lmi{\varphi \in W(T^*M \setminus \Lambda)}F^{\varphi^{-1}}.
\end{equation}
Here, we use $W^+( S^* M \setminus \Lambda)$ 
to represent positive isotopies compactly supported away from $\Lambda$. 
We will try to use the full faithfulness criterion in \cite[Section 5.2]{Kuo-wrapped-sheaves}.

In order to understand the wrapping functors, we will need to understand the symplectic geometry of the doubling. The following construction will be important in the proof of Theorem \ref{thm: doubling-small-piece}. Let $\Lambda \subseteq S^*M$ be a singular Legendrian subset, we can define the $U$-shape Lagrangian filling $\Lambda \times \cup_{\epsilon,\epsilon'}$ of the double copy $\Lambda_{\epsilon} \cup \Lambda_{\epsilon'}$ as follows. Let $f: (\epsilon, \epsilon') \to \bR_{>0}$ be a smooth function such that $f(s) \to +\infty$ when $s \to \epsilon$ or $\epsilon'$. 
\begin{equation}\label{eq:u-shape-filling}
\Lambda \times \cup_{\epsilon,\epsilon'} = \{(x, r\xi) \mid (x, \xi) \in \Lambda_s \subseteq S^*M, r = f(s) \in \bR_{>0}\}.
\end{equation}
Under the Liouville flow in $T^*M$, $\Lambda \times \cup_{\pm\epsilon}$ can be sent to an arbitrary small neighbourhood of 
$\widehat\Lambda_{\pm\epsilon}$.

As a running example of Theorem \ref{thm: doubling-small-piece}, the reader may consider the case $M = N = \bR$, $\Lambda = \{(0, -1)\}$ and $\Sigma = \{(0, 1)\} \subseteq S^*\bR$, as illustrated in Figure \ref{fig:double-small-piece}. When we apply the relative doubling functor in Proposition \ref{thm:relative-doubling}, we will see the Legendrian on the left of the figure, given by a standard Legendrian unknot in $S^*\mathbb{R}^2$. One can see that it is isotopic to the Legendrian on the right of the figure which consists of two pieces, each of which is a U-shape Lagrangian filling of the two points $\Lambda_{\pm\epsilon}$ and respectively $\Sigma_{\pm\epsilon}$. That is contained in a small neighbourhood of $\widehat\Lambda_{\pm\epsilon} \times \widehat\Sigma_{\pm\epsilon}$, where the full faithfulness criterion in \cite[Section 5.2]{Kuo-wrapped-sheaves} applies.

\begin{proof}[Proof of Theorem \ref{thm: doubling-small-piece}]
    We consider the Lagrangian subset $\widehat\Lambda_{\pm\epsilon} \times \widehat\Sigma_{\pm\epsilon}$. We will show that there exists some particularly nice choice of the relative Legendrian doubling $\hat T_{-\delta}(\Lambda \times \Sigma \times \overline{\bR}) \cup \hat T_\delta(\Lambda \times \Sigma \times \overline{\bR})$ that is contained in a small neighbourhood of $\widehat\Lambda_{\pm\epsilon} \times \widehat\Sigma_{\pm\epsilon}$. 

    We construct that particular relative doubling as follows. Note that we have a decomposition
    \begin{align*}
    \widehat\Lambda_{\pm\epsilon} \times \widehat\Sigma_{\pm \epsilon} = (M \times N) \cup \big( \big((\widehat\Lambda_{\pm\epsilon} \times \Sigma_{\pm \epsilon}\big) \cup \big(\Lambda_{\pm \epsilon} \times \widehat\Sigma_{\pm\epsilon}) \big) \big) \times \bR_+.
    \end{align*}
    First, consider $\widehat\Lambda_{\pm\epsilon} \times \Sigma_{\pm\epsilon}$ with a small neighbourhood $D^*M \times U(\Sigma_{\pm\epsilon})$ where $D^*M$ is the disk cotangent bundle of $M$. Construct the standard U-shape Lagrangian filling $\Lambda \times \cup_{\pm\epsilon}$ of $\Lambda_{\pm\epsilon}$ inside the disk bundle $D^*M$ as in Equation (\ref{eq:u-shape-filling}). Under the Liouville flow on $D^*M$, we know that it can be sent to a small neighbourhood of $\widehat\Lambda_{\pm\epsilon}$. Now we have an isotropic subset
    $$(\Lambda \times \cup_{\pm\epsilon}) \times \Sigma_{\pm\epsilon} \subseteq D^*M \times U(\Sigma_{\pm\epsilon}).$$
    Next, consider $\Lambda_{\pm\epsilon} \times \widehat\Sigma_{\pm\epsilon}$. We can similarly construct an isotropic
    $$\Lambda_{\pm\epsilon} \times (\Sigma \times \cup_{\pm\epsilon}) \subseteq U(\Lambda_{\pm\epsilon}) \times D^*N.$$
    Their boundary are equal to $\Lambda_{\pm\epsilon} \times \Sigma_{\pm\epsilon}$ and hence can be glued together which defines a Legendrian doubling. See Figure \ref{fig:double-small-piece}.

\begin{figure}
    \centering
    \includegraphics[width=0.7\textwidth]{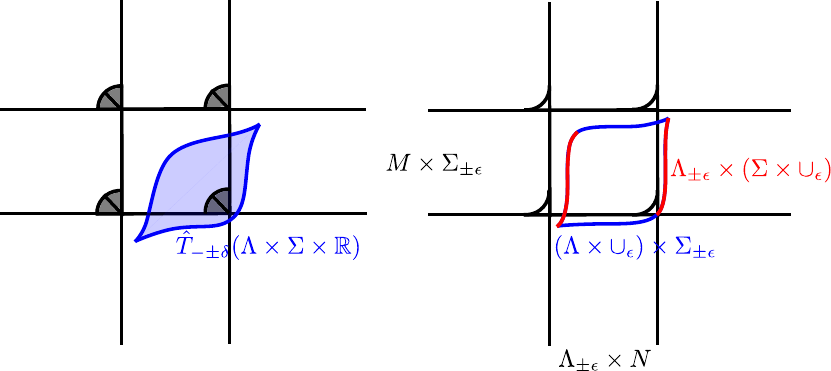}
    \caption{Let $M = N = \bR$, $\Lambda = \{(0, -1)\}$ and $\Sigma = \{(0, 1)\} \subseteq S^*\bR$. The figure on the left is the relative doubling construction along $\Lambda \times \Sigma \times \bR$. The figure on the right is the union of $\Lambda_{\pm\epsilon} \times (\Sigma \times \cup_\epsilon)$ (in blue) and $(\Lambda \times \cup_\epsilon) \times \Sigma_{\pm\epsilon}$ (in red), which is now contained in a small neighbourhood of $\widehat\Lambda_{\pm\epsilon} \times \widehat\Sigma_{\pm\epsilon}$.}
    \label{fig:double-small-piece}
\end{figure}

    We claim that this Legendrian doubling is contact isotopic to $\hat T_{-\delta}(\Lambda \times \Sigma \times \overline{\bR}) \cup \hat T_\delta(\Lambda \times \Sigma \times \overline{\bR})$ in the complement of $\widehat\Lambda_{\pm\epsilon} \times \widehat\Sigma_{\pm\epsilon}$. First, note that the two branches $\Lambda_\epsilon \times (\Sigma \times \cup_{\pm\epsilon})$ and $\Lambda_{-\epsilon} \times (\Sigma \times \cup_{\pm\epsilon})$ are connected through the isotopy $\Lambda_s \times (\Sigma \times \cup_{\pm\epsilon})$; the Lagrangian filling $\Lambda \times \cup_{\pm\epsilon}$ of $\Lambda_\epsilon \cup \Lambda_{-\epsilon}$ and the Lagrangian filling $\Lambda \times \cup_{s,-\epsilon}$ of $\Lambda_s \cup \Lambda_{-\epsilon}$. Therefore, we have a Legendrian isotopy
    $$(\Lambda_\epsilon \times (\Sigma \times \cup_{\pm\epsilon})) \cup ((\Lambda \times \cup_{\pm\epsilon}) \times \Sigma) \cong (\Lambda_s \times (\Sigma \times \cup_{\pm\epsilon})) \cup ((\Lambda \times \cup_{s,-\epsilon}) \times \Sigma) \cong \Lambda_{-\epsilon} \times (\Sigma \times \cup_{\pm\epsilon}).$$
    This implies that the Legendrian doubling is isotopic to double copies of the Legendrian $\Lambda_{-\epsilon} \times (\Sigma \times \cup_{\pm\epsilon})$ (in other words, the standard Legendrian unknot times $\Lambda_{-\epsilon} \times (\Sigma \times \cup_{\pm\epsilon})$).
    Second, note that we also have a Legendrian isotopy
    $$\Lambda_{-\epsilon} \times (\Sigma \times \cup_{\pm\epsilon}) \cong \Lambda_{-\epsilon} \times (\Sigma \times \cup_{t,-\epsilon}) \cong \Lambda_{-\epsilon} \times \Sigma_{-\epsilon} \times \bR.$$
    This implies that the Legendrian doubling is isotopic to double copies of the Legendrian $\Lambda_{-\epsilon} \times \Sigma_{-\epsilon} \times \bR$ (in other words, the standard Legendrian unknot times $\Lambda_{-\epsilon} \times \Sigma_{-\epsilon} \times \bR$). Therefore,
    $$\big((\Lambda \times \cup_{-\epsilon}) \times \Sigma_{\pm\epsilon} \big) \cup \big(\Lambda_{\pm\epsilon} \times (\Sigma \times \cup_{\pm\epsilon}) \big) \cong \hat T_{-\delta}(\Lambda \times \Sigma \times \overline{\bR}) \cup \hat T_\delta(\Lambda \times \Sigma \times \overline{\bR}).$$
    
    Consider the composition in Proposition \ref{thm:relative-doubling}
    $$m_{\Lambda \times \Sigma}^l = \iota_{\widehat\Lambda_{\pm\epsilon} \times \widehat\Sigma_{\pm\epsilon}}^* \circ w_{\Lambda \times \Sigma \times \bR}[-1] = \wrap_{\widehat\Lambda_{\pm\epsilon} \times \widehat\Sigma_{\pm\epsilon}}^+ \circ w_{\Lambda \times \Sigma \times \bR}[-1].$$
    The localization functor can be characterized in terms of wrapping. Hence the fact that the singular support of the relative doubling is contained in an arbitrary small neighbourhood of $ \widehat\Lambda_{\pm\epsilon} \times \widehat\Sigma_{\pm\epsilon}$ at infinity means that the localization or wrapping functor is fully faithful by \cite[Theorem 5.15]{Kuo-wrapped-sheaves}.

    Proposition \ref{prop:stop-removal-microlocalization} says that the fiber of the functor
    $$\iota^*_{\widehat\Lambda_{\pm\epsilon} \times \widehat\Sigma_{\pm\epsilon} \setminus (\Lambda \times \Sigma \times \bR \times \bR_+)}: \Sh_{\widehat\Lambda_{\pm\epsilon} \times \widehat\Sigma_{\pm\epsilon}}(M \times N) \twoheadrightarrow \Sh_{\widehat\Lambda_{\pm\epsilon} \times \widehat\Sigma_{\pm\epsilon} \setminus (\Lambda \times \Sigma \times \bR \times \bR_+)}(M \times N)$$
    is the essential image of the left adjoint of the microlocalization functor
    $$m_{\Lambda_{-\epsilon} \times \Sigma_{\Lambda_\epsilon}}^l: \msh_{\widehat\Lambda_{\pm\epsilon} \times \widehat\Sigma_{\pm\epsilon}}(\Lambda_{-\epsilon} \times \Sigma_{-\epsilon}) \to \Sh_{\widehat\Lambda_{\pm\epsilon} \times \widehat\Sigma_{\pm\epsilon}}(M \times N).$$
    Then the result follows from Lemma \ref{lem:contact-trans}.
\end{proof}

Using the same technique, we can also identify $\msh_{\widehat\Lambda \times \widehat\Sigma}(\Lambda \times \widehat\Sigma)$ as sheaves microsupported on the product of the doubling. 

\begin{theorem}\label{thm: doubling-product-piece}
Let $\Lambda \subseteq S^*M$ and $\Sigma \subseteq S^*N$ be a compact subanalytic isotropic subset. There is a fully faithful functor
$$m_{\Lambda \times \widehat\Sigma}^l: \msh_{\widehat\Lambda \times \widehat\Sigma}(\Lambda \times \widehat\Sigma) \hookrightarrow \Sh_{\widehat\Lambda_{\pm\epsilon} \times \widehat\Sigma}(M \times N)$$
which induces a localization sequence
$$\msh_{\widehat\Lambda \times \widehat\Sigma}(\Lambda \times \widehat\Sigma) \hookrightarrow \Sh_{\widehat\Lambda_{\pm\epsilon} \times \widehat\Sigma}(M \times N) \twoheadrightarrow \Sh_{\widehat\Lambda_{\epsilon} \times \widehat\Sigma}(M \times N).$$
\end{theorem}

First, we have the relative doubling functor on $\Lambda \times \widehat\Sigma$.

\begin{proposition}[{\cite[Theorem 4.47]{Kuo-Li-spherical}}]\label{thm:relative-doubling-mix}
Let $\hat T_t$ be a non-negative contact flow supported on the open set $\Lambda \times \widehat\Sigma$. Then for $\delta > 0$ sufficiently small, there is a fully faithful functor
$$w_{\Lambda \times \widehat\Sigma}: \msh_{\widehat\Lambda \times \widehat\Sigma}(\Lambda \times \widehat\Sigma) \hookrightarrow \Sh_{\hat T_{-\delta}(\Lambda \times \widehat\Sigma) \cup \hat T_\delta(\Lambda \times \widehat\Sigma)}(M \times N).$$
Moreover, there is an equivalence of functors $m_{\Lambda_{-\epsilon} \times \widehat\Sigma}^l = \iota_{\widehat\Lambda_{\pm\epsilon} \times \widehat\Sigma}^* \circ w_{\Lambda \times \widehat\Sigma}[-1]$.
\end{proposition}

Now, as a running example of Theorem \ref{thm: doubling-product-piece}, the reader may again consider the case $M = N = \bR$, $\Lambda = \{(0, -1)\}$ and $\Sigma = \{(0, 1)\} \subseteq S^*\bR$, as illustrated in Figure \ref{fig:double-mixed-piece}. When we apply the relative doubling functor in Proposition \ref{thm:relative-doubling}, we will see the Legendrian on the left of the figure. One can see that it is isotopic to the Legendrian on the right of the figure which consists of two copies of $\widehat\Sigma$, joined by a U-shape Lagrangian filling of the two points $\Lambda_{\pm\epsilon}$. That is contained in a small neighbourhood of $\widehat\Lambda_{\pm\epsilon} \times \widehat\Sigma$, where the full faithfulness criterion in \cite[Section 5.2]{Kuo-wrapped-sheaves} applies.

\begin{proof}[Proof of Theorem \ref{thm: doubling-product-piece}]
    Similar to the proof of Theorem \ref{thm: doubling-small-piece}, we will show that there exists some nice choice of the relative Legendrian doubling $\hat T_{-\delta}(\Lambda_{\epsilon} \times \widehat\Sigma) \cup \hat T_\delta(\Lambda_{\epsilon} \times \widehat\Sigma)$ that is contained in a small neighbourhood of $\widehat\Lambda_{\pm\epsilon} \times \widehat\Sigma$. 

    Consider $\widehat\Lambda_{\pm\epsilon} \times \Sigma$ with a neighbourhood $D^*M \times U(\Sigma)$. Construct the standard U-shape Lagrangian filling $\Lambda \times \cup_{\pm\epsilon}$ of $\Lambda_{\pm \epsilon}$ in the disk bundle $D^*M$ as in Equation (\ref{eq:u-shape-filling}). Under the Liouville flow on $D^*M$, it can be sent to a small neighbourhood of $\widehat\Lambda_{\pm\epsilon}$. The we consider the isotropic subset
    $$\Lambda_{\pm\epsilon} \times \widehat\Sigma \subseteq U(\Lambda_{\pm\epsilon}) \times D^*N, \;\; (\Lambda \times \cup_{\pm\epsilon}) \times \Sigma \subseteq D^*M \times U(\Sigma).$$
    Gluing them together defines a Legendrian doubling. See Figure \ref{fig:double-mixed-piece}. The Legendrian doubling is contact isotopic to $\hat T_{-\delta}(\Lambda_{\epsilon} \times \widehat\Sigma) \cup \hat T_\delta(\Lambda_{\epsilon} \times \widehat\Sigma)$ in the complement of $\widehat\Lambda_{\pm\epsilon} \times \widehat\Sigma$, as the Legendrians $(\Lambda_\epsilon \times \widehat\Sigma) \cup ((\Lambda \times \cup_{\pm\epsilon}) \times \Sigma)$ and $\Lambda_{-\epsilon} \times \widehat\Sigma$ are isotopic through $(\Lambda_s \times \widehat\Sigma) \cup ((\Lambda \times \cup_{-\epsilon,s}) \times \Sigma)$, where $\Lambda \times \cup_{-\epsilon,s}$ is the Lagrangian filling of $\Lambda_s \cup \Lambda_{-\epsilon}$.

\begin{figure}
    \centering
    \includegraphics[width=0.6\textwidth]{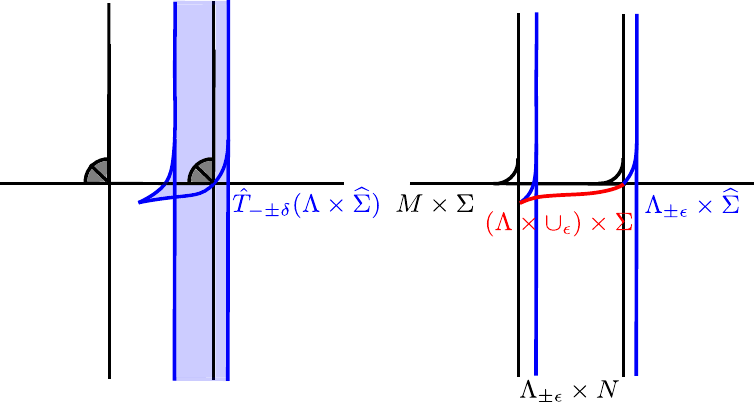}
    \caption{Let $M = N = \bR$, $\Lambda = \{(0, -1)\}$ and $\Sigma = \{(0, 1)\} \subseteq S^*\bR$. The figure on the left is the relative doubling construction along $\Lambda \times \widehat\Sigma$. The figure on the right is the construction that glues together $\Lambda_{\pm\epsilon} \times \widehat\Sigma$ (in blue) and $(\Lambda \times \cup_\epsilon) \times \Sigma$ (in red), which is now contained in a small neighbourhood of $\widehat\Lambda_{\pm\epsilon} \times \widehat\Sigma$.}
    \label{fig:double-mixed-piece}
\end{figure}
    
    Consider the composition in Theorem \ref{thm:relative-doubling-mix}
    $$m_{\Lambda_{\epsilon} \times \widehat\Sigma}^l = \iota_{\widehat\Lambda_{\pm\epsilon} \times \widehat\Sigma}^* \circ w_{\Lambda \times \widehat\Sigma}[-1] = \wrap_{\widehat\Lambda_{\pm\epsilon} \times \widehat\Sigma}^+ \circ w_{\Lambda \times \widehat\Sigma}[-1].$$
    Hence the fact that the singular support of the relative doubling is contained in an arbitrary small neighbourhood of $\widehat\Lambda_{\pm\epsilon} \times \widehat\Sigma$ at infinity means that the localization or wrapping functor is fully faithful by \cite[Theorem 5.15]{Kuo-wrapped-sheaves}.
\end{proof}

\subsection{K\"unneth formula for microsheaves} \label{sec: Kunneth-microsheaf}
Using the results in Section \ref{sec:Kunneth-doubling}, we will now prove the K\"unneth formula for microsheaves.

Theorem \ref{thm: doubling} and \ref{thm: doubling-small-piece} implies that we have two fully faithful embeddings (using the product of doubling and the relative doubling of the product):
\begin{gather*}
\msh_{\Lambda}(\Lambda) \otimes \msh_{\Lambda}(\Sigma) \hookrightarrow \Sh_{\widehat\Lambda_{\pm\epsilon}}(M) \otimes \Sh_{\widehat\Sigma_{\pm\epsilon}}(N), \\
\msh_{\Lambda \times \Sigma \times \bR}(\Lambda \times \Sigma) \hookrightarrow \Sh_{\widehat\Lambda_{\pm\epsilon} \times \widehat\Sigma_{\pm\epsilon}}(M \times N).
\end{gather*}
Our goal in this subsection is to show that there essential images agree under K\"unneth formula for sheaves, which will then imply the K\"unneth formula for microsheaves with isotropic microsupports.

First, however, we will show a simpler case, namely, the K\"unneth formula between microsheaves and sheaves, by comparing the fully faithful functors in Theorem \ref{thm: doubling} and Theorem \ref{thm: doubling-product-piece}.

\begin{theorem}\label{thm: kunneth-sheaf-microsheaf}
    Let $\Lambda \subseteq S^*M, \Sigma \subseteq S^*N$ be compact subanalytic isotropic subsets. Then there is an equivalence
    $$\msh_\Lambda(\Lambda) \otimes \Sh_{\widehat\Sigma}(N) = \msh_{\Lambda \times \widehat\Sigma}(\Lambda \times \widehat\Sigma).$$
\end{theorem}
\begin{proof}
    Consider the doubling construction in Theorem \ref{thm: doubling}, we have a recollement induced by the inclusion functor $m_{\Lambda}^l \otimes\,\id = m_\Lambda^l \otimes m_{\widehat\Sigma}^l$:
    $$\msh_{\Lambda}(\Lambda) \otimes \Sh_{\widehat\Sigma}(N) \hookrightarrow \Sh_{\widehat\Lambda_{\pm\epsilon}}(M) \otimes \Sh_{\widehat\Sigma}(N) \twoheadrightarrow \Sh_{\widehat\Lambda_\epsilon}(M) \otimes \Sh_{\widehat\Sigma}(N).$$
    By K\"unneth formula Theorem \ref{pd:g} and Proposition \ref{prop:stop-removal-microlocalization}, the essential image of the inclusion is the corepresentatives of microstalk functors on $\Lambda_{-\epsilon} \times \widehat\Sigma$. Consider Theorem \ref{thm: doubling-product-piece}, we also have a recollement induced by the inclusion functor $m_{\Lambda \times \widehat\Sigma}^l$:
    $$\msh_{\widehat\Lambda \times \widehat\Sigma}(\Lambda \times \widehat\Sigma) \hookrightarrow \Sh_{\widehat\Lambda_{\pm\epsilon} \times \widehat\Sigma}(M \times N) \twoheadrightarrow \Sh_{\widehat\Lambda_{\epsilon} \times \widehat\Sigma}(M \times N).$$
    By Proposition \ref{prop:stop-removal-microlocalization}, we know that the essential image is equal to the corepresentatives of microstalk functors on $\Lambda_{-\epsilon} \times \widehat\Sigma$. Thus we get the isomorphism.
\end{proof}

Then, we prove the K\"unneth formula for microsheaves. Basically, in the product category of sheaves $\Sh_{\widehat\Lambda_\epsilon \times \widehat\Sigma_{\pm\epsilon}}(M \times N)$, using the recollement on each factor, we can find objects that come from the product of sheaves on both factors, from the product of sheaves and microsheaves, and finally from the product of microsheaves on both factors. The following proposition analyzes the objects that come from the product of sheaves and microsheaves.  See Figure \ref{fig:kunneth-microsheaf}.

\begin{figure}[h]
    \centering
    \includegraphics[width=0.9\textwidth]{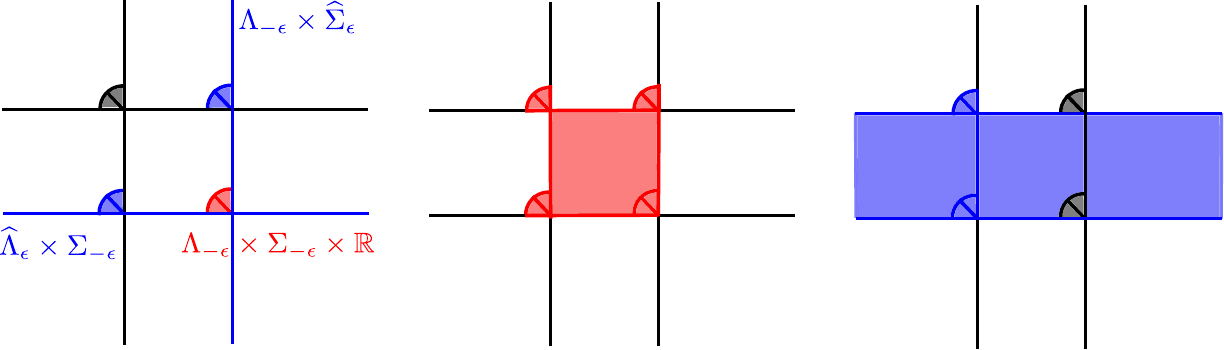}
    \caption{Different subcategories in $\Sh_{\widehat\Lambda_\epsilon \times \widehat\Sigma_{\pm\epsilon}}(M \times N)$ that come from the recollements. The red piece is the essential image of the microstalk corepresentatives on $\Lambda \times \Sigma \times \bR$. The blue pieces are the essential images of the microstalk corepresentatives of $\Lambda_{-\epsilon} \times \widehat\Sigma_{\epsilon}$ and $\widehat\Lambda_{\epsilon} \times \Sigma_{-\epsilon}$.}
    \label{fig:kunneth-microsheaf}
\end{figure}

\begin{proposition}\label{thm: doubling-mixed-piece}
    Let $\Lambda \subseteq S^*M, \Sigma \subseteq S^*N$ be compact subanalytic isotropic subsets. Then the left adjoint of microlocalization on $\Lambda_{-\epsilon} \times \widehat\Sigma_\epsilon$ is a fully faithful embedding
    $$m_{\Lambda \times \widehat\Sigma_\epsilon}^l: \msh_{\widehat\Lambda_{\pm\epsilon} \times \widehat\Sigma_{\pm\epsilon}}(\Lambda_{-\epsilon} \times \widehat\Sigma_\epsilon) \hookrightarrow \Sh_{(\widehat\Lambda_\epsilon \times \widehat\Sigma_{\pm\epsilon}) \cup (\widehat\Lambda_{\pm\epsilon} \times \widehat\Sigma_\epsilon)}(M \times N).$$
    In particular, there is a recollement of the form
    \begin{align*}
    \msh_{\widehat\Lambda_{\pm\epsilon} \times \widehat\Sigma_{\pm\epsilon}}(\Lambda_{-\epsilon} \times \widehat\Sigma_{\epsilon}) \hookrightarrow \Sh_{(\widehat\Lambda_\epsilon \times \widehat\Sigma_{\pm\epsilon}) \cup (\widehat\Lambda_{\pm\epsilon} \times \widehat\Sigma_\epsilon)}(M \times N) \twoheadrightarrow \Sh_{\widehat\Lambda_\epsilon \times \widehat\Sigma_{\pm\epsilon}}(M \times N).
    \end{align*}
    For the left adjoint of microlocalization on $\widehat\Lambda_\epsilon \times \Sigma_{-\epsilon}$, similar statement holds. In particular, 
    $$\Sh_{(\widehat\Lambda_\epsilon \times \widehat\Sigma_{\pm\epsilon}) \cup (\widehat\Lambda_{\pm\epsilon} \times \widehat\Sigma_\epsilon)}(M \times N) = \langle \Sh_{\widehat\Lambda_\epsilon \times \widehat\Sigma_{\pm\epsilon}}(M \times N), \Sh_{\widehat\Lambda_{\pm\epsilon} \times \widehat\Sigma_\epsilon}(M \times N) \rangle.$$
\end{proposition}
\begin{proof}
    Consider $\Lambda_{-\epsilon} \times \widehat\Sigma_{\epsilon}$ as a subset in $\widehat\Lambda_{\pm\epsilon} \times \widehat\Sigma_{\pm\epsilon}$. By Theorem \ref{thm: doubling-product-piece}, the left adjoint of microlocalization is fully faithful
    \begin{align*}
    m_{\Lambda_{-\epsilon} \times \widehat\Sigma_{\epsilon} \subseteq \widehat\Lambda_{\pm\epsilon} \times \widehat\Sigma_{\pm\epsilon}}^l: \msh_{\widehat\Lambda_{\pm\epsilon} \times \widehat\Sigma_{\pm\epsilon}}(\Lambda_{-\epsilon} \times \widehat\Sigma_{\pm\epsilon}) \hookrightarrow \Sh_{\widehat\Lambda_{\pm\epsilon} \times \widehat\Sigma_{\pm\epsilon}}(M \times N).
    \end{align*}
    Then, consider $\Lambda_{-\epsilon} \times \widehat\Sigma_{\epsilon}$ as a subset in $(\widehat\Lambda_\epsilon \times \widehat\Sigma_{\pm\epsilon}) \cup (\widehat\Lambda_{\pm\epsilon} \times \widehat\Sigma_\epsilon)$. By Proposition \ref{wdstopr}, we have a commutative diagram
    $$\resizebox{\textwidth}{!}{\xymatrix{
    \sD^\mu(\Lambda_{-\epsilon} \times \Sigma_{-\epsilon}, \Lambda_{-\epsilon} \times \widehat\Sigma_{\epsilon})\, \ar@{^{(}->}[r] \ar[d]^{\rotatebox{90}{$\sim$}} & \msh_{\widehat\Lambda_{\pm\epsilon} \times \widehat\Sigma_{\pm\epsilon}}(\Lambda_{-\epsilon} \times \widehat\Sigma_{\pm\epsilon}) \ar@{^{(}->}[d] \ar@{->>}[r] & \msh_{\widehat\Lambda_{\pm\epsilon} \times \widehat\Sigma_{\epsilon}}(\Lambda_{-\epsilon} \times \widehat\Sigma_{\epsilon}) \ar@{^{(}->}[d] \\
    \sD^\mu(\Lambda_{-\epsilon} \times \Sigma_{-\epsilon}, \widehat\Lambda_{\pm\epsilon} \times \widehat\Sigma_{\pm\epsilon})\, \ar@{^{(}->}[r] & \Sh_{\widehat\Lambda_{\pm\epsilon} \times \widehat\Sigma_{\pm\epsilon}}(M \times N) \ar@{->>}[r] & \Sh_{(\widehat\Lambda_\epsilon \times \widehat\Sigma_{\pm\epsilon}) \cup (\widehat\Lambda_{\pm\epsilon} \times \widehat\Sigma_\epsilon)}(M \times N).
    }}$$
    Here, by Proposition \ref{prop:stopremoval}, $\sD^\mu(\Lambda_{-\epsilon} \times \Sigma_{-\epsilon}, \Lambda_{-\epsilon} \times \widehat\Sigma_{\epsilon})$ and respectively $\sD^\mu(\Lambda_{-\epsilon} \times \Sigma_{-\epsilon}, \widehat\Lambda_{\pm\epsilon} \times \widehat\Sigma_{\pm\epsilon})$ are compactly generated by the corepresentatives of microstalks along $\Lambda_{-\epsilon} \times \Sigma_{-\epsilon}$. Since the left horizontal functors and middle vertical functor are both fully faithful, we know that the left vertical functor is an isomorphism. Then, since the right horizontal functors are localizations, we can conclude that the right vertical functor has to be fully faithful. Therefore, we get the fully faithful embedding. The recollement follows from Proposition \ref{prop:stop-removal-microlocalization}. 
    
    Finally, consider the recollemont in Theorem \ref{thm: doubling-product-piece} which fits into the following diagram
    $$\xymatrix{
    \msh_{\widehat\Lambda_{\pm\epsilon} \times \widehat\Sigma_{\pm\epsilon}}(\Lambda_{-\epsilon} \times \widehat\Sigma_{\epsilon}) \ar@{^{(}->}[r] \ar[d]^{\rotatebox{90}{$\sim$}}  & \Sh_{(\widehat\Lambda_\epsilon \times \widehat\Sigma_{\pm\epsilon}) \cup (\widehat\Lambda_{\pm\epsilon} \times \widehat\Sigma_\epsilon)}(M \times N) \ar@{->>}[r] \ar@{->>}[d]  & \Sh_{\widehat\Lambda_\epsilon \times \widehat\Sigma_{\pm\epsilon}}(M \times N)  \ar@{->>}[d]  \\
    \msh_{\widehat\Lambda_{\pm\epsilon} \times \widehat\Sigma_{\pm\epsilon}}(\Lambda_{-\epsilon} \times \widehat\Sigma_{\epsilon}) \ar@{^{(}->}[r] & \Sh_{\widehat\Lambda_{\pm\epsilon} \times \widehat\Sigma_\epsilon}(M \times N) \ar@{->>}[r] & \Sh_{\widehat\Lambda_\epsilon \times \widehat\Sigma_\epsilon}(M \times N).
    }$$
    Since fiber categories are identical, we can get a pull-back square of sheaf categories, where the functors are left adjoints of the inclusions. This shows the last statement.
\end{proof}

\begin{remark}
    Similar to Theorem \ref{thm: doubling-small-piece}, we can show that there exists some particularly nice choice of the relative Legendrian doubling $\hat{T}_{-\epsilon}(\widehat\Lambda \times \Sigma) \cup \hat{T}_{\epsilon}(\widehat\Lambda \times \Sigma)$ that is contained in a small neighbourhood in $(\widehat\Lambda_\epsilon \times \widehat\Sigma_{\pm\epsilon}) \cup (\widehat\Lambda_{\pm\epsilon} \times \widehat\Sigma_\epsilon)$.
\end{remark}

\begin{proof}[Proof of Theorem \ref{thm: Kunneth-Fourier-Mukai-microsheaves}~(1)]
By Theorem \ref{thm: doubling}, there are localization sequences
$$\msh_\Lambda(\Lambda) \hookrightarrow \Sh_{\Lambda_{-\epsilon} \cup \Lambda_\epsilon}(M) \twoheadrightarrow \Sh_{\Lambda_\epsilon}(M), \; \msh_\Sigma(\Sigma) \hookrightarrow \Sh_{\Sigma_{-\epsilon} \cup \Sigma_\epsilon}(N) \twoheadrightarrow \Sh_{\Sigma_\epsilon}(N).$$
By Corollary \ref{product-of-quotient-categories}, $\msh_\Lambda(\Lambda) \otimes \msh_\Sigma(\Sigma)$ is then the fiber of the functor
$$\langle \iota_{\Lambda_{\epsilon}}^* \otimes \id, \id \otimes \,\iota_{\Sigma_{\epsilon}}^*\, \rangle : \Sh_{\Lambda_{\pm\epsilon}}(M) \otimes \Sh_{\Sigma_{\pm\epsilon}}(N) \twoheadrightarrow \langle \Sh_{\Lambda_{\pm\epsilon}}(M) \otimes \Sh_{\Sigma_\epsilon}(N), \Sh_{\Lambda_\epsilon}(M) \otimes \Sh_{\Sigma_{\pm\epsilon}}(N) \rangle.$$
By the K\"unneth formula for sheaves and Proposition \ref{prop:stopremoval-kunneth},
we know that the functor is the localization functor
$$
\langle \iota_{\widehat\Lambda_\epsilon \times \widehat\Sigma_{\pm\epsilon}}^*, \iota_{\widehat\Lambda_{\pm\epsilon} \times \widehat\Sigma_\epsilon}^* \rangle : \Sh_{\widehat\Lambda_{\pm\epsilon} \times \widehat\Sigma_{\pm\epsilon}}(M \times N) \twoheadrightarrow \langle \Sh_{\widehat\Lambda_\epsilon \times \widehat\Sigma_{\pm\epsilon}}(M \times N), 
\Sh_{\widehat\Lambda_{\pm\epsilon} \times \widehat\Sigma_\epsilon}(M \times N) \rangle.
$$
Using Proposition \ref{thm: doubling-mixed-piece}, this is equivalent to the following localization functor:
$$\iota_{(\widehat\Lambda_\epsilon \times \widehat\Sigma_{\pm\epsilon}) \cup (\widehat\Lambda_{\pm\epsilon} \times \widehat\Sigma_\epsilon)}^*: \Sh_{(\widehat\Lambda_{\pm\epsilon} \times \widehat\Sigma_{\pm\epsilon})}(M \times N) \twoheadrightarrow \Sh_{(\widehat\Lambda_\epsilon \times \widehat\Sigma_{\pm\epsilon}) \cup (\widehat\Lambda_{\pm\epsilon} \times \widehat\Sigma_\epsilon)}(M \times N).$$
The complement is $(\widehat\Lambda_{\pm\epsilon} \times \widehat\Sigma_{\pm\epsilon}) \setminus (\widehat\Lambda_\epsilon \times \widehat\Sigma_{\pm\epsilon}) \cup (\widehat\Lambda_{\pm\epsilon} \times \widehat\Sigma_\epsilon) = (\Lambda \times \Sigma \times \bR) \times \bR_+$. Consider the left adjoints of the inclusion. Using Proposition \ref{prop:stop-removal-microlocalization}, we know that the fiber of the localization is the image of
$$m_{\Lambda \times \Sigma \times \bR}^l: \msh_{\widehat\Lambda_{\cup,\epsilon} \times \widehat\Sigma_{\cup, \epsilon}}(\Lambda \times \Sigma \times \bR) \to \Sh_{\widehat\Lambda_{\cup,\epsilon} \times \widehat\Sigma_{\cup,\epsilon}}(M \times N).$$
By Theorem \ref{thm: doubling-small-piece}, this is fully faithful. Hence we have
$$\msh_\Lambda(\Lambda) \otimes \msh_\Sigma(\Sigma) = \msh_{\widehat\Lambda_{\pm\epsilon} \times \widehat\Sigma_{\pm\epsilon}}(\Lambda \times \Sigma \times \bR).$$
which shows the K\"unneth formula for microsheaves with isotropic supports.
\end{proof}

\begin{proposition}\label{prop:double-product-equiv}
Let $\Lambda \subseteq S^*M, \Sigma \subseteq S^*N$ be compact subanalytic isotropic subsets. Then the left adjoint of microlocalization on $\Lambda \times \Sigma \times \bR$ induces an equivalence
$$m_{\Lambda \times \Sigma \times \bR}^l: \msh_{\widehat\Lambda \times \widehat\Sigma}(\Lambda \times \Sigma) \xrightarrow{\sim} \Sh_{\widehat\Lambda_{\cup,\epsilon} \times \widehat\Sigma_{\cup,\epsilon}}(M \times N). $$
Similarly, we have $m_{\Lambda \times \widehat\Sigma}^l: \msh_{\widehat\Lambda \times \widehat\Sigma}(\Lambda \times \widehat\Sigma) \xrightarrow{\sim} \Sh_{\widehat\Lambda_{\cup,\epsilon} \times \widehat\Sigma}(M \times N)$.
\end{proposition}

\begin{remark}\label{rem:kunneth-microsheaf}
The Proposition follows from the fact that the K\"unneth formula $\msh_\Lambda(\Lambda) \otimes \msh_\Sigma(\Sigma) \simeq \msh_{\Lambda \times \Sigma \times \bR}(\Lambda \times \Sigma)$ is induced by the functor
$$\msh_\Lambda(\Lambda) \otimes \msh_\Sigma(\Sigma) \xrightarrow{m_\Lambda^l \otimes \,m_\Sigma^l} \Sh_{\widehat\Lambda_{\cup,\epsilon} \times \widehat\Sigma_{\cup,\epsilon}}(M \times N) \xrightarrow{m_{\Lambda \times \Sigma}} \msh_{\Lambda \times \Sigma \times \bR}(\Lambda \times \Sigma).$$
\end{remark}

\section{Duality and the Fourier-Mukai property}\label{sec:VerdSerr}

In this section, we study dualizability, kernels and colimit preserving functors of sheaf categories with isotropic microsupport. We first exhibit an equivalence $\VD{\widehat\Lambda}: \Sh_{-\widehat\Lambda}^c(M)^{op} \xrightarrow{\sim} \Sh_{\widehat\Lambda}^c(M)$ for a manifold $M$ and a subanalytic conic isotropic $\widehat\Lambda \subseteq T^* M$, and, as a corollary, we obtain classification of colimit-preserving functors by sheaf kernels
$$ \Fun^L(\Sh_{\widehat\Lambda}(M),\Sh_{\widehat\Sigma}(N)) = \Sh_{-\widehat\Lambda \times \widehat\Sigma}(M \times N)$$
through convolutions for any such pairs $(M, \Lambda)$ and $(N,\Sigma)$. Although we first prove its existence through categorical methods, we will show that, in the case when $\hat{\Lambda} \supseteq 0_M$, it can also be obtained geometrically by using the model of wrapped sheaves for $\Sh_\Lambda^c(M)$ \cite{Kuo-wrapped-sheaves}. 

We also study the relation between this standard dual and the Verdier dual. Assume that $\widehat\Lambda$ has compact intersection with the zero section, in which case there is an inclusion $\Sh_{\widehat\Lambda}^b(M) \subseteq \Sh_{\widehat\Lambda}^c(M)$ of sheaves with perfect stalks into compact objects. We show that the classical Verdier duality $\VD{M}: \Sh_{-\widehat\Lambda}^b(M) \xrightarrow{\sim} \Sh_{\widehat\Lambda}^b(M)$, is related to the standard duality $\VD{\widehat\Lambda}$ by the following geometric construction. Let $\varphi_t$ be a sufficiently small positive contact flow on $S^*M$ such that $\alpha(\partial_t\varphi_t) > 0$, which defines an equivalence functor by Guillermou--Kashiwara--Schapira \cite{Guillermou-Kashiwara-Schapira} which we denote by 
$$K(\varphi_t) \circ (-): \Sh(M) \to \Sh(M), \;\;\; F \mapsto \varphi_t(F).$$
Denote by $S_{\widehat\Lambda}^+: \Sh_{\widehat\Lambda}(M) \rightarrow \Sh_{\widehat\Lambda}(M)$ the (positive) wrap-once functor $S_{\widehat\Lambda}^+(F) = \iota_{\widehat\Lambda}^* \circ \varphi_\epsilon(F)$ and by $S_{\widehat\Lambda}^-: \Sh_{\widehat\Lambda}(M) \rightarrow \Sh_{\widehat\Lambda}(M)$ the negative wrap-once functor $S_{\widehat\Lambda}^-(F) = \iota_{\widehat\Lambda}^! \circ \varphi_{-\epsilon}(F)$ (where the word wrapping comes from the interpretation of Equation (\ref{for: large-wrappings}) when $\widehat\Lambda$ contains the zero section, as shown by the first author in \cite{Kuo-wrapped-sheaves}). We will show that 
$$\VD{M}(F) =  S_{\widehat\Lambda}^- \circ \VD{\widehat\Lambda}(F) \otimes \omega_M$$ 
for $F \in \Sh_{-\widehat\Lambda}^b(M)$. 
We mention a similar question was previously studied, in the setting of Betti geometric Langlands program, in \cite{Arinkin-Gaitsgory-Kazhdan-Raskin-Rozenblyum-Varshavsky} (though the space they consider is a non-quasi-compact stack where it is hard to the duality).  
Assume $S_{\widehat\Lambda}^+$ is invertible, then we can extend the Verdier duality to $\Sh_{-\widehat\Lambda}^c(M) \to \Sh_{\widehat\Lambda}^c(M)$ by the formula one the right hand side. We show that the converse is also true in a sense (see Theorem \ref{converse-statement-Verdier} for a precise statement).

\subsection{Convolution of sheaves}

We recall the notion of convolution. 
Let $X_i$, $i = 1, 2, 3$, be locally compact Hausdorff topological spaces, and write $X_{ij} = X_i \times X_j$, for $i < j$,  $X_{123} = X_1 \times X_2 \times X_3 $, and $\pi_{ij}: X_{123} \rightarrow X_{ij}$ for the corresponding projections. 

\begin{definition}
For $F \in \Sh(X_{12})$, $G \in \Sh(X_{23})$, the \textit{convolution} is defined to be 
$$G \circ_{X_2} F \coloneqq {\pi_{13}}_! (\pi_{23}^* G \otimes \pi_{12}^* F ) \in \Sh(X_{13}).$$
\end{definition}

\begin{remark}
When there is no confusion what $X_2$ is, we will usually surpass the notation and simply write it as $G \circ F$. This is usually the case when $X_1 = \{*\}$, $X_2 = X$, and $X_3 = Y$ and we think of $X$ as the source and $Y$ as the target, $G \in \Sh(X \times Y)$ as a functor sending $F \in \Sh(X)$ to $G \circ F \in \Sh(Y)$. Note that from its expression, this functor is colimit-preserving. 
\end{remark}

\begin{lemma}[{\cite[Proposition 3.6.2]{KS}}]\label{convrad}
For a fixed $G \in \Sh(X_{23})$, the functor $G \circ (-) : \Sh(X_{12}) \rightarrow \Sh(X_{13})$ induced by convoluting with $G$ has a right adjoint, which we denote by $\sHom^\circ(G,-): \Sh(X_{13}) \rightarrow \Sh(X_{12})$, that is given by
\begin{equation}\label{cvr}
H \mapsto  {\pi_{12}}_* \sHom(\pi_{23}^* G  ,\pi_{13}^! H).
\end{equation}
\end{lemma}

We study the effect of convolving with sheaves with prescribed microsupport.

\begin{lemma}\label{mscomp}
Let $K \in \Sh(M \times N)$ and $Y$ be a conic closed subset of $T^* N$.
If the microsupport $\ms(K)$ is contained in $T^* M \times Y$, 
then $\ms(\pi_{2*} K)$ and  $\ms(\pi_{2!} K)$ are both contained in $Y$, where $\pi_1, \pi_2$ denote the projections from $M \times N$ to $M$ and $N$.
\end{lemma}

\begin{proof}
Standard microsupport estimation for pushforward requires the proper support condition \cite[Proposition 5.4.4.]{KS}.
We thus pick an increasing sequence of relative compact open set $\{U_i\}_{i \in \NN}$ of $M$ such that 
$M = \bigcup_{ i \in \NN} U_i$ and notice that the canonical map $\operatorname{colim}_{i \in \NN} K_{M \times U_i} \rightarrow K$
is an isomorphism.
Denote by ${\pi_2}_\pi$ the projection to the second component on the cotangent bundle and compute that 
\begin{align*}
\ms({p_2}_! H)
&=  \ms\big(\mathrm{colim}_{i \in \NN} \, \pi_{2!} K_{U_i \times N}\big) \subseteq \overline{ \bigcup{}_{i \in \NN} \ms(\pi_{2!} K_{U_i \times N})  } 
\subseteq \overline{ \bigcup{}_{i \in \NN} \ {\pi_2}_\pi (  \ms(K_{U_i \times N}) \cap M \times T^* N)  } \\
&\subseteq \overline{ \bigcup{}_{i \in \NN} \ {\pi_2}_\pi ( T^* M \times Y \cap M \times T^* N)  } \subseteq \overline{ \bigcup{}_{i \in \NN} \ {\pi_2}_\pi ( M \times  Y)  } \subseteq \overline{ \bigcup{}_{i \in \NN} Y  } = Y.
\end{align*}
To prove the case for $\pi_{2*}$ we further require that $U_i \subseteq \overline{U_i} \subseteq {U_{i+1}}$
and apply the same computation to the limit $L = \lim_{i \in \NN}  \Gamma_{\overline{U_i} \times N}(K)$. 
\end{proof}

\begin{proposition}\label{conv-ms}
Let $K \in \Sh_{T^* M \times Y}(M \times N)$. Then the assignment $F \mapsto K \circ F$ defines a functor
$$K \circ (-): \Sh(M) \rightarrow \Sh_Y(N).$$
\end{proposition}

\begin{proof}
We recall that, for $F \in \Sh(M)$, $K \circ F \coloneqq {\pi_N}_! (K \otimes \pi_M^* F)$.
The non-characteristic microsupport estimation \cite[Corollary 6.4.5]{KS} implies that
$$\ms(K \otimes \pi_M^* F) \subseteq \ms(K) \,\widehat{+}\, (\ms(F) \times M)
\subseteq (T^* M \times Y) \,\widehat{+}\, (T^* M \times 0_N).$$
Now the description of $\widehat{+}$ implies that if $(x,\xi,y,\eta)$ is a point on the right hand side, then it comes from a limiting point of a sum from $(x_n,\xi_n,y_n,\eta_n) \in T^* M \times Y$ and $(x^\prime_n,\xi^\prime_n,y^\prime_n,0) \in T^* M \times N$. Thus $(y,\eta) \in Y$, $\ms(K \otimes \pi_M^* F) \subseteq T^* M \times Y$, and the statement is implied by the last lemma.
\end{proof}

We will see in the next section that the above integral transform classifies all colimiting preserving
functors between categories of sheaves with isotropic microsupport.
Before we leave this section, we notice that for a conic closed subset $\widehat X \subseteq T^* M$,
we can take any $K \in \Sh(M \times M)$
and obtain a universal integral kernel $\iota_{-\widehat X \times \widehat X}^* (K) \in \Sh_{-\widehat X \times \widehat X}(M \times M)$.
By the above proposition, $\iota_{-\widehat X \times \widehat X}^*(K)$ defines a functor
$\iota_{-\widehat X \times \widehat X}^*(K) \circ (-) : \Sh(M) \rightarrow \Sh_{\widehat X}(M).$
On the other hand, we can consider similarly functors $\Sh(M) \rightarrow \Sh_{\widehat X}(M)$
which are defined by $F \mapsto \iota_{-\widehat X \times \widehat X}^*(K) \circ \iota_{\widehat X}^*(F)$
or $F \mapsto \iota_{\widehat X}^*( K \circ F)$.
The claim is that they are all the same.

\begin{lemma}\label{lad&ker}
The following functors $\Sh(M) \rightarrow \Sh_{\widehat X}(M)$ are equivalent to each other:
\begin{enumerate}
\item $F \mapsto \iota_{\widehat X}^*( K \circ \iota_{\widehat X}^* (F) )$,
\item $F \mapsto \iota_{-\widehat X \times \widehat X}^*(K) \circ F$,
\item $F \mapsto \iota_{-\widehat X \times \widehat X}^*(K) \circ \iota_{\widehat X}^*(F)$.
\end{enumerate}
In particular, $\iota_{\widehat X}^* (F) = \iota_{-\widehat X \times \widehat X}^* (1_{\Delta}) \circ F.$
\end{lemma}

\begin{proof}

We note all three of the expressions on the right hand side are in $\Sh_{\widehat X}(M)$ by the previous Proposition \ref{conv-ms},
and we can see directly that, by Lemma \ref{convrad} and the right adjoint version of Proposition \ref{conv-ms} that
\begin{align*}
\Hom(\iota_{-\widehat X \times \widehat X}^*(K) \circ F,G)
&= \Hom \big( F, \sHom^\circ( \iota_{-\widehat X \times \widehat X}^*(K), G)  \big) \\
&= \Hom \big(\iota_{\widehat X}^*(F),\sHom^\circ( \iota_{-\widehat X \times \widehat X}^*(K), G )\big) \\
&= \Hom\big(\iota_{-\widehat X \times \widehat X}^*(K) \circ \iota_{\widehat X}^*(F),G\big)
\end{align*}
for $F \in \Sh(M)$, $G \in \Sh_{\widehat X}(M)$.
Thus (ii) and (iii) are the same. 

Now we show that $\Hom(\iota_{-\widehat X \times \widehat X}^*(K) \circ F,G) = \Hom(\iota_{\widehat X}^*( K \circ \iota_{\widehat X}^* F),G)$ 
for $F \in \Sh(M)$, $G \in \Sh_{\widehat X}(M)$. 
We've seen that the left hand side is the same as $\Hom(F,\sHom^\circ(\iota_{-\widehat X \times \widehat X}^* (K),G) )$
and the second term in the Hom is in $\Sh_{\widehat X}(M)$. 
A similar computation will imply that the right hand side is the same as $\Hom(F, \iota_{\widehat X}^! \sHom^\circ(K,G) )$
and the second term in the Hom is again in $\Sh_{\widehat X}(M)$. 
This means that we can evaluate at $F \in \Sh_{\widehat X}(M)$ and prove the equality only for this case.
Assume such a case, so tautologically $F = \iota_{\widehat X}^* F$, and we compute that
\begin{align*}
\Hom(\iota_{-\widehat X \times \widehat X}^*(K) \circ F,G)  
&= \Hom \big(\iota_{-\widehat X \times \widehat X}^* (K), \sHom(\pi_1^* F, \pi_2^! G) \big) \\
&= \Hom \big( K, \sHom(\pi_1^* F, \pi_2^! G) \big) 
= \Hom \big(K \circ F,G\big) \\
&= \Hom \big(\iota_{\widehat X}^*( K \circ \iota_{\widehat X}^* (F)), G\big).
\end{align*}
Note that for the third equality, we use the fact that $\sHom(\pi_1^* F, \pi_2^! G) \in \Sh_{-\widehat X \times \widehat X}(M \times M)$ by \cite[Proposition 5.4.14]{KS}.
\end{proof}

\subsection{Dualizability of sheaves} \label{sec: dual}

Let $(\SC,\otimes,1_\SC)$ be a symmetric monoidal ($\infty$-)category.
We recall the notion of dualizability which we will use later.
\begin{definition}[{\cite[Definition 4.6.1.1]{Lurie-HA}}]\label{smdual}
An object $X$ in $\SC$ is dualizable if there exists $Y \in \SC$ and a unit and a counit
$\eta: 1_\SC \rightarrow Y \otimes X$ and $\epsilon: X \otimes Y \rightarrow 1_\SC$
such that the pair $(\eta,\epsilon)$ satisfies the standard triangle equality that the following compositions are identities
$$X \xrightarrow{\id_X \otimes \eta} X \otimes Y \otimes X \xrightarrow{\epsilon \otimes \id_X} X,$$
$$Y \xrightarrow{\eta \otimes \id_Y} Y \otimes X \otimes Y \xrightarrow{\id_Y \otimes \epsilon} Y.$$
\end{definition}
We note that these conditions uniquely classifies $(Y,\eta,\epsilon)$, so, when confusion is unlikely to occur, we will use the notation $(X^\vee, \eta_X, \epsilon_X)$ for ``the dual" of a dualizable $X$.

\begin{proposition}[{\cite[Chapter 1, Section 4.1.5]{Gaitsgory-Rozenblyum}}]\label{dual-gives-inHom}
When $X$ is dualizable, the functor $X^\vee \otimes (-)$ is both the right and left adjoint of $X \otimes (-)$ since $(X^\vee)^\vee =X$.
In particular, for any $Y \in \SC$, we have $\underline{Hom}(X,Y) = Y \otimes X^\vee$.
\end{proposition}
\begin{remark}[{\cite[Lemma 4.6.1.6]{Lurie-HA} \cite[Chapter 1, Section 4.3.2--4.3.3]{Gaitsgory-Rozenblyum} \cite[Equation (2.1) \& (2.2)]{Brav-Dyckerhoff2}}]\label{rem:dual-functor}
When $X$ and $Y$ are both dualizable, we thus have isomorphisms
$$\Hom(1_\SC, X^\vee \otimes Y) = \Hom(X, Y) = \Hom(X \otimes Y^\vee, 1_\SC),$$
where the morphism $f: X \to Y$ corresponds to 
\begin{align*}
\varphi_f: 1_\SC \xrightarrow{\eta_X} X \otimes X^\vee \xrightarrow{f \otimes \Id_{X^\vee}} Y \otimes X^\vee, \;\;
\psi_f: Y^\vee \otimes X \xrightarrow{\Id_{Y^\vee} \otimes f} Y^\vee \otimes Y \xrightarrow{\epsilon_Y} 1_\SC.
\end{align*}
In particular, under the equivalence 
$$\Fun^L(1_\SC, X \otimes X^\vee) = \Fun^L(X, X) = \Fun^L(X^\vee \otimes X, 1_\SC),$$
$\Id_X$ always corresponds to $\eta_X$ under the first isomorphism and $\epsilon_X$ under the second isomorphism.
\end{remark}

\begin{definition}
For a morphism $f: X \rightarrow Y$ between dualizable objects, the dual morphism $\ND{M}(F): Y^\vee \rightarrow X^\vee$ is defined to be the composition
$$Y^\vee \xrightarrow{\eta_X \otimes \Id_{Y^\vee}} X^\vee \otimes X \otimes Y^\vee \xrightarrow{\Id_{X^\vee} \otimes f \otimes\, \Id_{Y^\vee}} X^\vee \otimes Y \otimes Y^\vee \xrightarrow{\Id_{X^\vee} \otimes \,\epsilon_Y} X^\vee.$$
\end{definition}

The following lemma states that there is induced dualizability on retractions.
 
\begin{lemma}[{\cite[Lemma 2.2]{Kuo-Shende-Zhang-Hochschild-Tamarkin}}] \label{lem: duality-of-retraction}
For a duality pair $(X,X^\vee,\epsilon_X,\eta_X)$ in $\SC$,
let $e: X \rightarrow X$ be an idempotent which can be written as $X \xrightarrow{r} Y \xrightarrow{i} X$
for some inclusion $i$ and some retraction $r$. 
Assume that the dual idempotent $e^\vee: X^\vee \rightarrow X^\vee$ also splits to $X^\vee \xrightarrow{s} Z \xrightarrow{j} X^\vee$. 
Then the pair
\[\eta_Y \coloneqq (s \otimes r) \circ \eta_X : 1_{\SC} \rightarrow Z \otimes Y,\;\; \epsilon_Y \coloneqq \epsilon_X \circ (i \otimes j) : Y \otimes Z \rightarrow 1_{\SC}\]
exhibits $Z$ as the dual of $Y$. 
\end{lemma} 

The relevant proposition concerning dualizability which we need is the following:
Let $\sC \in \PrLcs$ be compactly generated.
Denote by $C = \sC^c$ its compact objects and by $\sC^\vee \coloneqq \Ind(C^{op})$ the Ind-completion of its opposite category.
We first mention that the proof of the Proposition \ref{cg:d} below implies that
$\sC^\vee \otimes \sC = \Fun^{ex}(C^{op} \otimes C,\cV) = \Fun^L(\sC^\vee \otimes \sC,\cV)$.
Here the superscript `ex' means exact functors and the `L' means colimit-preserving functors.
As a result, the Hom-pairing $\Hom_{C}: C^{op} \otimes C \rightarrow \cV$ induces a functor
$$\epsilon_\sC: \sC^\vee \otimes \sC \rightarrow \cV$$
by extending $\Hom_{C}$ to the Ind-completion.
On the other hand, as a functor from $C^{op} \otimes C$ to $\cV$,
it also defines an object in $\sC \otimes \sC^\vee $ by the above identification, which is equivalent to a functor
$$ \eta_\sC: \cV \rightarrow \sC \otimes \sC^\vee.$$

\begin{proposition}[{\cite[Proposition 4.10]{Hoyois-Scherotzke-Sibilla}}]\label{cg:d}
If $\sC \in \PrLst$ is compactly generated by $C$, then it is dualizable with respect to the tensor product $\otimes$ on $\PrLst$,
and the triple $(\sC^\vee,\eta_\sC,\epsilon_\sC)$ exhibits $\sC^\vee \coloneqq \Ind(C^{op})$ as a dual of it.
\end{proposition}

\begin{remark}
We note that when $C$ contains only one object, $\Fun^{ex}(C^{op} \otimes C,\cV)$ recover the classical notion of bi-modules.
As a result, the diagonal bimodule $C_\Delta$ is also often referred to as the identity bimodule $\Id_C$.
\end{remark}


Classifying colimiting-preserving functors shares a close relation with the notion of duality in Definition \ref{smdual}.
In the algebraic geometric setting, this is usually referred as Fourier-Mukai \cite{Ben-Zvi-Francis-Nadler}.
One strategy to prove such a theorem, inspired by an earlier result in the derived algebro-geometric setting \cite[Section 9]{Gaitsgory-IndCoh}, 
is that the evaluation and coevaluation should be given by some sort of diagonals geometrically.
The equivalence between such geometric diagonals and the categorical diagonals discussed in Proposition \ref{cg:d}, 
which is implied by the uniqueness of duals, will provide such a classification.

In our case, we denote by $\Delta: M \hookrightarrow M \times M$ the inclusion of the diagonal 
and by $p: M \rightarrow \pt$ the projection to a point.
By Theorem \ref{pd:g}, there is an identification 
$\Sh_{\widehat\Lambda}(M) \otimes \Sh_{-\widehat\Lambda}(M) = \Sh_{\widehat\Lambda \times -\widehat\Lambda}(M \times M)$.
Under this identification, we propose a duality data $(\eta,\epsilon)$ 
between $\Sh_{\widehat\Lambda}(M)$ and $\Sh_{-\widehat\Lambda}(M)$ in $\PrLst$ which is given by

\begin{equation}\label{formula: standard_duality_data}
\begin{split}
\epsilon &= p_! \Delta^* : \Sh_{-\widehat\Lambda \times \widehat\Lambda}(M \times M) \rightarrow \cV \\
\eta &= \iota_{\widehat\Lambda \times -\widehat\Lambda}^* \Delta_* p^*
: \cV \rightarrow \Sh_{\widehat\Lambda \times -\widehat\Lambda}(M \times M).
\end{split}
\end{equation}

Recall that we use
$\iota_{\widehat\Lambda \times -\widehat\Lambda}^*: \Sh(M \times M) \rightarrow \Sh_{\widehat\Lambda \times -\widehat\Lambda}(M \times M)$
to denote the left adjoint of the inclusion 
$\Sh_{\widehat\Lambda \times -\widehat\Lambda}(M \times M) \hookrightarrow \Sh(M \times M)$.
Note also that since $\cV$ is compactly generated by $1_\cV$,
the colimit-preserving functor $\eta$ is determined by its value on $1_\cV$ 
so we will abuse the notation and identify it with $\eta$.
In order to check the triangle equalities, we first identify $\id \otimes \,\epsilon$.

\begin{lemma}\label{ev}
Under the identification $$\Sh_{\widehat\Lambda}(M) \otimes \Sh_{-\widehat\Lambda \times \widehat\Lambda}(M \times M)
= \Sh_{\widehat\Lambda \times -\widehat\Lambda \times \widehat\Lambda}(M \times M \times M),$$ the functor 
$$\id \otimes \,\epsilon: \Sh_{\widehat\Lambda}(M) \otimes \Sh_{-\widehat\Lambda \times \widehat\Lambda}(M \times M) 
\rightarrow \Sh_{\widehat\Lambda}(M) \otimes \cV = \Sh_{\widehat\Lambda}(M)$$ 
is identified as the functor 
$$\pi_{1!} (\id \times \Delta)^*: \Sh_{\widehat\Lambda \times -\widehat\Lambda \times \widehat\Lambda}(M \times M \times M)
\rightarrow \Sh_{\widehat\Lambda}(M).$$
\end{lemma}

\begin{proof}
Since both of the functors are colimit-preserving and the categories are compactly generated,
it is sufficient to check that $\pi_{1!} (\id \times \Delta)^* \circ \boxtimes = \id \otimes (p_! \Delta^*)$
on pairs $(F,G)$ for $F \in \Sh_{\widehat\Lambda}^c(M)$ and $G \in \Sh_{-\widehat\Lambda \times \widehat\Lambda}^c(M \times M)$
by Lemma \ref{pgc,PrLst}.

Note that we do not need the compactness assumption for the following computation.
Let $\tilde\pi_1: M^3 \rightarrow M$ and $\tilde\pi_{23}: M^3 \rightarrow M^2$ denote the projections
$\tilde\pi_1(x,y,z) = x$ and $\tilde\pi_{23}(x,y,z) = (y,z)$.
We note that $\tilde\pi_1 \circ (\id \times \Delta) = \pi_1$ and $\tilde\pi_{23} \circ (\id \times \Delta) = \Delta \circ \pi_2$.
Thus,
\begin{align*}
\pi_{1!} (\id \times \Delta)^* (F \boxtimes G)
&= \pi_{1!} (\id \times \Delta)^*  (\tilde\pi_1^* F \otimes \tilde\pi_{23}^* G) = \pi_{1!}  (\pi_1^* F \otimes \pi_2^* \Delta^* G) \\
&=  F \otimes (\pi_{1!} \pi_2^* \Delta^* G) = F \otimes ( p^* p_!\Delta^* G) = F \otimes_{\cV} (p_! \Delta^* G).
\end{align*}
Here, we use the fact that $*$-pullback is compatible with $\otimes$ for the second equality, the projection formula for the third, and base change for the forth. The last equality is by definition the action of the coefficient category $\cV$ on $\Sh(M)$.
\end{proof}

\begin{remark}
A similar computation will imply that $\eta \otimes \id$ can be identify with
$$\iota_{\widehat\Lambda \times -\widehat\Lambda \times \widehat\Lambda}^* (1_\Delta \boxtimes -):
\Sh_{\widehat\Lambda}(M) \rightarrow \Sh_{\widehat\Lambda \times -\widehat\Lambda \times \widehat\Lambda}(M \times M \times M).$$
\end{remark}

Now we check the triangle equality $(\id \otimes \epsilon) \circ (\eta \otimes \id) = \id$.
In other words, we check that the composition of the following functors 
\begin{equation}\label{eq:duality}
\begin{tikzpicture}
\node at (0,1.5) {$ \Sh_{\widehat\Lambda}(M)$};
\node at (7,1.5) {$\Sh_{\widehat\Lambda \times -\widehat\Lambda}(M \times M) \otimes \Sh_\Lambda(M)$};
\node at (7,0) {$\Sh_{\widehat\Lambda \times -\widehat\Lambda \times \widehat\Lambda}(M \times M \times M)$};
\node at (13,0) {$\Sh_{\widehat\Lambda}(M)$};
\draw [->, thick] (0.9,1.5) -- (4.3,1.5) node [midway, above] 
{$(\iota_{\widehat\Lambda \times -\widehat\Lambda}^* \Delta_* p^*)\otimes \id$};
\draw [->, thick] (9.3,0) -- (12.1,0) node [midway, above] {${p_1}_! (\id \times \Delta)^*$};
\draw [->, thick] (7,1.2) -- (7,0.3) node [midway, right] {$\boxtimes$};
\end{tikzpicture}
\end{equation}
is the identity. The other triangle equality can be checked symmetrically.

\begin{proposition}\label{teholds}
The above equality Equation (\ref{eq:duality}) holds.
\end{proposition}

\begin{proof}
Let $F \in \Sh_\Lambda(M)$.
The composition of the first two arrows sends $(1_\cV,F)$ to 
$$\big((\iota_{\widehat\Lambda \times -\widehat\Lambda}^* \Delta_* p^*) \boxtimes \id \big) (1_\cV,F)
= (\iota_{\widehat\Lambda \times -\widehat\Lambda}^* \Delta_* 1_M) \boxtimes F.$$
Apply $\pi_{1!} (\id \times \Delta)^*$ and we obtain
$$\pi_{1!} (\id \times \Delta)^* 
\big( (\iota_{\widehat\Lambda \times -\widehat\Lambda}^* \Delta_* 1_M) \boxtimes F \big)
= \pi_{1!} \big( (\iota_{\widehat\Lambda \times -\widehat\Lambda}^* \Delta_* 1_M) \otimes \pi_2^* F \big).$$
To see that $\pi_{1!} \big( (\iota_{\widehat\Lambda \times -\widehat\Lambda}^* \Delta_* 1_M) \otimes \pi_2^* F \big) = F$,
we use the Yoneda lemma to evaluate at $\Hom(-,H)$ for $H \in \Sh_\Lambda(M)$
and compute that
\begin{align*}
\Hom\big(\pi_{1!} \big( (\iota_{\widehat\Lambda \times -\widehat\Lambda}^* \Delta_* 1_M) \otimes \pi_{1!} F \big),H\big)
&= \Hom\big(\iota_{\widehat\Lambda \times -\widehat\Lambda}^* \Delta_* 1_M, \sHom( \pi_2^* F,\pi_1^! H) \big)\\
&= \Hom\big( \Delta_* 1_M, \sHom( \pi_2^* F,\pi_1^! H)  \big) \\
&= \Hom\big( 1_M,  \Delta^! \sHom( \pi_2^* F,\pi_1^! H)  \big) \\
&= \Hom\left( 1_M,   \sHom( F,  H)  \right) = \Hom(F,H).
\end{align*}
For the second equality, we use the fact that $\ms(\pi_2^* F) = M \times \ms(F)$ and
$\ms(\pi_1^! H) = \ms(H) \times M$, which further implies the microsupport estimation
$$ \ms( \sHom( \pi_2^* F,\pi_1^! H) ) \subseteq  (\ms(H) \times M) + (M \times -\ms(F))$$
by \cite[Proposition 5.4.14]{KS}.
\end{proof}

Because duals are unique, there is an equivalence  $\Sh_{-\widehat\Lambda}(M) = \Sh_{\widehat\Lambda}(M)^\vee$.
By passing to compact objects, we obtain an equivalence on small categories:
\begin{definition}\label{def: standard_daulity}
We denote by $\SD{\widehat\Lambda}: \Sh_{-\widehat\Lambda}^c(M) \xrightarrow{\sim} \Sh_{\widehat\Lambda}^c(M)^{op}$ the equivalence whose Ind-completion induces the equivalence $\Sh_{-\widehat\Lambda}(M) = \Sh_\Lambda(M)^\vee$
associated to the duality data in Equation~(\ref{formula: standard_duality_data}) and call it the \textit{standard duality} associated to $\widehat\Lambda \subseteq T^*M$.
\end{definition}
Thus, there is a commutative diagram given by the counits:

$$
\begin{tikzpicture}
\node at (0,2.5) {$\Sh_{-\widehat\Lambda \times \widehat\Lambda}(M \times M)$};
\node at (5,2.5) {$\cV$};
\node at (0,1.5) {$\Sh_{-\widehat\Lambda}(M) \otimes \Sh_{\widehat\Lambda}(M)$};
\node at (0,0) {$\Sh_{\widehat\Lambda}(M)^\vee \otimes \Sh_{\widehat\Lambda}(M)$};
\node at (5,0) {$\cV$};

\draw [->, thick] (1.6,2.5) -- (4.6,2.5) node [midway, above] {$p_! \Delta^* $};
\draw [->, thick] (1.9,0) -- (4.6,0) node [midway, below] {$ \Hom(-,-)$};

\draw [double equal sign distance, thick] (0,2.2) -- (0,1.8) node [midway, left] {$ $}; 
\draw [double equal sign distance, thick] (0,1.2) -- (0,0.3) node [midway, left] {$\Ind(\SD{\widehat\Lambda}) \otimes \id$}; 
\draw [double equal sign distance, thick] (5,2.2) -- (5,0.3) node [midway, right] {$ $};
\end{tikzpicture}
$$
Here we abuse the notation and use $\Hom(-,-)$ to denote the functor induced by its Ind-completion.
In particular, for $F \in \Sh_{-\widehat\Lambda}^c(M)$ and $G \in \Sh_{\widehat\Lambda}(M)$, there is an identification
\begin{equation} \label{eq: duality-as-pairing}
\Hom(\SD{\widehat\Lambda} F, G) = p_! (F \otimes G).
\end{equation}
A consequence of this identification is that colimit-preserving functors are given by integral transforms, i.e.,
Theorem \ref{thm: fmt} discussed in the introduction.
We mention the following proof is adapted from \cite{Ben-Zvi-Francis-Nadler} where they study a similar
statement in the setting of algebraic geometry.

\begin{proof}[Proof of Theorem \ref{thm: fmt}]
The identification is a composition which follows from K\"unneth formula Theorem \ref{pd:g}, the duality formula Definition--Theorem \ref{def-thm: canonical-duality} and Remark \ref{rem:dual-functor}
$$ \Sh_{-\widehat\Lambda \times \widehat\Sigma}(M \times N) = \Sh_{-\widehat\Lambda}(M) \otimes \Sh_{\widehat\Sigma}(N) 
= \Sh_{\widehat\Lambda}(M)^\vee \otimes \Sh_{\widehat\Sigma}(N) = \Fun^L(\Sh_{\widehat\Lambda}(M),\Sh_{\widehat\Sigma}(N)).$$
By passing to left adjoint with respect to $\Fun^L$ as in Remark \ref{rem:dual-functor}, the map corresponds to a map of the form
$$ \Sh_{\widehat\Lambda}(M) \otimes \Sh_{-\widehat\Lambda \times \widehat\Sigma}(M \times N) = \Sh_{\widehat\Lambda \times -\widehat\Lambda \times \Sigma}(M \times M \times N) \rightarrow \Sh_{\widehat\Sigma}(N)$$
where the second arrow is given by the co-unit.
Thus, write $\tilde{\pi}_1: M \times M \times N \to M$ the projection to the first factor and $\tilde{\pi}_{23}: M \times M \times N \to M \times N$ the projection to the second and third factor. For $F \in \Sh_\Lambda(M)$ and $K \in \Sh_{-\widehat\Lambda \times \widehat\Sigma}(M \times N)$, the image in $\Sh_{\widehat\Sigma}(N)$ is given by
\begin{equation*}
(\epsilon_\Lambda \times \id) (F \boxtimes K) \coloneqq {\pi_2}_! (\Delta \times \id)^* ( {\tilde{\pi}_{23}}^* K \otimes \tilde{\pi}_1^* F) 
= {\pi_2}_! (K \otimes \pi_1^* F) \eqqcolon K \circ F. \qedhere
\end{equation*}
\end{proof}

\begin{corollary} \label{cor: dual-morphism-versus-kernel}
Denote by $v:M \times N \xrightarrow{\sim} N \times N$ the coordinate swapping map $v(x,y) = (y,x)$.
Then under the equivalence $\Fun^L(\Sh_\Lambda(M),\Sh_\Sigma(N)) = \Sh_{-\Lambda \times \Sigma}(M \times N)$ of Theorem \ref{thm: fmt}, passing to dual functors
$$(-)^\vee: \Fun^L(\Sh_\Lambda(M),\Sh_\Sigma(N)) = \Fun^L(\Sh_{-\Sigma}(N),\Sh_{-\Lambda}(M))$$
is realized by
$$ v^*: \Sh_{-\Lambda \times \Sigma}(M \times N) = \Sh_{\Sigma \times -\Lambda}(N \times N).$$
\end{corollary}

\begin{proof}
This is a standard exercise of six-functor formalism. 
\end{proof}

Using the doubling construction for microsheaves supported on isotropic subsets, we can immediately show that microsheaves are also dualizable.

\begin{corollary}
Let $\Lambda \subseteq S^*M$ be a compact subanalytic isotropic subset and $\widehat\Lambda_{\cup,\epsilon} \subseteq T^* M$ be the doubling defined in Theorem \ref{thm: doubling}. Then the triple $(\Sh_{-\widehat\Lambda_{\cup,\epsilon}}(M), \epsilon, \eta)$ where
\begin{align*}
\epsilon &= p_!\Delta^*: \Sh_{-\widehat\Lambda_{\cup,\epsilon} \times \widehat\Lambda_{\cup,\epsilon}}(M \times M) \to \cV, \\
\eta &= \iota_{\widehat\Lambda_{\cup,\epsilon} \times -\widehat\Lambda_{\cup,\epsilon}}^* \Delta_*p^*: \cV \to  \Sh_{\widehat\Lambda_{\cup,\epsilon} \times -\widehat\Lambda_{\cup,\epsilon}}(M \times M)
\end{align*}
exhibits $\Sh_{-\widehat\Lambda_{\cup,\epsilon}}(M)$ as the dual of $\Sh_{\widehat\Lambda_{\cup,\epsilon}}(M)$. Therefore, under the equivalence of Theorem \ref{thm: doubling}, $\msh_{-\Lambda}(-\Lambda)$ is the dual of $\msh_\Lambda(\Lambda)$.
\end{corollary}

\begin{remark}\label{rem:dual-microsheaf}
Using the K\"unneth formula for microsheaves and Remark \ref{rem:kunneth-microsheaf}, we can write down the duality data for $\msh_\Lambda(\Lambda)$ directly:
\begin{align*}
\epsilon &= p_!\Delta^*m_{\Lambda \times -\Lambda}^l: \msh_{-\Lambda \times \Lambda \times \bR}(-\Lambda \times \Lambda) \to \cV, \\
\eta &= m_{\Lambda \times -\Lambda} \iota_{\widehat\Lambda_{\cup,\epsilon} \times -\widehat\Lambda_{\cup,\epsilon}}^* \Delta_*p^*: \cV \to \msh_{\Lambda \times -\Lambda \times \bR}(\Lambda \times -\Lambda).
\end{align*}
\end{remark}

\begin{proof}[Proof of Theorem \ref{thm: Kunneth-Fourier-Mukai-microsheaves}~(2)]
By the K\"unneth formula and duality formula for microsheaves, we have
$$\Fun^L\left(\msh_\Lambda(\Lambda), \msh_\Sigma(\Sigma) \right) = \msh_\Lambda(\Lambda)^\vee \otimes \msh_\Sigma(\Sigma) = \msh_{-\Lambda}(-\Lambda) \otimes \msh_\Sigma(\Sigma) = \msh_{-\Lambda \times \Sigma \times \bR}(-\Lambda \times \Sigma).$$
We can write down the identification by Remark \ref{rem:dual-microsheaf}. For $F \in \msh_\Lambda(\Lambda)$ and $K \in \msh_{-\Lambda \times \Sigma \times \bR}(-\Lambda \times \Sigma)$, the image in $\msh_\Sigma(\Sigma)$ is given by
\begin{equation*}
m_\Sigma \pi_{2!}(m_{-\Lambda \times \Sigma}^l(K) \otimes \pi_1^* m_\Lambda^l(F)),
\end{equation*}
where $m_{-\Lambda \times \Sigma}^l: \msh_{-\Lambda \times \Sigma \times \bR}(-\Lambda \times \Sigma) \to \Sh_{-\widehat\Lambda_{\cup,\epsilon} \times \widehat\Sigma_{\cup,\epsilon}}(M \times N)$ is given in Proposition \ref{prop:double-product-equiv} and $m_\Lambda^l: \msh_{\Lambda}(\Lambda) \to \Sh_{\widehat\Lambda_{\cup,\epsilon}}(M)$ is given by the doubling functor in Theorem \ref{thm: doubling}.
\end{proof}

\subsection{Standard duality through wrappings} \label{sec: geometric-dual} 
In this section we discuss a symplecto-geometric way to construct the standard duality, defined in Definition \ref{def: standard_daulity}. For this discussion, we will fix a conic isotropic $\widehat\Lambda \subseteq T^* M$ that contains the zero section and its intersection with the cosphere bundle $\Lambda \subseteq S^*M$ and restrict ourselves to this case. In this setting, a category of wrapped sheaves $\wsh_\Lambda(M)$ is defined in \cite[Definition 4.1]{Kuo-wrapped-sheaves} via a construction parallel to that of a wrapped Fukaya category. Furthermore, it is shown that there is a canonical equivalence \cite[Theorem 1.3]{Kuo-wrapped-sheaves}
\begin{equation}\label{for: wrapped-equivalence}
\wrap_\Lambda^+: \wsh_\Lambda(M) \xrightarrow{\sim} \Sh_\Lambda(M)^c
\end{equation}
induced from the wrapping functors defined in Equation (\ref{for: large-wrappings}).
The main goal of this section is to show that the naive duality 
\begin{align}
\ND{M}: \Sh(M) &\rightarrow \Sh(M)^{op} \label{for: naive-dual} \\
F &\mapsto \ND{M}(F) \coloneqq \sHom(F,1_M) \notag
\end{align} 
induces a dual on $\msh_\Lambda(M)$ and it corresponds to the standard duality via (\ref{for: wrapped-equivalence}). 
 
We begin with recalling that, for any contact isotopy $\Phi: S^* M \times I \rightarrow S^* M$, where $I$ is an open interval containing $0$, there exists a unique Guillermou--Kashiwara--Schapira sheaf kernel $K(\Phi) \in \Sh(M \times M \times I)$, by \cite[Theorem 3.7]{Guillermou-Kashiwara-Schapira}, such that
\begin{enumerate}
\item $K(\Phi) |_{t = 0} = 1_\Delta$, and
\item $\msif(K(\Phi)) \subseteq \Lambda_\Phi$, the contact movie of $\Phi$.
\end{enumerate}
Furthermore, when $\Phi$ is positive, there exists a continuation map $K(\Phi) |_s \rightarrow K(\Phi) |_t$ for $s \leq t$ in $I$. Such GKS kernels then provides the notion of isotopies of sheaves via convolution, by setting, for $F \in \Sh(M)$, $\Phi(F) \coloneqq K(\Phi) \circ F \in \Sh(M \times I)$ and $F_t \coloneqq \Phi(F) |_t \in \Sh(M)$. Similarly, when $\Phi$ is positive, the continuation maps of $K(\Phi)$ induces those for $\Phi(F)$. We also use the notation $F^\Phi$ or $F^w$ when the exact isotopy is not important.

In principle, we would like to take the category, 
$$\{F \in \Sh(M) \big| \, \supp(F) \ \text{is compact,} \, \ms(F)   \ \text{is Lagrangian disjoint from} \ \Lambda, \text{and}\,  F_x \in \cV_0, \, \forall x \in M\},$$
and invert along continuation maps to obtain wrapped sheaves. However, for technical reason, essentially because of \cite[Theorem 8.4.2]{KS}, we have to further restrict to those $F$ such that $\ms(F)$ is subanalytic up to an isotopy, and we denote the resulting category by $\widetilde{\wsh_M}(M)$.
 
\begin{definition}
The category of wrapped sheaves $\wsh_\Lambda(M)$  away from $\Lambda$ is defined by
$$\wsh_\Lambda(M) \coloneqq \widetilde{\wsh_M}(M)/ \sC_\Lambda(M)$$
where $\sC_\Lambda(M)$ is the (small) stable category generated by cofibers of continuation maps
$$C \coloneqq \mathrm{Cofib}( F \xrightarrow{c} F^\varphi).$$
\end{definition} 

\begin{lemma} \label{lem: dual-GKS}
For a contact isotopy, $\Phi: S^* M \times I \rightarrow S^* M$, denote by $\Phi^a: S^* M \times I \rightarrow S^* M$ its conjugation with the antipodal map, i.e., $\Phi^a([x,\xi],t) = - \Phi ([x, -\xi], t)$. Then, we have
$$\sHom( K(\Phi), \omega_M \boxtimes 1_{M \times I}) = K(\Phi^a).$$
In particular, at each time-$t$ slice, we have $K(\Phi^a) |_t = \sHom(K(\Phi) |_t, p_1^* \omega_M)$ where
$p_1: M \times M \rightarrow M$ is given by $p_1(x,y) = x$.
\end{lemma}

\begin{proof}
We would like to check that the left hand side satisfies the same uniqueness conditions as $K(\Phi^a)$.
As explained in \cite{Gui} that $K(\Phi)$ is constructible with perfect stalks and thus, by \cite[Exercise V.13]{KS}, $\ms(\sHom( K(\Phi), \omega_M \boxtimes 1_{M \times I})) = - \ms(K(\Phi))$ and (2) of the uniqueness condition for $K(\Phi^a)$ follows from the observation that $- \Lambda_\Phi = \Lambda_{\Phi^a}$. To check the first condition, we compute that
$$
 \sHom( K(\Phi), \omega_M \boxtimes 1_{M \times I} ) \big|_0 = \sHom ( K(\Phi) \big|_0, (\omega_M \boxtimes 1_{M \times I}) \big|_0 )  = \sHom(1_\Delta, \omega_M \boxtimes 1_M)
$$
where we use the fact that $\Lambda_{\Phi^a} \cap N^*(M \times M \times \{0\}) = \varnothing$ to pass the $*$-restriction over $\sHom$ by \cite[Proposition 5.4.13]{KS}.
Denote by $p_i: M \times M \rightarrow M$ the projection to the $i$-th component, since the $p_i$'s are smooth, we have the base change $\omega_M \boxtimes 1_M = p_2^* \omega_M = p_1^! 1_M$, and we conclude that
$$
 \sHom( K(\Phi), \omega_M \boxtimes 1_{M \times I} ) = \sHom(1_\Delta, p_1^! 1_M) = \Delta_* \Delta^! p_1^!  1_M = 1_\Delta,
$$
which is the uniqueness condition (1).
\end{proof}

\begin{lemma} \label{lem: dual-versus-convolution-constructible}
Denote $\Phi$, $\Phi^a$ as the above Lemma \ref{lem: dual-GKS}. Let $F \in \Sh_{\RR-c}(M)^b_0$ be a (real) constructible sheaf with compact support and perfect stalk, then we have 
$$ \ND{M \times I}(K(\Phi) \circ F) = K(\Phi^a) \circ \ND{M}(F).$$
In particular, $\ND{M}$ sends continuation maps to continuation maps. 
\end{lemma}

\begin{proof}
Denote by $q_x$ and $q_{yt}$ the projections from $M \times M \times I$ to $M$ and $M \times I$ by $q_x(x,y,t) = x$ and $q_{yt}(x,y,t) = (y,t)$ and so $K(\Phi) \circ F \coloneqq {q_{yt}}_! ( K(\Phi) \otimes q_x^*F)$. We first compute that
\begin{align*}
 \ND{M \times I}(K(\Phi) \circ F)  
 &= \sHom({q_{yt}}_! ( K(\Phi) \otimes q_x^*F), 1_{M \times I}) = {q_{yt}}_* \sHom(K(\Phi) \otimes q_x^*F, q_{yt}^! 1_{M \times I}) \\
 &= {q_{yt}}_* \sHom(K(\Phi) \otimes q_x^*F, q_x^* \omega_M) =  {q_{yt}}_* \sHom(q_x^*F, \sHom(K(\Phi) , q_x^* \omega_M) ). 
\end{align*}
We remark that we have not yet used any assumption on $F$. Now, we apply the above Lemma \ref{lem: dual-GKS}, and conclude that $ \ND{M \times I}(K(\Phi) \circ F)  = {q_{yt}}_* \sHom(q_x^*F, K(\Phi^a))$. Since $\msif(K(\Phi^a)) \subseteq \Lambda_{\Phi^a}$ and $\ms(q_x^* F) = \ms(F) \times {M \times I}$, their microsupports do not intersect, and we can apply \cite[Proposition 5.4.14]{KS} to conclude that
$$  \sHom(q_x^*F, K(\Phi^a)) = \sHom(q_x^*F, 1_{M \times M \times I}) \otimes K(\Phi^a) = K(\Phi^a) \otimes q_x^* \ND{M}(F)$$ 
since $F$ is constructible with perfect stalks. Lastly, since $\supp(F)$ is compact, we have 
\begin{equation*}
\ND{M \times I}(K(\Phi) \circ F)  = {q_{yt}}_* \left( K(\Phi^a) \otimes q_x^* \ND{M}(F) \right) = {q_{yt}}_! \left( K(\Phi^a) \otimes q_x^* \ND{M}(F) \right) \eqqcolon K(\Phi^a) \circ \ND{M}(F). \qedhere
\end{equation*}
\end{proof}


\begin{proposition}
The naive duality $\ND{M}: \Sh(M) \rightarrow \Sh(M)^\vee$ defined in (\ref{for: naive-dual}) induces an anti-equivalence, \begin{align*}
\WD{\Lambda}: \wsh_\Lambda(M) &\xrightarrow{\sim} \wsh_\Lambda(M)^{op} \\
F &\mapsto \ND{M}(F).
\end{align*}
\end{proposition}

\begin{proof}
By \cite[Theorem 8.4.2]{KS}, objects in $\widetilde{\wsh}_\Lambda(M) \subseteq \Sh_{\RR-c}(M)^b$, (real) constructible sheaves with perfect stalks, where the naive duality restricts to an equivalence 
$$\ND{M}: \Sh_{\RR-c}(M)^b \xrightarrow{\sim} \Sh_{\RR-c}(M)^{b,op}$$
by \cite[Proposition 3.4.3]{KS}.
Thus, but the above Lemma \ref{lem: dual-versus-convolution-constructible} implies that $\ND{M}$ sends $\widetilde{\wsh}_\Lambda(M)$ to $\widetilde{\wsh}_{-\Lambda}(M)$ and $\sC_\Lambda(M)$ to $\sC_{-\Lambda}(M)$, and thus it descends to an equivalence
\begin{equation*}
\WD{\Lambda}: \wsh_\Lambda(M) \xrightarrow{\sim} \wsh_\Lambda(M)^{op}. \qedhere
\end{equation*}
\end{proof}
 
\begin{theorem} \label{thm: standard-dual-by-wrappings}
There is a commutative diagram, consisting of equivalences:
$$
\begin{tikzpicture}
\node at (0,2) {$\wsh_\Lambda(M)$};
\node at (5,2) {$\Sh_\Lambda(M)^c$};
\node at (0,0) {$\wsh_{-\Lambda}(M)^{op}$};
\node at (5,0) {$\Sh_{-\Lambda}(M)^{c,op}$};

\draw [->, thick] (1.1,2) -- (3.9,2) node [midway, above] {$\wrap_\Lambda^+$};
\draw [->, thick] (1.1,0) -- (3.9,0) node [midway, above] {$(\wrap_\Lambda^+)^{op}$};

\draw [->, thick] (0,1.7) -- (0,0.3) node [midway, right] {$\WD{\Lambda}$}; 
\draw [->, thick] (5,1.7) -- (5,0.3) node [midway, right] {$D_\Lambda$};
\end{tikzpicture}
$$ 

\end{theorem}

\begin{proof} 
We've explained that all the functors in the diagrams are equivalences. Since $\WD{\Lambda}$ is induced from the restriction of $\ND{M}$ to $\widetilde{\wsh}_\Lambda(M)$, it is sufficient to exhibit commutativity for the the following diagram:

$$
\begin{tikzpicture}
\node at (0,2) {$\wsh_\Lambda(M)$};
\node at (5,2) {$\Sh_\Lambda(M)^c$};
\node at (0,0) {$\wsh_{-\Lambda}(M)^{op}$};
\node at (5,0) {$\Sh_{-\Lambda}(M)^{c,op}$};

\draw [->, thick] (1.1,2) -- (3.9,2) node [midway, above] {$\iota_\Lambda^*$};
\draw [->, thick] (1.1,0) -- (3.9,0) node [midway, above] {$(\iota_\Lambda^*)^{op}$};

\draw [->, thick] (0,1.7) -- (0,0.3) node [midway, right] {$\ND{M}$}; 
\draw [->, thick] (5,1.7) -- (5,0.3) node [midway, right] {$D_\Lambda$};
\end{tikzpicture}
$$ 

We remark that, although $\iota_\Lambda^* = \wrap_\Lambda^+$, we use the first expression to emphasize that the rest of the proof is completely categorical. Let $F \in \wsh_\Lambda(M)$ and $G \in \Sh_{-\Lambda}(M)^c$. Since $\Hom(\SD{\Lambda} \iota_\Lambda^*(F) , G) = p_!\left( \iota_\Lambda^*(F) \otimes G \right)$, we will pair the latter with $V \in \cV$, and compute that
\begin{align*}
\Hom(p_!\left( \iota_\Lambda^*(F) \otimes G \right), V) 
&= \Hom( \iota_\Lambda^*(F)  , \sHom(G, p^! V) ) = \Hom(F, \sHom(G,p^! V)) \\
&= \Hom(p_! (F \otimes G), V) = \Hom(p_*(F \otimes G), V).
\end{align*}
Here, we use the fact that, since $p^! V$ is a local system, $\ms(\sHom(G,p^!(V)) \subseteq - \ms(G) \subseteq \Lambda$ for the second equality.  Note also that we use the fact that $\supp(F)$ is compact in the last equality. But then, we recall that $\ms(F) \cap \Lambda = \varnothing$ and thus $\ms(\ND{M}(F)) \cap \ms(G) \subseteq \ms(\ND{M}(F)) \cap \left(-\Lambda \right) =\varnothing$, and \cite[Proposition 5.4.14]{KS} applies. Thus, we can compute that
\begin{align*}
\Hom(\SD{\Lambda}\iota_\Lambda^*(F), G) &= p_*(F \otimes G) = \Hom(1_M, \sHom(\ND{M}(F),1_M) \otimes  G) \\
&=  \Hom(1_M, \sHom(\ND{M}(F), G)  ) 
= \Hom(\iota_\Lambda^*(\ND{M}(F)),G). \qedhere
\end{align*}
\end{proof}

\subsection{Verdier duality and Serre functor}\label{sec:verdierdual}
In this section, we assume that $\widehat\Lambda \subseteq T^* M$ has compact intersection with the zero section.
We will compare the duality $\SD{\widehat\Lambda}: \Sh_{-\widehat\Lambda}^c(M) \xrightarrow{\sim} \Sh_{\widehat\Lambda}^c(M)^{op}$ we obtain from the last subsection with the more classical Verdier duality.
Recall that, for a locally compact Hausdorff space $X$, the Verdier duality is a functor
\begin{align*}
\VD{M}: \Sh(X) &\rightarrow \Sh(X)^{op} \\
F &\mapsto \sHom(F,\omega_X)
\end{align*}
where $\omega_X \coloneqq p^! 1_\cV$ is the dualizing sheaf of $X$.
We note that when $X = M$ is a $C^1$-manifold of dimension $n$, $\omega_M$ is an invertible local system. We will discuss this in more details in \cite{Kuo-Li-Calabi-Yau} following the formulation of Volpe \cite{Volpe-six-operations}.

We also note that $\VD{M}$ is not an equivalence on the (large) category $\Sh(M)$ so we have to restrict to smaller categories.
Recall that we use the notation $\Sh_{\widehat\Lambda}^b(M)$ to denote the full subcategory of $\Sh_{\widehat\Lambda}(M)$ consisting of sheaves with perfect stalks.
In this case, the Verdier dual
\begin{align*}
\VD{M}: \Sh_{\widehat\Lambda}^b(M)^{op} &\xrightarrow{\sim} \Sh_{-\widehat\Lambda}^b(M) \\
F &\mapsto \VD{M}(F) \coloneqq \sHom(F,\omega_M)
\end{align*}
is an equivalence since the double dual $F \rightarrow \VD{M} (\VD{M} (F))$ is an isomorphism by \cite[Proposition 3.4.3]{KS}.
From now on, we assume $M$ is compact for the rest of the subsection. Then by Proposition \ref{mc=cc} and Proposition \ref{prop:stopremoval}, we have $\Sh_{\widehat\Lambda}^b(M) \subseteq \Sh_{\widehat\Lambda}^c(M)$. One can then ask what is the relation between $\SD{\widehat\Lambda}$ and $\VD{M}$.
We will explicitly use the following consequence from the rigidity assumption on $\cV$ in the computation.

\begin{lemma}[{\cite[Proposition 4.9]{Hoyois-Scherotzke-Sibilla}}]
Assume $\cV_0$ is a rigid symmetric monoidal category.
Then there is a canonical equivalence of symmetric monoidal $\infty$-categories
$$\cV_0 \rightarrow \cV_0^{op}, \, X \mapsto X^\vee \coloneqq \Hom(X,1_{\cV}).$$
In particular, $(X^\vee)^\vee = X$.
\end{lemma}

First, we recall the perturbation lemma \cite[Section 4.1]{Kuo-Li-spherical} which explains the effect of a positive contact isotopy, which will be crucial later.

\begin{proposition}[Perturbation lemma]\label{prop:perturbation}
Let $\widehat\Lambda \subseteq T^*M$ be a subanalytic isotropic that has compact intersection with the zero section. Let $T_\epsilon$ be any positive contact push-off displacing $\Lambda$ from itself. Then for $F, G \in \Sh_{\Lambda}(M)$ such that $\mathrm{supp}(F) \cap \mathrm{supp}(G)$ is compact in $M$,
$$\Hom(F, G) \simeq \Hom({F}, T_\epsilon{G}).$$
\end{proposition}

Recall from the introduction of this section that the wrap-once functors are defined by
$$S_{\widehat\Lambda}^+(F) = \iota_{\widehat\Lambda}^* \circ \varphi_\epsilon(F), \;\;\; S_{\widehat\Lambda}^-(F) = \iota_{\widehat\Lambda}^! \circ \varphi_{-\epsilon}(F).$$
When $\widehat\Lambda$ contains the zero section, they agree with the functors $S_\Lambda^+$ and $S_\Lambda^-$ which are introduced in \cite[Section 4.2]{Kuo-Li-spherical}. In general, we have the following relation:

\begin{corollary} \label{cor: inverse-Serre}
Let $\widehat\Lambda \subseteq T^*M$ be a subanalytic isotropic. Denote by $\iota_{0*}: \Sh_{\widehat\Lambda}(M) \to \Sh_{M \cup\widehat\Lambda}(M)$ the inclusion functor with left adjoint $\iota_0^*$ and right adjoint $\iota_0^!$.
Then the wrap-once functor and negative wrap-once functor on $\Sh_{\widehat\Lambda}(M)$ is given by
\begin{align*}
S_{\widehat\Lambda}^+(F) = \iota_0^* S_\Lambda^+(F), \; S_{\widehat\Lambda}^-(F) = \iota_0^! S_\Lambda^-(F).
\end{align*} 
\end{corollary}

\begin{remark} \label{rmk: wrap-once}
The functor $S_\Lambda^+$ does not depend on the choice of the push-off. In fact, the authors give a characterization of $S_\Lambda^+$, in terms of spherical adjunctions,
as the dual cotwist associated to the spherical adjunction of microlocalization $\Sh_\Lambda(M) \rightarrow \msh_\Lambda(\Lambda)$. See \cite[Section 5]{Kuo-Li-spherical} for a detailed discussion.
We note also that its right adjoint $S_\Lambda^-$ admits a similar definition and characterization.
\end{remark}

We also recall the result on the Serre duality induced by wrapping around once functor $S_{\widehat\Lambda}^+$ or negative wrap-once functor $S_{\widehat\Lambda}^-$. We will further investigate the wrap-once functor later on and show that this is the inverse dualizing functor \cite{Kuo-Li-Calabi-Yau}.

\begin{proposition}[Sabloff--Serr{e} duality {\cite[Proposition 4.10]{Kuo-Li-spherical}}]\label{prop: sab-serre}
Let $\widehat\Lambda \subseteq T^*M$ be a subanalytic isotropic that has compact intersection with the zero section. Let $T_\epsilon$ be any small positive contact push-off displacing $\Lambda$ from itself and $T_{-\epsilon}$ the inverse negative contact push-off. Then for $F \in \Sh_{\Lambda}^b(M), G \in \Sh_{\Lambda}(M)$ such that $\mathrm{supp}(F) \cap \mathrm{supp}(G)$ is compact in $M$,
$$\Hom(T_\epsilon{F}, {G} \otimes \omega_M) \simeq \Hom({F}, T_{-\epsilon}{G} \otimes \omega_M) \simeq p_!(\VD{M}(F) \otimes G) \simeq \Hom({G, F})^\vee,$$
where we use the notation $p: M \rightarrow \{*\}$. In particular, 
$$\Hom(S_{\widehat\Lambda}^+({F}), {G} \otimes \omega_M) \simeq \Hom({F}, S_{\widehat\Lambda}^-({G} \otimes \omega_M)) \simeq p_!(\VD{M}(F) \otimes G) \simeq \Hom({G, F})^\vee.$$
\end{proposition}

The duality in terms of positive Hamiltonian push-off was first established in the context of Legendrian contact homology by Sabloff and Ekholm--Etnyre--Sabloff \cite{Sabduality,EESduality}. Here, $S_{\widehat\Lambda}^+$ plays the role of the inverse Serre functor while $S_{\widehat\Lambda}^-$ plays the role of the Serre functor. We emphasize however that it is not always true that $S_{\widehat\Lambda}^\pm$ sends $\Sh_{\widehat\Lambda}^b(M)$ to $\Sh_{\widehat\Lambda}^b(M)$; see \cite[Section 4.2 \& 5.5]{Kuo-Li-spherical}.

The following proposition shows that $\SD{\widehat\Lambda}$ and $\VD{M}$ are related by the inverse Serr{e} functor in Corollary \ref{cor: inverse-Serre}. Recall that $\ND{M}(F) \coloneqq \sHom(F, 1_M)$.

\begin{proposition}\label{identifying_Verdier_with_standard}
For $F \in \Sh_{-\widehat\Lambda}^b(M)$, $\SD{\widehat\Lambda}(F) = S_{\widehat\Lambda}^+ \circ \ND{M}(F) = S_{\widehat\Lambda}^+\circ  \VD{M}(F) \otimes \omega_M^{-1}$.
\end{proposition}

\begin{proof}
Since $F$ has perfect stalk, we have $\ND{M} (\ND{M}F) \otimes G = F \otimes G$ by \cite[Proposition 3.4.4]{KS}. Let $G \in \Sh_{\widehat\Lambda}(M)$. Then by Proposition \ref{prop: sab-serre} 
\begin{equation*}
    \Hom(S_{\widehat\Lambda}^+ \circ \ND{M}(F), G) = p_*(\ND{M} (\ND{M}F) \otimes G) = p_* (F \otimes G) = \Hom(\SD{\widehat\Lambda}(F),G). \qedhere
\end{equation*}
\end{proof}

\begin{remark}\label{rem:kernelinverse}
We remark that, for a contact isotopy $\Phi: S^* M \times I \rightarrow S^* M$ with time-1 flow $\varphi: S^*M \to S^*M$, using exactly the same argument in Lemma \ref{lem: dual-GKS}, we can compute that
$$K(\Phi^{-1}) \circ|_I\, \sHom( K(\Phi), 1_{M \times M \times I}) = \Delta_*\omega_M^{-1} \boxtimes 1_I.$$
Since this sheaf is microsupported in $T^*(M \times M) \times I$, we can see using the same argument in Lemma \ref{lem: dual-versus-convolution-constructible} that $K(\varphi) \circ \ND{M}(F) = \ND{M}(K(\varphi^{-1}) \circ F)$.
As a consequence, $S^+_{\widehat\Lambda} \circ \ND{M}(F) = \ND{M}(S^-_{\widehat\Lambda} F)$ and $\VD{M}(F) = \ND{M}(F) \otimes \omega_M = S_{\widehat\Lambda}^- \circ \SD{\widehat\Lambda}(F) \otimes \omega_M$. 
\end{remark}

Note that if we assume $S_{\widehat\Lambda}^+$ is invertible, or equivalently its right adjoint $S_{\widehat\Lambda}^-$ is its inverse, then Proposition \ref{identifying_Verdier_with_standard} implies that the equivalence $\VD{M}: \Sh_{-\widehat\Lambda}^b(M)^{op} \xrightarrow{\sim} \Sh_{\widehat\Lambda}^b(M)$ can be extended to $\Sh_{\widehat\Lambda}^c(M)^{op} \xrightarrow{\sim} \Sh_{\widehat\Lambda}^c(M)$ as 
$$\left(F \mapsto S_{\widehat\Lambda}^- \circ \SD{\widehat\Lambda}(F) \otimes \omega_M\right).$$
Taking Ind-completion and we obtain an identification $\Sh_{-\widehat\Lambda}(M)^\vee  \cong \Sh_{\widehat\Lambda}(M)$, which further provides a duality pair $(\epsilon^V,\eta^V)$ as in Formula (\ref{formula: standard_duality_data}). 

\begin{lemma} \label{lem: Verdier-versus-canonical-dual}
Assume $S_{\widehat\Lambda}^+$ is invertible so that $(S_{\widehat\Lambda}^+)^{-1} = S_{\widehat\Lambda}^-$.
Then the co-unit $\epsilon^V$ is given by
$$ \epsilon^V = p_* \Delta^!: \Sh_{-\widehat\Lambda \times \widehat\Lambda}(M \times M) \rightarrow \cV.$$
\end{lemma}

\begin{proof}
It is sufficient to show, for $F \in \Sh_{-\widehat\Lambda}^c(M)^{op}$ and $G \in \Sh_{\widehat\Lambda}(M)$, there is an identification
$$\Hom\big(S_{\widehat\Lambda}^- \circ \SD{\widehat\Lambda}(F)  \otimes \omega_M,G \big) = p_* \Delta^!(F \boxtimes G).$$
Since $S_{\widehat\Lambda}^+$ is the inverse of $S_{\widehat\Lambda}^-$, the left hand side is given by
\begin{align*}
\Hom(S_{\widehat\Lambda}^- \circ \SD{\widehat\Lambda}(F) \otimes \omega_M,G ) 
= \Hom( \SD{\widehat\Lambda}(F) \otimes \omega_M, S_{\widehat\Lambda}^+ (G) )
= p_* (F \otimes S_{\widehat\Lambda}^+ (G) \otimes \omega_M^{-1} ).
\end{align*}

We first simplify the last expression to $p_* (F \otimes T_\epsilon G \otimes \omega_M^{-1})$, where $T_\epsilon$ means some small positive contact push-off. Indeed, for any $A \in \cV$, one computes that
\begin{align*}
&\Hom(p_* (F \otimes S_{\widehat\Lambda}^+ (G) \otimes \omega_M^{-1} ), A) = \Hom (S_{\widehat\Lambda}^+ (G), \sHom(F, p^! A \otimes \omega_M) ) \\
&= \Hom (T_\epsilon G, \sHom(F, p^!  A \otimes \omega_M ) ) = \Hom(p_* (F \otimes T_\epsilon G \otimes \omega_M^{-1} ), A)
\end{align*}
where we use the fact that $F$ is compactly supported in $-\widehat\Lambda$ for the second equality, so that $\iota_{\widehat\Lambda*}\sHom(F, p^! A \otimes \omega_M ) = \sHom(F, p^!  A \otimes \omega_M )$.

Since $\msif((T_\epsilon G)^\varphi) \cap \Lambda = \varnothing$,  we use \cite[Proposition 5.4.13]{KS} and \cite[Lemma 2.41]{Kuo-wrapped-sheaves} and show that
$$p_*(\Delta^*(F \boxtimes T_\epsilon G) \otimes \omega_M^{-1}) = p_*\Delta^! \left( F \boxtimes G^w \right) = \Hom(1_\Delta, F \boxtimes T_\epsilon G).$$
But then, we see the right hand side is $\Hom(1_\Delta, F \boxtimes G)$, 
since we have the constancy condition for the perturbation trick Proposition \ref{prop:perturbation}.
Thus, 
\begin{equation*}
    \Hom(S_{\widehat\Lambda}^- \circ \SD{\widehat\Lambda}(F \otimes \omega_M^{-1}),G ) 
= \Hom(1_\Delta, F \boxtimes G) = p_* \Delta^! (F \boxtimes G). \qedhere
\end{equation*}
\end{proof}

The main theorem of this section is that the converse is also true:
\begin{proof}[Proof of Theorem \ref{converse-statement-Verdier}]
The functor $p_* \Delta^!$ is colimit-preserving since we can displace the microsupport uniformly from $S^*_\Delta(M \times M)$ and trade $\Delta^!$ with $\Delta^*$:
$$p_*\Delta^!(F \boxtimes G) = p_*\Delta^!(F \boxtimes T_\epsilon G) = p_*(\Delta^*(F \boxtimes T_\epsilon G) \otimes \omega_M^{-1}).$$

The pair $(\epsilon, \eta) \coloneqq (\epsilon^V,\eta^V)$ gives a duality data if $(\epsilon \otimes \id) \circ (\id \otimes \eta) = \id$ and $(\id \otimes \epsilon) \circ (\eta \otimes \id) = \id$.
Let $K(T_\epsilon)$ (resp.~$K(T_{-\epsilon})$) be the sheaf quantization of any positive contact push-off $T_\epsilon$ (resp.~inverse contact push-off $T_{-\epsilon}$) at a sufficiently small time $\epsilon > 0$, in the sense of Guillermou--Kashiwara--Schapira \cite{Guillermou-Kashiwara-Schapira}. We note that the functor $(\epsilon \otimes \id): \Sh_{-\widehat\Lambda \times \widehat\Lambda}(M \times M)  \otimes \Sh_{-\widehat\Lambda}(M)$ under the identification $\Sh_{-\widehat\Lambda \times \widehat\Lambda}(M \times M) \otimes \Sh_{-\widehat\Lambda}(M) = \Sh_{-\widehat\Lambda \times \widehat\Lambda \times -\widehat\Lambda}(M \times M \times M)$ is given by
\begin{align*}
\Sh_{-\widehat\Lambda \times \widehat\Lambda \times -\widehat\Lambda}(M \times M \times M) &\rightarrow \Sh_{-\widehat\Lambda}(M) \\
A &\mapsto \tilde\pi_{3*}\left(\sHom\left(\tilde\pi_{12}^* K(\Phi^{-1}),  A \right) \right). 
\end{align*}
Indeed, it is sufficient to check for $A = H \boxtimes F$ for $H \in \Sh_{-\widehat\Lambda \times \widehat\Lambda}(M \times M)$ and $F \in \Sh_{-\widehat\Lambda}(M)$ by Theorem \ref{pd:g}, 
since the expression on the right is colimit preserving for a similar reason why $p_* \Delta^!$ is.
That is, $(\epsilon \otimes \id)$ sends the pair $(H,F)$ to
\begin{align*}
\left( p^* \Hom(K(T_{-\epsilon}),H)\right) \otimes  F
&= \left( \tilde\pi_{3*} \tilde\pi_{12}^* \sHom\left(K(T_{-\epsilon}),H\right) \right)  \otimes  F \\
&= \tilde\pi_{3*} \left(  \sHom\left(\tilde\pi_{12}^* K(T_{-\epsilon}),\tilde\pi_{12}^* H\right)  \otimes \tilde\pi_3^* F \right) \\
&= \tilde\pi_{3*} \left(   \sHom\left(\tilde\pi_{12}^* K(T_{-\epsilon}),  H \boxtimes F \right)  \right).
\end{align*}
Here, we use base change for the second equality, the projection formula for the second equality, and the fact that $K(T_{-\epsilon})$ has perfect stalks and is microsupported away from $H$ at infinity for the last equality \cite[Proposition 5.4.14]{KS}. Now to compute the endo-functor $(\epsilon \otimes \id) \circ (\id \otimes \eta)$ on $\Sh_{-\widehat\Lambda}(M)$, we set $A = F \boxtimes \eta$ for $F \in \Sh_{-\widehat\Lambda}(M)$. Thus,
\begin{align*}
(\epsilon \otimes \id) \circ (\id \otimes \eta)(F)
&=  \tilde\pi_{3*} \left(   \sHom\left(\tilde\pi_{12}^* K(T_{-\epsilon}),  F \boxtimes \eta \right)  \right) \\
&= \tilde\pi_{3*} \left( \sHom \left(\tilde\pi_{12}^* K(T_{-\epsilon}), 1_{M^3} \right) \otimes   (F \boxtimes \eta) ) \right) \\
&= \tilde\pi_{3*} \left( \tilde\pi_{12}^* \sHom \left( K(T_{-\epsilon}), 1_{M \times M} \right) \otimes   \tilde\pi_{23}^* \eta \otimes  \tilde\pi_1^* F) ) \right)
\end{align*}
Here we use \cite[Proposition 5.4.14]{KS} again to turn $\sHom$ into a $\otimes$ and then expand $\boxtimes$ by the definition. Our goal now is to organize the pull/push functors associated to the projections into the form of convolutions. This process is a special case of the proof for \cite[Proposition 3.6.4]{KS}.
\begin{align*}
(\epsilon \otimes \id) \circ (\id \otimes \eta)(F)
&= \tilde\pi_{3*} \left( \tilde\pi_{12}^* \sHom \left( K(T_{-\epsilon}), 1_{M \times M} \right) \otimes   \tilde\pi_{23}^* \eta \otimes  \tilde\pi_1^* F) ) \right) \\
&= \pi_{2*} \tilde\pi_{13*} \left( \big( \tilde\pi_{12}^* \sHom \left(K(T_{-\epsilon}), 1_{M \times M} \right) \otimes   \tilde\pi_{23}^*  \eta \big)
\otimes  \tilde\pi_{13}^* \pi_1^* F  \right) \\
&= \pi_{2*} \left( \tilde\pi_{13*} \big( \tilde\pi_{12}^* \sHom \left( K(T_{-\epsilon}), 1_{M \times M} \right) \otimes   \tilde\pi_{23}^* \eta \big)
\otimes  \pi_1^* F  \right) \\
&= \pi_{2*} \left( \left( \eta \circ_{M}  \sHom \left( K(T_{-\epsilon}), 1_{M \times M} \right) \right)
\otimes  \pi_1^* F  \right)  \\
&= \left( \eta \circ_{M}  \sHom ( K(T_{-\epsilon}), 1_{M \times M} ) \right) \circ F.
\end{align*}
By Lemma \ref{lad&ker}, the last expression is 
$$\left( \eta \circ_{M}  \sHom ( K(T_{-\epsilon}), 1_{M \times M} ) \right) \circ F = \big( \eta \circ_{M} \iota_{\widehat\Lambda \times -\widehat\Lambda}^* \left(  \sHom ( K(T_{-\epsilon}), 1_{M \times M} ) \right) \big) \circ F.$$
Thus, the requirement $(\epsilon \otimes \id) \circ (\id \otimes \eta) = \id$ implies that 
$$\eta \circ_{M} \iota_{\widehat\Lambda \times -\widehat\Lambda}^*  \left(  \sHom ( K(T_{-\epsilon}), 1_{M \times M} ) \right)=  \iota_{\widehat\Lambda \times -\widehat\Lambda}^* (1_\Delta)$$
since convolution kernels are determined by their effect on $\Sh_{-\widehat\Lambda}(M)$ by Theorem \ref{thm: fmt}. But, as a functor on $\Sh_{-\widehat\Lambda}(M)$, by Lemma \ref{lad&ker} and Remark \ref{rem:kernelinverse}, we have
$$\iota_{\widehat\Lambda \times -\widehat\Lambda}^* \sHom ( K(T_{-\epsilon}), 1_{M \times M} ) \circ F = \iota_{\widehat\Lambda}^*(\sHom ( K(T_{-\epsilon}), 1_{M \times M} ) \circ F) = \iota_{\widehat\Lambda}^* (K(T_{\epsilon})\circ F) \otimes \omega_M^{-1}$$
so we conclude that $S_{\widehat\Lambda}^*$ has a left inverse.

On the other hand, a similar computation will shows that $(\id \otimes \epsilon ) \circ (\eta \otimes \id) = \id$ on $\Sh_{\widehat\Lambda}(M)$ implies that 
$$G \circ \big( \iota_{\widehat\Lambda \times -\widehat\Lambda}^* (  \sHom ( K(\Phi^{-1}), 1_{M \times M} ) ) \circ_{M} \eta \big) = G$$
for all $G \in \Sh_{\widehat\Lambda}(M)$, and thus $\iota_{\widehat\Lambda \times -\widehat\Lambda}^* (  \sHom ( K(T_{-\epsilon}), 1_{M \times M} ) ) \circ_{M} \eta = \iota_{\widehat\Lambda \times -\widehat\Lambda}^* (1_\Delta)$. Now, the trick is that we can view this equality as in $\Sh_{\widehat\Lambda}(M)$ by convoluting from the left instead. Thus, we conclude that $S_{\widehat\Lambda}^+$ has a right inverse as well. In particular, $S_{\widehat\Lambda}^+$ is invertible with inverse $S_{\widehat\Lambda}^-$.
\end{proof} 

\begin{remark}
In \cite[Section 5.3]{Kuo-Li-spherical}, the authors show that $S_\Lambda^+$ is an equivalence when $\Lambda$ is either swappable or full stop. 
Thus, the Verdier dual $\VD{M}: \Sh_\Lambda^b(M)^{op} \xrightarrow{\sim} \Sh_\Lambda^b(M)$ extends to an equivalence $\Sh_\Lambda^c(M)^{op} \xrightarrow{\sim} \Sh_\Lambda(M)^c$ on all compact objects for those cases.
\end{remark}

\bibliographystyle{amsplain}
\bibliography{ref_KuoLi}

\begin{bibdiv}
\begin{biblist}

\bib{Arinkin-Gaitsgory-Kazhdan-Raskin-Rozenblyum-Varshavsky}{article}{
      author={Arinkin, Dima},
      author={Gaitsgory, Dennis},
      author={Kazhdan, David},
      author={Raskin, Sam},
      author={Rozenblyum, Nick},
      author={Varshavsky, Yakov},
       title={Duality for automorphic sheaves with nilpotent singular support},
        date={2020},
     journal={arXiv preprint arXiv:2012.07665},
}

\bib{AurouxAnti}{article}{
      author={Auroux, Denis},
       title={Mirror symmetry and {$T$}-duality in the complement of an anticanonical divisor},
        date={2007},
     journal={J. G\"{o}kova Geom. Topol. GGT},
      volume={1},
       pages={51\ndash 91},
      review={\MR{2386535}},
}

\bib{TiltingExercise}{article}{
      author={Beilinson, A.},
      author={Bezrukavnikov, R.},
      author={Mirkovi\'{c}, I.},
       title={Tilting exercises},
        date={2004},
        ISSN={1609-3321},
     journal={Mosc. Math. J.},
      volume={4},
      number={3},
       pages={547\ndash 557, 782},
         url={https://doi.org/10.17323/1609-4514-2004-4-3-547-557},
      review={\MR{2119139}},
}

\bib{Ben-Zvi-Francis-Nadler}{article}{
      author={Ben-Zvi, David},
      author={Francis, John},
      author={Nadler, David},
       title={Integral transforms and {D}rinfeld centers in derived algebraic geometry},
        date={2010},
        ISSN={0894-0347},
     journal={J. Amer. Math. Soc.},
      volume={23},
      number={4},
       pages={909\ndash 966},
         url={https://doi.org/10.1090/S0894-0347-10-00669-7},
      review={\MR{2669705}},
}

\bib{Brav-Dyckerhoff2}{article}{
      author={Brav, Christopher},
      author={Dyckerhoff, Tobias},
       title={Relative {C}alabi-{Y}au structures {II}: shifted {L}agrangians in the moduli of objects},
        date={2021},
        ISSN={1022-1824},
     journal={Selecta Math. (N.S.)},
      volume={27},
      number={4},
       pages={Paper No. 63, 45},
         url={https://doi.org/10.1007/s00029-021-00642-5},
      review={\MR{4281260}},
}

\bib{Czapla}{article}{
      author={Czapla, Ma\l~gorzata},
       title={Definable triangulations with regularity conditions},
        date={2012},
        ISSN={1465-3060},
     journal={Geom. Topol.},
      volume={16},
      number={4},
       pages={2067\ndash 2095},
         url={https://doi.org/10.2140/gt.2012.16.2067},
      review={\MR{3033513}},
}

\bib{DrinfeldGaitsgory}{article}{
      author={Drinfeld, V.},
      author={Gaitsgory, D.},
       title={Compact generation of the category of {D}-modules on the stack of {$G$}-bundles on a curve},
        date={2015},
        ISSN={2168-0930},
     journal={Camb. J. Math.},
      volume={3},
      number={1-2},
       pages={19\ndash 125},
         url={https://doi.org/10.4310/CJM.2015.v3.n1.a2},
      review={\MR{3356356}},
}

\bib{EESduality}{article}{
      author={Ekholm, Tobias},
      author={Etnyre, John~B.},
      author={Sabloff, Joshua~M.},
       title={A duality exact sequence for {L}egendrian contact homology},
        date={2009},
        ISSN={0012-7094},
     journal={Duke Math. J.},
      volume={150},
      number={1},
       pages={1\ndash 75},
         url={https://doi.org/10.1215/00127094-2009-046},
      review={\MR{2560107}},
}

\bib{FLTZCCC}{article}{
      author={Fang, Bohan},
      author={Liu, Chiu-Chu~Melissa},
      author={Treumann, David},
      author={Zaslow, Eric},
       title={A categorification of {M}orelli's theorem},
        date={2011},
        ISSN={0020-9910},
     journal={Invent. Math.},
      volume={186},
      number={1},
       pages={79\ndash 114},
         url={https://doi.org/10.1007/s00222-011-0315-x},
      review={\MR{2836052}},
}

\bib{Gaitsgory-IndCoh}{article}{
      author={Gaitsgory, Dennis},
       title={ind-coherent sheaves},
        date={2013},
        ISSN={1609-3321},
     journal={Mosc. Math. J.},
      volume={13},
      number={3},
       pages={399\ndash 528, 553},
}

\bib{Gaitsgory-Rozenblyum}{book}{
      author={Gaitsgory, Dennis},
      author={Rozenblyum, Nick},
       title={A study in derived algebraic geometry. {V}ol. {I}. {C}orrespondences and duality},
      series={Mathematical Surveys and Monographs},
   publisher={American Mathematical Society, Providence, RI},
        date={2017},
      volume={221},
        ISBN={978-1-4704-3569-1},
      review={\MR{3701352}},
}

\bib{Ganatra-Pardon-Shende3}{article}{
      author={Ganatra, Sheel},
      author={Pardon, John},
      author={Shende, Vivek},
       title={Microlocal {M}orse theory of wrapped {F}ukaya categories},
        date={2024},
        ISSN={0003-486X},
     journal={Ann. of Math. (2)},
      volume={199},
      number={3},
       pages={943\ndash 1042},
         url={https://doi.org/10.4007/annals.2024.199.3.1},
      review={\MR{4740209}},
}

\bib{Ganatra-Pardon-Shende2}{article}{
      author={Ganatra, Sheel},
      author={Pardon, John},
      author={Shende, Vivek},
       title={Sectorial descent for wrapped {F}ukaya categories},
        date={2024},
     journal={Journal of the American Mathematical Society},
      volume={37},
      number={2},
       pages={499\ndash 635},
}

\bib{Gao1}{article}{
      author={{Gao}, Yuan},
       title={{F}unctors of wrapped {F}ukaya categories from {L}agrangian correspondences},
        date={2018},
     journal={\href{https://arxiv.org/abs/1712.00225}{arXiv:1712.00225v2}},
}

\bib{Gui}{article}{
      author={Guillermou, St{\'e}phane},
       title={Quantization of conic {L}agrangian submanifolds of cotangent bundles},
        date={2012},
     journal={\href{https://arxiv.org/abs/1212.5818}{arXiv:1212.5818}},
}

\bib{Guillermou-Kashiwara-Schapira}{article}{
      author={Guillermou, St\'{e}phane},
      author={Kashiwara, Masaki},
      author={Schapira, Pierre},
       title={Sheaf quantization of {H}amiltonian isotopies and applications to nondisplaceability problems},
        date={2012},
        ISSN={0012-7094},
     journal={Duke Math. J.},
      volume={161},
      number={2},
       pages={201\ndash 245},
         url={https://doi.org/10.1215/00127094-1507367},
      review={\MR{2876930}},
}

\bib{Guillermou-Viterbo}{article}{
      author={Guillermou, St\'{e}phane},
      author={Viterbo, Claude},
       title={The singular support of sheaves is {$\gamma $}-coisotropic},
        date={2024},
        ISSN={1016-443X},
     journal={Geom. Funct. Anal.},
      volume={34},
      number={4},
       pages={1052\ndash 1113},
         url={https://doi.org/10.1007/s00039-024-00682-x},
      review={\MR{4768582}},
}

\bib{Hoyois-Scherotzke-Sibilla}{article}{
      author={Hoyois, Marc},
      author={Scherotzke, Sarah},
      author={Sibilla, Nicol\`o},
       title={Higher traces, noncommutative motives, and the categorified {C}hern character},
        date={2017},
        ISSN={0001-8708},
     journal={Adv. Math.},
      volume={309},
       pages={97\ndash 154},
         url={https://doi.org/10.1016/j.aim.2017.01.008},
      review={\MR{3607274}},
}

\bib{KS}{book}{
      author={Kashiwara, Masaki},
      author={Schapira, Pierre},
       title={Sheaves on manifolds},
      series={Grundlehren der mathematischen Wissenschaften [Fundamental Principles of Mathematical Sciences]},
   publisher={Springer-Verlag, Berlin},
        date={1994},
      volume={292},
        ISBN={3-540-51861-4},
        note={With a chapter in French by Christian Houzel, Corrected reprint of the 1990 original},
      review={\MR{1299726}},
}

\bib{KonHMS}{inproceedings}{
      author={Kontsevich, Maxim},
       title={Homological algebra of mirror symmetry},
organization={Springer},
        date={1995},
   booktitle={Proceedings of the international congress of mathematicians},
       pages={120\ndash 139},
}

\bib{Kuo-wrapped-sheaves}{article}{
      author={Kuo, Christopher},
       title={Wrapped sheaves},
        date={2023},
        ISSN={0001-8708},
     journal={Adv. Math.},
      volume={415},
       pages={Paper No. 108882, 71},
         url={https://doi.org/10.1016/j.aim.2023.108882},
      review={\MR{4543073}},
}

\bib{Kuo-Li-Calabi-Yau}{article}{
      author={Kuo, Christopher},
      author={Li, Wenyuan},
       title={Relative {C}alabi--{Y}au structure on microlocalization},
        date={2024},
     journal={\href{https://arxiv.org/abs/2408.04085}{arXiv:2408.04085}},
}

\bib{Kuo-Li-spherical}{article}{
      author={Kuo, Christopher},
      author={Li, Wenyuan},
       title={{S}pherical adjunction and {S}erre functor from microlocalization},
        date={2024},
     journal={\href{https://arxiv.org/abs/2210.06643}{arXiv:2210.06643}},
}

\bib{Kuo-Shende-Zhang-Hochschild-Tamarkin}{article}{
      author={Kuo, Christopher},
      author={Shende, Vivek},
      author={Zhang, Bingyu},
       title={On the {H}ochschild cohomology of {T}amarkin categories},
        date={2023},
     journal={\href{https://arxiv.org/abs/2312.11447}{arXiv:2312.11447}},
}

\bib{KuwaCCC}{article}{
      author={Kuwagaki, Tatsuki},
       title={The nonequivariant coherent-constructible correspondence for toric stacks},
        date={2020},
        ISSN={0012-7094},
     journal={Duke Math. J.},
      volume={169},
      number={11},
       pages={2125\ndash 2197},
         url={https://doi.org/10.1215/00127094-2020-0011},
      review={\MR{4132582}},
}

\bib{Lurie-HA}{article}{
      author={{Lurie}, Jacob},
       title={Higher algebra},
        date={2017},
     journal={Preprint, available at \href{https://www.math.ias.edu/~lurie/}{https://www.math.ias.edu/~lurie/}},
}

\bib{D.Massey}{article}{
      author={Massey, David~B.},
       title={The {S}ebastiani-{T}hom isomorphism in the derived category},
        date={2001},
        ISSN={0010-437X},
     journal={Compositio Math.},
      volume={125},
      number={3},
       pages={353\ndash 362},
         url={https://doi.org/10.1023/A:1002608716514},
      review={\MR{1818986}},
}

\bib{Mukai}{article}{
      author={Mukai, Shigeru},
       title={Duality between {$D (X)$} and with its application to {P}icard sheaves},
        date={1981},
     journal={Nagoya Mathematical Journal},
      volume={81},
       pages={153\ndash 175},
}

\bib{Nad}{article}{
      author={Nadler, David},
       title={Microlocal branes are constructible sheaves},
        date={2009},
        ISSN={1022-1824},
     journal={Selecta Math. (N.S.)},
      volume={15},
      number={4},
       pages={563\ndash 619},
         url={https://doi.org/10.1007/s00029-009-0008-0},
      review={\MR{2565051}},
}

\bib{NadWrapped}{article}{
      author={Nadler, David},
       title={Wrapped microlocal sheaves on pairs of pants},
        date={2016},
     journal={\href{https://arxiv.org/abs/1604.00114}{arXiv:1604.00114}},
}

\bib{Nadler-Shende}{article}{
      author={{Nadler}, David},
      author={{Shende}, Vivek},
       title={Sheaf quantization in {W}einstein symplectic manifolds},
        date={2021},
     journal={\href{https://arxiv.org/abs/2007.10154}{arXiv:2007.10154v2}},
}

\bib{Nadler-Zaslow}{article}{
      author={{Nadler}, David},
      author={{Zaslow}, Eric},
       title={Constructible sheaves and the {F}ukaya category},
        date={2009},
        ISSN={0894-0347},
     journal={J. Amer. Math. Soc.},
      volume={22},
      number={1},
       pages={233\ndash 286},
         url={https://doi.org/10.1090/S0894-0347-08-00612-7},
      review={\MR{2449059}},
}

\bib{Orlov}{article}{
      author={Orlov, Dmitri~Olegovich},
       title={Derived categories of coherent sheaves and equivalences between them},
        date={2003},
     journal={Russian Mathematical Surveys},
      volume={58},
      number={3},
       pages={511},
}

\bib{Preygel}{thesis}{
      author={Preygel, Anatoly},
       title={Thom-{S}ebastiani {\&} duality for matrix factorizations},
        type={Ph.D. Thesis},
        date={2011},
}

\bib{Robalo-Schapira}{article}{
      author={Robalo, Marco},
      author={Schapira, Pierre},
       title={A lemma for microlocal sheaf theory in the {$\infty$}-categorical setting},
        date={2018},
        ISSN={0034-5318},
     journal={Publ. Res. Inst. Math. Sci.},
      volume={54},
      number={2},
       pages={379\ndash 391},
}

\bib{Sabduality}{article}{
      author={Sabloff, Joshua~M.},
       title={Duality for {L}egendrian contact homology},
        date={2006},
        ISSN={1465-3060},
     journal={Geom. Topol.},
      volume={10},
       pages={2351\ndash 2381},
         url={https://doi.org/10.2140/gt.2006.10.2351},
      review={\MR{2284060}},
}

\bib{Schurmann}{book}{
      author={Sch\"{u}rmann, J\"{o}rg},
       title={Topology of singular spaces and constructible sheaves},
      series={Instytut Matematyczny Polskiej Akademii Nauk. Monografie Matematyczne (New Series) [Mathematics Institute of the Polish Academy of Sciences. Mathematical Monographs (New Series)]},
   publisher={Birkh\"{a}user Verlag, Basel},
        date={2003},
      volume={63},
        ISBN={3-7643-2189-X},
         url={https://doi.org/10.1007/978-3-0348-8061-9},
      review={\MR{2031639}},
}

\bib{Shende-h-principle}{article}{
      author={Shende, Vivek},
       title={Microlocal category for {W}einstein manifolds via the h-principle},
        date={2021},
        ISSN={0034-5318},
     journal={Publ. Res. Inst. Math. Sci.},
      volume={57},
      number={3-4},
       pages={1041\ndash 1048},
         url={https://doi.org/10.4171/prims/57-3-10},
      review={\MR{4322006}},
}

\bib{Toen-Morita-theory}{article}{
      author={To{\"e}n, Bertrand},
       title={The homotopy theory of dg-categories and derived {M}orita theory},
        date={2007},
     journal={Inventiones mathematicae},
      volume={167},
      number={3},
       pages={615\ndash 667},
}

\bib{Volpe-six-operations}{article}{
      author={Volpe, Marco},
       title={The six operations in topology},
        date={2023},
     journal={\href{https://arxiv.org/abs/2110.10212}{ arXiv:2110.10212v2}},
}

\end{biblist}
\end{bibdiv}

\end{document}